\documentclass[11pt]{amsart}
\usepackage{amssymb,amsmath,amscd,amsfonts,a4,graphicx,
latexsym,epsfig,color,soul}
\usepackage[all]{xy}

\usepackage{comment}
\usepackage{enumitem}
\usepackage{tikz}

\everymath{\displaystyle}

\newtheorem{teo}{Theorem}[section]
\newtheorem{lem}[teo]{Lemma}
\newtheorem{cor}[teo]{Corollary}

\newtheorem{prop}[teo]{Proposition}
\newtheorem{defi}[teo]{Definition}

\newtheorem{remark}[teo]{Remark}

\newcommand{\fleche}[4]{                     
            \begin{array}{lcll} #1 & \longrightarrow & #2 \\   %
                         #3 &\longmapsto & #4          %
            \end{array}}

\newcommand{\fonc}[5]{                     
            \begin{array}{lcll}#1 :& #2 & \longrightarrow & #3 \\   %
                         &#4 &\longmapsto & #5          %
            \end{array}}

\newcommand{\cvd}{\hfill$\Box$}

\newcommand{\mr}{\mathbb{R}}
\newcommand{\mc}{\mathbb{C}}
\newcommand{\mz}{\mathbb{Z}}
\newcommand{\mh}{\mathbb{H}}

\newcommand{\mn}{\mathbb{N}}

\newcommand{\R}{\mathbb{R}}
\newcommand{\C}{\mathbb{C}}
\newcommand{\Z}{\mathbb{Z}}

\newcommand{\bw}{{\bf w}}
\newcommand{\bu}{{\bf u}}
\newcommand{\bx}{{\bf x}}
\newcommand{\bv}{{\bf v}}
\newcommand{\bl}{{\bf l}}

\newcommand{\Log}{{\rm Log}}

\newcommand{\Hh}{{\mathcal H}}

\newcommand{\rmb}{{\rm \bf b}}
\newcommand{\ra}{\rightarrow}

\def\cvd{\hfill$\Box$}
\def\fin{\hfill$\Box$}

\title[Volume Conjecture and quantum hyperbolic invariants]
{Volume Conjecture and quantum hyperbolic invariants: the figure eight knot complement}

\author{St\'ephane Baseilhac$^1$, Fathi Ben Aribi$^{2}$}

\begin{document}

\maketitle

\centerline{${}^1$ IMAG, Univ Montpellier, CNRS, France}
\centerline{${}^2$ Sorbonne Université, Université Paris Cité, CNRS, IMJ-PRG, F-75005 Paris, France}
\medskip \centerline{${}^1$ stephane.baseilhac@umontpellier.fr\ ,\ ${}^2$ fathi.ben-aribi@imj-prg.fr}

\begin{abstract} We compute the real part of the semi-classical limit of the sequence of quantum hyperbolic invariants (QHI) of the figure-eight knot complement $M$. We show that  it is rigid, in the sense that it does not depend on the choice of holonomy representation of $M$, and it is either $0$ or equal to the hyperbolic volume of $M$ divided by $2\pi$, depending on a parity condition satisfied by logarithms of the holonomy eigenvalues on the canonical longitude, where the logarithms are parameters of the QHI of $M$. Along the way we also survey some relevant general features of the QHI.
\end{abstract}
\bigskip



\tableofcontents

\section{Introduction}

\subsection{Main result}\label{QHIintro1} In this paper we compute the real part of the semi-classical limit of the sequence of {\it quantum hyperbolic invariants} (or \textit{QHI}) associated to the (oriented) complement $M:=S^3 \setminus K$ of the figure-eight knot $K:=4_1$. We notably obtain exponential growth of the invariants, with rate equal to the hyperbolic volume ${\rm Vol}(M)/2\pi$, or subexponential growth, depending on a choice of parameters of the QHI.  \smallskip

Let us state our main result. Denote by $\lambda_K$ the canonical longitude of $K$, and by $X_{hyp}(M)$ the  irreducible component of the $PSL(2,\mc)$-character variety of $M$ containing the character of the discrete faithful representations. The terms $h_\rho$ and $k_c$ in the statement below will be described later (see Section \ref{sec:QHIcusped}), but should be thought of as additional geometric and topological structures on $M$. In particular $h_\rho$ is a function on the set of curves on $\partial \bar{M}$, the boundary torus at infinity, with values in $(2 i \pi)\mz$. The QHI, denoted $\Hh_{N}$, depend on an odd integer $N\geq 3$ and take values in $\mc$ modulo multiplication by $2N$-th roots of unity.

\begin{teo}\label{asyinvteo} For every character $\rho\in X_{hyp}(M)$ and charge weight $k_c$, the following holds:
	\smallskip
	
	\noindent (a) If $h_\rho(\lambda_{K}) \in 2\pi i(2\mz+1)$, then 
	$$\lim_{N\ra +\infty} \frac{2\pi}{N} \Log \vert \Hh_{N}(M,\rho,h_\rho,k_c)\vert = {\rm Vol}(M).$$ 
	(b) If $h_\rho(\lambda_{K}) \in 4\pi i \mz$, then
	$$\lim_{N\ra +\infty} \frac{2\pi}{N} \Log \vert \Hh_{N}(M,\rho,h_\rho,k_c)\vert = 0.$$
	
\end{teo} 
This result belongs to the realm of the \textit{Volume Conjectures}, a set of conjectures first considered by Kashaev-Murakami-Murakami \cite{Kvol,MM}, which predict that the semi-classical limit of certain sequences of quantum invariants of finite volume complete hyperbolic manifolds is equal to their volume. Theorem \ref{asyinvteo} shows new phenomena, due to the fact that the QHI depend on geometric data: they satisfy some kind of semi-classical rigidity, since the choice of character $\rho$ has no effect on their semi-classical limit, and their growth, exponential or not, is determined by a geometrico-topological datum, the parity of \scalebox{0.8}{$\dfrac{1}{2\pi i}$}$h_\rho(\lambda_{K})$. 

The proof of Theorem \ref{asyinvteo} is based on an (easily obtained) integral representation of the QHI, a description of the so-called {\it quantum gluing varieties}, which form parameter spaces of the QHI and come from previous works of the first author with R. Benedetti, and elementary uses of the saddle point method and Perron's method to estimate the integrals when $N\to +\infty$. The crux of the proof lies in finding deformations of the integration contours to ``good positions'' in order to apply these methods. A technical difficulty is that the endpoints of contours converge to poles of the integrand as $N\to +\infty$. The analysis of case (b) is far more demanding than that of case (a).

Many tools of this proof should generalize to study the QHI asymptotics for arbitrary cusped manifolds, in situations more general than those implemented by the two kinds of values of $h_\rho(\lambda_{K})$ in Theorem \ref{asyinvteo}. For this reason, since the QHI theory is not well known, in some parts of the paper we allow ourselves to present relevant general aspects of this theory, in a hopefully pedestrian way, and introduce notions that should appear to be useful mainly for generalizations of Theorem \ref{asyinvteo} (we do this in Sections \ref{QGVsec}, \ref{LOGW}, \ref{ASYWEIGHTS}, \ref{sec:QHIcusped} and \ref{sec:class_int}, see the plan of the paper below). In particular, we discuss the space of logarithmic limits of log-parameters, on which the QHI asymptotics can be studied, and where the cases (a)-(b) of Theorem \ref{asyinvteo} correspond to very peculiar situations.

We remark that, beyond the semi-classical limit shown in Theorem \ref{asyinvteo}, sub exponential terms can be obtained (see Remark \ref{fulldev}). Also, the alternative between the cases (a) and (b) of Theorem \ref{asyinvteo} seemingly reflects that the QHI could detect lifts to $SL(2,\mc)$ of $PSL(2,\mc)$-characters; we plan to develop this in a future work.

In a forthcoming version of this work we will consider also similar asymptotics of QHI for hyperbolic Dehn fillings of $M$.

In the rest of this introduction we present the QHI and the structure of the proof of Theorem \ref{asyinvteo}.
\smallskip

\noindent {\bf The QHI and their relations to other invariants}. The QHI refer to two main types of invariants: 
\begin{itemize}
	\item either invariants of a pair $(N,L)$, where $N$ is a compact oriented $3$-manifold (with or without boundary) and $L \subset N$ is a properly embedded tangle;
	\item or invariants of a {\it cusped} manifold $M$, i.e., $M$ is an oriented non-compact $3$-manifold admitting a complete hyperbolic structure with finite volume. 
\end{itemize}
In both cases the manifold is endowed with additional geometric structures, including a $PSL(2,\mc)$-character of the fundamental group. We will refer to the first type as \textit{QHI of pairs}. These QHI generalize the Kashaev invariants defined in \cite{K0,K1} (see \cite{Top, LINK}). In particular  the Kashaev invariants $\langle L\rangle_N$ of a link $L$ in $S^3$, which are specializations of the colored Jones polynomials (\cite{MM}), coincide with the QHI of the pair $(S^3, L)$ (\cite{LINK}). The QHI of cusped manifolds have been defined in \cite{GT}, and further refined in \cite{AGT,NA}. By means of a hyperbolic Dehn filling theorem proved in \cite{AGT0}, we can relate some QHI of $M$ with the QHI of pairs $(N,L)$ where $N \setminus L = M$. 

 Both types of QHI are defined by state sums over triangulations very similar to Neumann's simplicial formulas for the Chern-Simons invariant of $PSL(2,\mc)$-characters (\cite{N}); in these state sums the classical dilogarithm functions are replaced with tensors called ``matrix" dilogarithms, satisfying a five term identity similar to the Roger's one for dilogarithms. Because of this strong structural relationship, it is natural that expect that a ``Volume Conjecture for the QHI'' should occur, relating their semi-classical limit with the Chern-Simons invariant of $PSL(2,\mc)$-characters. This was proposed in \cite{KyotoT}; Theorem \ref{asyinvteo} (a) proves it for the figure-eight knot complement. Numerical evidences of the parity phenomenon (a) vs.\!\! (b) in Theorem \ref{asyinvteo} were provided in \cite[Section 9]{AGT} also for the complement of the knot $5_2$.\\
 The QHI of mapping cylinders are easier to deal with, essentially because the mapping cylinder axis trivializes asymmetrical features of the QHI construction. It was shown in \cite{GD} that the {\it reduced} QHI (see Section \ref{sec:QHIcusped}) of mapping tori of pseudo-Anosov diffeomorphisms of a punctured surface $S$ are the trace of intertwiners between {\it local} representations of the quantum Teichm\"uller space of $S$ (in the sense of \cite{B-B-L}). This has been greatly clarified recently by Ishibashi (\cite{Ishibashi}), who provided, in particular, an alternative construction of these reduced QHI based on quantum cluster algebras associated to ``dotted'' ideal triangulations of $S$, in the sense of Kashaev (\cite{KqTeich}). Also, Garoufalidis-Yu (\cite{GY}) obtained a decomposition of 
these reduced QHI (in the case where the holonomy is boundary-parabolic and the log weight $l\kappa=0$) in terms of the Bonahon-Liu-Wong-Yang intertwiners of irreducible representations of the Kauffman bracket skein algebra of $S$ (\cite{B-L}, \cite[Theorem 16]{BWY1}) and the $\mathfrak{gl}_1$-invariants of mapping classes studied by Gocho (\cite{Gocho}) and Murakami-Ohtsuki-Okada (\cite{MOO}).\\ In particular, the results of this paper can be compared with those of Bonahon-Wong-Yang in \cite{BWY2}, which uses a different approach. In another direction, the Andersen--Kashaev \cite{AK} partition functions, which are derived from infinite dimensional quantum Teichm\"uller theory, make a kind of infinite dimensional version of the QHI and share many structural features with them. It is thus an interesting problem to clarify the relations between the results of this paper and those of Guéritaud, Piguet-Nakazawa and the second author in [11], and of Guilloux, Wong and the second author in [BAW, BAGW] concerning FAMED triangulations. 

The relations between the QHI and the quantum invariants of Reshetikhin--Turaev or Turaev-Viro type are very intringuing. As mentioned above, the QHI of $(S^3, L)$ coincide with the Kashaev invariant $\langle 4_1\rangle_N$. As an example of relations between these invariants and the QHI of $S^3\setminus L$ for a hyperbolic link $L$, we exhibit in Appendix \ref{sec:KvsQHI} a discrete Laplace-transform type formula connecting the QHI of $M=S^3\setminus 4_1$ and $\langle 4_1\rangle_N$. This formula seems to generalize to other knots. It may remind the reader of a conjecture of Andersen-Kashaev \cite[Conjecture 1]{AK}. Clarifying such relationships could be useful to understand the Volume Conjecture for the colored Jones polynomials, since the QHI are intrinsically geometric, being rational functions on varieties covering the geometric component of the variety of $PSL(2,\mc)$-characters (see the formula \eqref{splitHNteo1.1} and the comments thereafter).\\ 
Even more intringuing is the problem of extending Zagier's quantum modularity conjecture (\cite{ZagierQM,GarZag}) to the QHI.\\
In this direction, we note that a ``parity'' phenomenon similar to the alternative (a) vs.\!\! (b) in Theorem \ref{asyinvteo} happens with the Turaev-Viro invariants $TV(M,q)$ of cusped manifolds $M$ \cite{CY}, and with the colored Jones polynomials $J_{L,N}(q)$ of hyperbolic links $L$ at roots of unity $q$ \cite{DKY}. Namely, it is conjectured that the sequence of invariants $\textstyle TV(M,\exp(\frac{2i\pi}{N}))$, with $N$ odd, grows exponentially with growth rate Vol$(M)/2\pi$ as $N\ra +\infty$, whereas $\textstyle TV(M,\exp(\frac{i\pi}{N}))$ is expected to have polynomial growth (according to Witten's asymptotic expansion conjecture). Similarly, it is expected that $\textstyle J_{L,N}(\exp(\frac{2i\pi}{N+1/2}))$ has semi-classical limit Vol$(S^3\setminus L)/2\pi$, whereas it is known that for any integer $l$, $\textstyle J_{L,N}(\exp(\frac{2i\pi}{N+l}))$ grows polynomially with $N$.
\smallskip

\noindent {\bf Plan of the paper.} In Section \ref{QHIintro2} we give a state sum formulation of Theorem \ref{asyinvteo} that may be more accessible to readers acquainted with the volume conjectures. Section \ref{sec:sketchproof} presents the structure of the proof of Theorem \ref{asyinvteo}.\\
The geometric setup of the QHI is described in Section \ref{NOTA}. Sections \ref{QGVsec} and \ref{LOGW} form a self-contained toolbox, providing material and, we hope, enough explanations in order to understand the QHI far beyond the case of  $M=S^3\setminus 4_1$, which is developed in Section \ref{4_1}. So Sections \ref{QGVsec} and \ref{LOGW} could be skipped at a first reading. Also, Section \ref{ASYWEIGHTS} may be skipped at a first reading; it introduces a logarithmic limit space for the tower of quantum gluing varieties, which is useful to understand geometrically all asymptotical manipulations to follow.\\ The proof of Theorem \ref{asyinvteo} begins in Section \ref{sec:QHI}. Sections \ref{FUNCNOT} and \ref{sec:asySNfev26} provide basic but fundamental asymptotic results on the functions used in the QHI state sum formulas. Section \ref{sec:QHIcusped} presents general features of the QHI of cusped manifolds; in particular, the notation $\Hh_{N}(M,\rho,h_\rho,k_c)$, which was not used previously in the literature on QHI, is explained. The combination of Sections \ref{sec:statesum}, \ref{INTREPEST} and \ref{loglimmars25} gives an integral representation of the QHI of $M:=S^3\setminus 4_1$.\\ The core of the analysis is developed in Section \ref{sec:SPM} and \ref{sec:rect:contour}. In Section \ref{sec:SPM} we consider as a warm-up situation the asymptotics of simplified, so called ``classical'' integrals. These integrals are defined in Section \ref{sec:class_int}, where it is also shown how the critical sets of the integrands recover a neighborhood of the geometric solution in the gluing variety. The rest of Section \ref{sec:SPM} is devoted to the asymptotics of the classical integrals for some specific choices of logarithmic limits of log-parameters, related to Theorem \ref{asyinvteo}.  In Section \ref{sec:rect:contour} we adapt this analysis to the genuine, ``quantum'' integrals coming from the QHI of $M$.  
\medskip

\noindent {\bf Notations.} In all the paper $N$ is an odd integer greater than $1$. We put $$\zeta := e^{\frac{2\pi i}{N}}\ ,\ \mc^*:= \mc\setminus \{0\}\ ,\ \mc_* := \mc\setminus \{0,1\}.$$
We denote by $\Log$ the complex (Neperian) logarithm with imaginary part in $]-\pi,\pi]$, and by $\arg(z)\in ]-\pi,\pi]$ the principal determination of the argument of $z\in \mc$. It is used in any evaluation of fractional or real powers of non-zero complex numbers; for instance $x^{1/N} := \exp(\Log(x)/N)$. Except for $\mc^*$ and $\mc_*$ just defined, we reserve the notation ${}^*$ to the complex conjugation. The real and imaginary parts of a complex number $z$ are denoted $\mathfrak{R}(z)$ and $\mathfrak{I}(z)$. Finally, we denote by $=_{\mu_N}$ the equality up to multiplication by $N$-th root of $1$.
\medskip

\noindent {\bf Acknowledgments.} The first author had uncountable discussions with Riccardo Benedetti on the topic of this work, which therefore owes much to his ideas. The second author was supported by the FNRS in his ``Research Fellow'' position at UCLouvain under Grant no. 1B03320F, by the ANR project ``SyTriQ'' and by the grant ``Tremplin nouveaux entrants'' at Sorbonne-Université. We thank Elisha Falbel for suggesting the compacity argument in Section \ref{sub:b:+:perron}.

\subsection{Triangulation-version of Theorem \ref{asyinvteo}}\label{QHIintro2} Again denote $M:=S^3 \setminus 4_1$, the figure-eight knot complement. Let $(T,b)$ be its usual geometric triangulation with two tetrahedra $U$ and $V$, and with a \textit{branching} $b$ (i.e. an orientation of edges without cycles) as shown in Figure \ref{Teight}.

\begin{figure}[ht]
	\begin{center}
		\includegraphics[width=8cm]{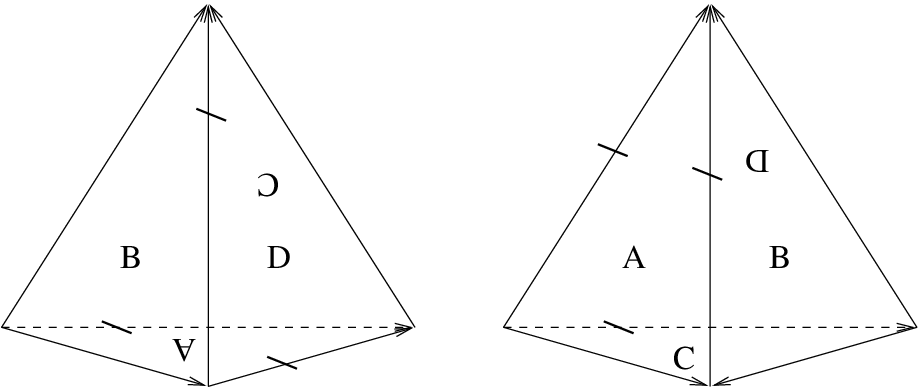}
		\caption{\label{Teight} The face and edge identifications of a geometric branched ideal triangulation of the figure eight knot complement.}
	\end{center}
\end{figure}

For the following notations, see Section \ref{QGVsec} for details. Denote by $u=(u_0,u_1,u_2)$ and $v=(v_0,v_1,v_2)$ triples of complex shape parameters of ideal hyperbolic tetrahedra associated to the edges of $U$ and $V$, respectively (so $u_i\in \mc_*$, $u_{i+1} = (1-u_i)^{-1}$, and similarly for the $v_i$'s, with indices mod $3$). Assume that they satisfy Thurston's gluing equations along the edges of $T$. Denote by $\rho\colon \pi_1(M)\rightarrow PSL(2,\mc)$ the holonomy defined by $(u,v)$. For each odd integer $N \geqslant 3$, we take $N$-th roots $\bu_{j,N}$ and $\bv_{j,N}$ of $u_{j}$ and $v_{j}$, respectively, thus getting new triples $\bu_N:=(\bu_{0,N},\bu_{1,N},\bu_{2,N})$ and $\bv_N=(\bv_{0,N},\bv_{1,N},\bv_{2,N})$. These $N$-th roots satisfy certain compatibility conditions. 

The invariant $\Hh_{N}(M,\rho,h_\rho,k_c)$ in Theorem \ref{asyinvteo} can be computed as the evaluation at some point $(\bu_N,\bv_N)$ of a rational function $\Hh_{N,(T,b,c)}$ on $\mc^6$. This function results from a {\it state sum} associated to $(T,b)$, {\it i.e.}, the total contraction of a tensor network, whose tensors correspond to the tetrahedra $U$, $V$ (as is common in quantum topology) and are made of quantum dilogarithms. The function $\Hh_{N,(T,b,c)}$ depends on a so-called charge $c$, which is a tuple of integers associated to the pairs of opposite edges of the tetrahedra $U$ and $V$, again satisfying certain compatibility conditions. All these ``certain compatibility conditions'' are fixed by the choices of $h_\rho$ and $k_c$ in Theorem \ref{asyinvteo}. We can thus write 
$$ \Hh_{N}(M,\rho,h_\rho,k_c) = \Hh_{N,(T,b,c)}(\bu_N,\bv_N).$$
For details on $\Hh_{N,(T,b,c)}$ the most computation-oriented readers may want to look at the formulas \eqref{ssumrednonred} (where the charge $c$ has been fixed to some specific value), \eqref{functomegadef} and \eqref{form1}. 
\smallskip

We will see that Theorem \ref{asyinvteo} can be reduced to the following result, which will be our main objective:

\begin{teo}[Triangulation-version of Theorem \ref{asyinvteo}]\label{thm:main_triang}
	Let $(\bu_N)_{N}$ and $(\bv_N)_{N}$ be sequences as above. 	Then, for every charge $c$ the state sum $\Hh_{N,(T,b,c)}$ satisfies:
	\smallskip
	
	\noindent (a) If $\lim_{N\ra +\infty}\bu_{1,N}=-1$ and $\lim_{N\ra +\infty}\bv_{1,N}=-1$, then
	$$\lim_{N\ra +\infty} \frac{2\pi}{N} \Log \vert \Hh_{N,(T,b,c)}(\bu_N,\bv_N)\vert = {\rm Vol}(M).$$ 
	(b)  If $\lim_{N\ra +\infty}\bu_{1,N}=1$ and $\lim_{N\ra +\infty}\bv_{1,N}=1$, then
	$$\lim_{N\ra +\infty} \frac{2\pi}{N} \Log \vert \Hh_{N,(T,b,c)}(\bu_N,\bv_N)\vert = 0.$$
	Moreover, (a) and (b) are the only two possibilities.
\end{teo}

\subsection{Structure of the proof of Theorem \ref{asyinvteo}}\label{sec:sketchproof} Let us denote by $PM(N)$ (for \textit{positive monomial}) the set of complex-valued functions of $N$ of the form $$\exp\left (O_{N\to \infty}(1) + \log(N) O_{N\to \infty}(1)\right ).$$
In other words, a function $f(N)$ is in $PM(N)$ if there exist constants $C, C'>0$ and $\alpha,\beta \in \R$ such that
$$ C N^\alpha \leqslant \vert f(N)\vert \leqslant C' N^\beta$$
for any $N$ large enough. 
Consequently, it satisfies $\lim_{N\ra +\infty} \textstyle \frac{2\pi}{N}\Log \left \vert f(N) \right \vert = 0$.

The following basic lemma will be used several times in this paper.
\begin{lem}\label{lem:PM:SE}
Let $f(N)$ and $g(N)$ be functions of $N$ such that $\lim_{N\ra +\infty} \textstyle \frac{2\pi}{N}\Log \left \vert f(N) \right \vert = a > 0$, and $g(N)$ is in $PM(N)$. Then we have
$$\lim_{N\ra +\infty} \dfrac{2\pi}{N}\Log \left \vert g(N)\cdot f(N) \right \vert = a \quad ,\quad  \lim_{N\ra +\infty}  \frac{2\pi}{N}\Log \left \vert f(N) + g(N) \right \vert = a.$$
\end{lem}

\begin{proof}
The first property is immediate. For the second one, note that since $a>0$ we have $\textstyle \frac{1}{2}|f(N)|>|g(N)|$ for any $N$ large enough, and thus $\textstyle \left \vert \vert f(N) \vert - \frac{\vert f(N) \vert}{2} \right \vert \leqslant \left \vert f(N) + g(N) \right \vert \leqslant \left \vert f(N)\right \vert + \vert f(N) \vert  $ by triangle inequalities; the limit follows. 
\end{proof}
\medskip

The proof of Theorem \ref{asyinvteo} takes the following steps I -- X (for the sake of brevity we do not define all the terms that appear in the formulas):
\medskip

\begin{enumerate}[label=\Roman*)]
\item (On the QHI of cusped manifolds, Section \ref{sec:QHIcusped}) This part recalls some features of the invariant $\Hh_{N}(M,\rho,h_\rho,k_c)$, using the structures introduced in Section \ref{NOTA}. In particular, we observe that:
\smallskip

\begin{itemize}
\item (Formula \eqref{splitHNteo1.1}) $\Hh_{N}(M,\rho,h_\rho,k_c)$ can be split into the product of two invariants, $\alpha_N(M,\rho,k_c;\mathfrak{s})$ (the \textit{symmetry defect}) and $\Hh_N^{red}(M,\rho,\kappa;\mathfrak{s})$ (the \textit{reduced QHI}).
\item (Formula \eqref{ssumrednonred}) Refomulating this factorization on the triangulation $(T,b)$ of $M:=S^3\setminus K$ shown in Figure \ref{Teight}, one can write the state sum $\Hh_{N,(T,b,c)}(\bu_N,\bv_N)$ of $\Hh_{N}(M,\rho,h_\rho,k_c)$ (the subject of Theorem \ref{thm:main_triang}) as 
$$\Hh_{N,(T,b,c)}(\bu_N,\bv_N) = (\bu_{0,N}^{-1}\bv_{0,N}^{-1})^{\frac{N-1}{2}}\Hh_{N}^{red}(\bu_{0,N},\bu_{1,N},\bv_{0,N},\bv_{1,N}).$$
\end{itemize}
\smallskip

\item (Getting rid of the symmetry defect, Lemma \ref{lem:lim:sym:defect}) With a quick computation, we check that $\textstyle \lim_{N\ra +\infty} (2\pi/N) \Log \left \vert (\bu_{0,N}^{-1}\bv_{0,N}^{-1})^{\frac{N-1}{2}} \right \vert =0$.
\end{enumerate}
\medskip

These two steps imply that in the proof of Theorem \ref{thm:main_triang} we can replace $\Hh_{N,(T,b,c)}(\bu_N,\bv_N)$ with $\Hh_{N}^{red}(\bu_{0,N},\bu_{1,N},\bv_{0,N},\bv_{1,N})$. From known results it has the factorization formula (\ref{form1}):
$$\Hh_{N}^{red}(\bu_{0,N},\bu_{1,N},\bv_{0,N},\bv_{1,N})   = \frac{g_N(\bu_{0,N})g_N(\bv_{0,N}^*)^*}{|g_N(1)|^2}\ \Sigma_N(\bu_{0,N},\bu_{1,N})\ \Sigma_N(\bv_{0,N}^*,\bv_{1,N}^*)^*.$$
\medskip 

\begin{enumerate}
\item[III)] (Getting rid of $g_N(1)$, $g_N(\bu_{0,N})$ and $g_N(\bv_{0,N}^*)$, Formula (\ref{cstBBP}) and Lemma \ref{factorest}) With quick computations, we establish that $|g_N(1)|=\sqrt{N}$, and thus $g_N(1)$ is in $PM(N)$. Also, by using a uniform estimate of the quantum dilogarithm functions (Lemma \ref{estimatehatS}) we show that in the situation of Theorem \ref{thm:main_triang} the functions $g_N(\bu_{0,N})$ and $g_N(\bv_{0,N}^*)$ are in $PM(N)$.

\end{enumerate}
\medskip

Together with the steps I-II, this further step implies that, in order to prove Theorem \ref{thm:main_triang}, it is enough to show that $\textstyle \lim_{N\ra +\infty} (2\pi/N) \Log \left \vert \Sigma_N(.,.)\right \vert$ is equal to either ${\rm Vol}(M)/2$ (case (a)) or $0$ (case (b)), both for $\Sigma_N(\bu_0,\bu_1)$ and for $\Sigma_N(\bv_{0,N}^*,\bv_{1,N}^*)$. 
\medskip 

\begin{enumerate}
\item[IV)] (Writing $\Sigma_N$ as as an integral, Lemma \ref{lemintrep}) Using properties of the quantum dilogarithm $\hat S_N(.)$, we re-write $\Sigma_N(\bu_{0,N},\bu_{1,N})$ as in formula (\ref{int01}) :
$$\Sigma_N(\bu_{0,N},\bu_{1,N}) = 1 + \frac{1}{\hat S_N(\bl_{0,N})}\sum_{\beta=1}^{N-1} e^{\frac{2\pi i}{N} \beta^2 -\bl_1\beta} \textstyle \hat S_N(\bl_{0,N}+2\beta i\pi/N),$$
where $(\bl_{k,N})$ is a sequence of logarithms of $(\bu_{k,N})$ of a specific form. Using the residue theorem, we re-write this formula as in (\ref{int1}):
$$\Sigma_N(\bu_{0,N},\bu_{1,N}) =1+ \frac{1}{\hat S_N(\bl_{0,N})} \frac{N}{4i\pi}\int_{C_N} e^{\frac{N}{2i\pi}(z^2 - \bl_{1,N}z)}\hat S_N(\bl_{0,N} + z)\textstyle \coth(\frac{Nz}{2})dz$$
where $C_N$ is a closed rectangular contour circling the segment $\textstyle [\frac{2i\pi}{N}, 2i\pi - \frac{2i\pi}{N}]$ and close to it. We denote by $is_N^-$ and $2i\pi-is_N^+$ its  intersection points with $i\mr$; by construction $\textstyle s_N^\pm:=\alpha^\pm\frac{\pi}{N}$ with $\alpha^\pm\in ]0,2[$.
\smallskip

\item[V)] (Getting rid of $\hat S_N(\bl_{0,N})$, Lemma \ref{factorest}) With properties of the quantum dilogarithms, we establish that $\hat S_N(\bl_{0,N})$ is in $PM(N)$.
\smallskip

\item[VI)] (Equivalence of Theorem \ref{asyinvteo} and Theorem \ref{thm:main_triang}, Section \ref{loglimmars25}) Using the description of the quantum gluing variety in Section \ref{4_1}, we show that the two cases (a) and (b) in Theorem \ref{asyinvteo} correspond to the cases (a) and (b) in Theorem \ref{thm:main_triang}. Moreover, in the above integral these cases correspond to having $\bl_{k,\infty}:= \textstyle \lim_{N\rightarrow +\infty} \bl_{k,N}$ such that $\bl_{0,\infty} = -2i\pi$, $\bl_{1,\infty} = i\pi$, and $\bl_{0,\infty} = -2i\pi$, $\bl_{1,\infty} =2i\pi$ respectively. The same facts occur for $\Sigma_N(\bv_0^*,\bv_1^*)$.
\end{enumerate}
\medskip

We then turn to the semi-classical limit of the integral
\begin{equation}\label{intclejuil25}\int_{C_N} e^{\frac{N}{2i\pi}(z^2 - \bl_{1,N}z)}\hat S_N(\bl_{0,N} + z)\textstyle \coth(\frac{Nz}{2})dz .
\end{equation}
Somewhat counter-intuitively, we begin with simpler integrals, over the segment $[is_N^-,2i\pi-is_N^+]$ instead of $C_N$, where we fix $\textstyle s_N^-=\alpha^-\frac{\pi}{N}$, $\alpha^-\in ]0,1[$ (whereas $\textstyle s_N^+=\alpha^+\frac{\pi}{N}$ has $\alpha^+\in ]0,2[$ arbitrary), and with integrands defined by using functions $f_{\pm}(z)$, depending on $\bl_{0,\infty}$ and $\bl_{1,\infty}$, which can be regarded as semi-classical approximations of the above integrand when $\bl_{0,\infty}=-2i\pi$ and $\bl_{1,\infty} = i\pi$ or $2i\pi$. 
\begin{enumerate}
\item[VII)] (Saddle point method on a classical integral: contour $[is_N^-,2i\pi-is_N^+]$, case (a): $\bl_{1,\infty} = i\pi$, sections \ref{sub:small:vert:int} and \ref{sub:SPM:a}) In this case we prove that
$$\lim_{N\rightarrow +\infty} \dfrac{2\pi}{N}\Log \left \vert \int_{I_N} e^{\frac{N}{2i\pi} f_{-}(z)}dz \right \vert  = \frac{1}{2}{\rm Vol}(M)\ \ \ {\rm and}\ \ \ \int_{I_N} e^{\frac{N}{2i\pi} f_{+}(z)}dz = O_{N\rightarrow +\infty}(1).$$
\item[VIII)] (Saddle point method on a classical integral: contour $[is_N^-,2i\pi-is_N^+]$,  case (b): $\bl_{1,\infty} = 2i\pi$, sections \ref{sub:small:vert:int} and \ref{sub:SPM:b}) In this case we prove that
$$\int_{I_N} e^{\frac{N}{2i\pi} f_{\pm}(z)}dz  = \dfrac{\text{Constant}_\pm}{\sqrt{N}}(e^{i r_\pm N}+o_{N\to\infty}(1))$$
where $\text{Constant}_-$, $\text{Constant}_+$ are distinct non-zero complex numbers and $r_+,r_-\in \R$ are distinct as well (see \eqref{Cst+}-\eqref{Cst-}).
\smallskip

\item[IX)]  (Lemma \ref{qcintjui25}) We split the contour $C_N$ in its subcontours $C_N^+$, $C_N^-$ lying in the right and left half planes, and we show that the integral \eqref{intclejuil25} is a linear combination of integrals of the form  
$$\int_{C_N^+} e^{\frac{N}{2i\pi} f_{\pm}(z)}\rho(z)\exp\left(\frac{1}{N}{\Psi_{u_0,N}(z)}\right)\left(1+\frac{1}{N}R_N(z)\right) dz,$$
where the functions $f_\pm$ are as in steps VII and VIII, and $\rho(z)$, $\Psi_{u_0,N}(z)$ and $R_N(z)$ are holomorphic functions of $z$ in a domain containing the strip $\mr + i]0,2\pi[$ with some cuts, with moduli bounded by constants independent of $N$ on compact subsets of this strip with cuts. The function $\Psi_{u_0,N}(z)$ measures $\hat S_N(\bl_{0,N} + z)e^{-\frac{N}{2i\pi} f_{\pm}(z)}$; its properties are studied in Lemma \ref{extboundPsiNjanv26}. 

\item[X)]  (Lemma \ref{lem:reformintqjui25} and Corollary \ref{cor:reductionintjui25}) By deforming the contours $C_N^\pm$ to $[is_N^-,2i\pi-is_N^+]$ relatively to the endpoints, one checks that the analysis of steps VII and VIII can be adapted to describe the asymptotics of the integrals of steps IX, and gives similar semi-classical asymptotics. Then, we can deduce easily that
$$ \dfrac{2\pi}{N}\Log \left \vert \int_{C_N} e^{\frac{N}{2i\pi}(z^2 - \bl_{1,N}z)}\hat S_N(\bl_{0,N} + z)\textstyle \coth(\frac{Nz}{2})dz \right \vert $$
and finally
$$\dfrac{2\pi}{N}\Log \left \vert \Sigma_N(\bu_{0,N},\bu_{1,N}) \right\vert \quad \mathrm{and} \quad \dfrac{2\pi}{N}\Log \left \vert \Sigma_N(\bv_{0,N}^*,\bu_{1,N}^*) \right\vert$$
tend to either ${\rm Vol}(M)/2$ when $\bl_{0,\infty} = -2i\pi$ and $\bl_{1,\infty} = i\pi$ (case (a)), or to $0$ when $\bl_{0,\infty} = -2i\pi$ and $\bl_{1,\infty} =2i\pi$ (case (b)).
\end{enumerate}
By the conclusions of steps III and VI, this will complete the proof.

\begin{remark}{\rm The factorization formula of $\Hh_{N}^{red}(\bu_{0,N},\bu_{1,N},\bv_{0,N},\bv_{1,N})$ in step II is specific to the QHI of $M=S^3\setminus K$. On another hand, the integral representation in step IV immediately generalizes to the QHI state sums of arbitrary cusped manifolds, as well as the steps I, III and V.}
\end{remark}

\section{Geometric setup}\label{NOTA} 
We introduce in Sections \ref{QGVsec} and \ref{LOGW} the material required to define the reduced QHI state sums associated to the geometric ideal triangulations of cusped manifolds which support a branching.  The case of the figure-eight knot complement is developed in Section \ref{4_1}. The generalization to arbitrary ({\it i.e.}, not necessarily branched) geometric ideal triangulations, possibly with some flat tetrahedra (see Definition \ref{defgeomtri}), is described in \cite{AGT,NA}. Though any cusped manifold has a geometric ideal triangulation (see Section \ref{sec:tri_GV}), to our knowledge it is still an open question whether or not there is a branched geometric one. For hyperbolic links in $S^3$ which have a reduced alternating diagram, the answer is positive; this follows from the main result of \cite{GMT,SakYok}. 
\smallskip

In this section an odd integer $N\geq 3$ is fixed. Therefore we will systematically omit the subscript ``$N$'' from the notation of the quantum parameters, thus hiding their dependence on $N$ (as they are $N$-th roots). Hence we write $\bw_0$ for $\bw_{0,N}$, and only bold characters indicate the quantum nature of the parameters.

\subsection{The quantum gluing varieties} \label{QGVsec} These form the parameter spaces of the QHI state sum formulas. 

\subsubsection{Quantum $3$-simplices} (see \cite{AGT}, Section 3)\label{sec:q3Delta} A {\it quantum $3$-simplex} is a tuple $(\Delta,b,\bw)$ where:
\begin{itemize}
\item $(\Delta,b)$ is an oriented tetrahedron with ordered vertices, say $v_0,\ldots,v_3$. The letter ``$b$'' denotes the system of edge orientations, called {\it branching}, which encodes the vertex ordering and is defined by the rule: an edge with endpoints $v_i$ and $v_j$ points towards $v_j$ if $j>i$. 

\noindent We call $(\Delta,b)$ a {\it branched tetrahedron}. It has a natural {\it $b$-orientation}, defined by the ordered triple of oriented edges $(v_0v_1)$, $(v_0v_2)$ and $(v_0v_3)$. We encode this $b$-orientation by a sign $*_b(\Delta)\in \{-1,+1\}$: $*_b(\Delta)=+1$ if the $b$-orientation coincides with the given orientation of $\Delta$, $*_b(\Delta)=-1$ otherwise.
\item $\bw:=(\bw_0,\bw_1,\bw_2)\in \mc^*$ satisfies the following requirements: 
\smallskip

\noindent (a) $w:=(w_0,w_1,w_2) := (\bw_0^N,\bw_1^N,\bw_2^N)\in (\mc_*)^3$, and we have the relations
\begin{equation}\label{cyclicshapeoct25}
w_{i+1}=\frac{1}{1-w_i}\quad {\rm (indices\ mod}(3));
\end{equation}
 \noindent (b) We have
 \begin{equation}\label{tetrelbasique}
 \bw_0\bw_1\bw_2 =e^{-*_b(\Delta)\frac{\pi i}{N}}.
 \end{equation}
 We call $w_i$ and $\bw_i$ a {\it shape parameter} and a {\it quantum shape parameter}, respectively.
 \end{itemize} 
 
One can regard $w_0,w_1,w_2$ as the three cross ratios of a $4$-tuple of points in $\mc\mathbb{P}^1=\partial_\infty \mh^3$, ordered up to even permutations. The triple $(w_0,w_1,w_2)$ up to cyclic permutations parametrizes the equivalence class of the $4$-tuple up to the action of Moebius transformations on $\mc\mathbb{P}^1$; equivalently, $(w_0,w_1,w_2)$ up to cyclic permutations parametrizes an orientation preserving isometry class of $3$-dimensional hyperbolic ideal tetrahedra (see {\it e.g.} \cite[Chapter E.6-i]{BP}). Each shape parameter is naturally assigned to a specific pair of opposite edges of $(\Delta,b)$: $w_0$ to the edges $(v_0v_1)$ and $(v_2v_3)$, $w_1$ to the edges $(v_1v_2)$ and $(v_0v_3)$, and $w_2$ to the edges $(v_0v_2)$ and $(v_1v_3)$. This way, e.g., $w_0$ computes the dilation coefficient of the loxodromic element in $PSL(2,\mc)$ fixing $v_0$ and $v_1$ and mapping $v_2$ to $v_3$.

Therefore a quantum $3$-simplex $(\Delta,b,\bw)$ encodes an hyperbolic ideal $3$-simplex up to isometry, plus a choice of $N$-th root $\bw_i$ of each shape parameter $w_i$, assigned to the same pair of opposite edges of $(\Delta,b)$, as described above. 
\medskip

Let $(\Delta,b,\bw)$ be a quantum $3$-simplex. We can write the quantum shape parameters in the form  
\begin{equation}\label{ecoldef}\bw_k = \exp\left(\frac{1}{N}(\Log(w_k) +2\pi i e_k)\right),
\end{equation}
where $e_0,e_1,e_2\in \mz/N$ satisfy
\begin{equation}\label{etetcol}
\sum_{k=0,1,2} (\Log(w_k) +2\pi i e_k) = -*_b\! (\Delta)\pi i\quad {\rm mod}(2\pi i N).
\end{equation}
We call $e_k$ a {\it (edge) $N$-color}. Alternatively, since $N$ is odd the map $\mz \rightarrow \mz/N\mz$, $d\mapsto \textstyle \overline{\frac{(N+1)}{2} d}$, is well-defined and surjective. Therefore it is also possible to write
\begin{equation}\label{coldef}\bw_k = \exp\left(\frac{1}{N}\left(\Log(w_k) +\pi i (N+1)d_k\right)\right),
\end{equation}
where $d_k \in \mz$ and 
\begin{equation}\label{tetcol}
\sum_{k=0,1,2} (\Log(w_k) +\pi i d_k) = -*_b\! (\Delta)\pi i.
\end{equation}
The usefulness of the expression \eqref{coldef} will be explained in Remark \ref{defflatcharge}. We call $d_k$ a {\it (edge) color}, and $\Log(w_k) +\pi i d_k$ a {\it log-parameter}. Since the shape parameters $w_k$ have imaginary parts of the same sign and $\textstyle \sum_{k=0,1,2} \Log(w_k) = {\rm sign}(\mathfrak{I}(w_0))\pi i$, we have
\begin{equation}\label{sumdeven}
d_0+d_1+d_2 = -*_b\! (\Delta)- {\rm sign}(\mathfrak{I}(w_0))\in \{-2,0,2\}.
\end{equation} 

\subsubsection{Gluing varieties and the holonomy map}\label{sec:tri_GV} In this section we explain how the structures carried by quantum $3$-simplices can be globalized on ideal triangulations.
\smallskip

Let $M$ be a cusped manifold, i.e. an oriented non-compact $3$-manifold admitting a complete hyperbolic metric of finite volume. We denote by $\bar{M}$ the compact manifold with interior $M$; $\partial \bar{M}$ is a union of $p$ tori. 

Let  $T$ be an ideal triangulation of $M$, i.e. a gluing by simplicial face pairings of tetrahedra with truncated (``ideal'') vertices, homeomorphic to $M$. 

Denote by $\Delta^1,\ldots,\Delta^n$ the tetrahedra of $T$, with the orientation induced from $M$. We call {\it abstract tetrahedra} and {\it abstract edges} of $T$ the tetrahedra $\Delta^i$ and their edges before their gluings in $T$. Give each tetrahedron $\Delta^i$ a branching $b^i$, as defined in Section \ref{sec:q3Delta}. Denote by $b$ the collection of these local branchings $b^i$. We do not ask that they match along common edges (so we can choose each $b^i$ so that $*_{b^i}(\Delta^i)=1$). We note however that for such global branchings $b$, when they exist (like for $(T,b)$ in Figure 1), the QHI state sums have the simplest form (\cite{AGT}).
\smallskip

\begin{defi} {\rm The {\it quantum gluing variety} $G_N(T,b)$ is the set of tuples $\bw:=(\bw^1,\ldots,\bw^n)\in \mc^{3n}$, where $\bw^i=(\bw_0^i,\bw_1^i,\bw_2^i)$ is a triple of quantum shape parameters on $(\Delta^i,b^i)$, $i=1,\ldots,n$, as in Section \ref{sec:q3Delta}, satisfying the following {\it edge relation} at each edge $e$ of $T$:}
\begin{equation}\label{edgerel}
\prod_{E\ra e} \bw(E)^{*_E(\Delta)} =  e^{-\frac{2\pi i}{N}}.
\end{equation}
{\rm Here the product is over all the abstract edges $E$ identified to $e$ in $T$ (which is denoted $E\ra e$), $\bw(E)$ is the quantum shape parameter of $E$ in the corresponding  tetrahedron $\Delta^i$, and $*_E(\Delta)$ is its branching sign.}
\end{defi}
Note that the branching $b$ enters the definition of $G_N(T,b)$ only in that it specifies a concrete model of the space in $\mc^{3n}$. Different branchings yield isomorphic spaces. In \cite[Section 3B and 4C]{AGT} the space $G_N(T,b)$ was denoted $G_0(T,b,c)_N$. 
\smallskip

For $N=1$, $G(T,b):=G_1(T,b)$ is Thurston's gluing variety associated to $T$, since in this case $\bw^i = w^i$ is a triple of shape parameters $(w_0^i,w_1^i,w_2^i)$ for the tetrahedron $\Delta^i$, satisfying the edge relations $\textstyle \prod_{E\ra e} w(E)^{*_E(\Delta)} =  1$ (with the notations of \eqref{edgerel}). 
\smallskip

It is a well-known fact that $G(T,b)\ne \emptyset$ if all the edges $e$ of $T$ are {\it essential}, {\it i.e.}, $e\cap \bar{M}$ is not homotopic to an arc in $\partial \bar{M}$ by a homotopy fixing the endpoints. In that case we have $G_N(T,b)\ne \emptyset$ (see Theorem \ref{teocoverGN}). Moreover, denote by $\tilde M$ the universal cover of $M$. By ``straightening'' the tetrahedra of $T$ by means of an isometry $\tilde M\ra \mh^3$, one can see that if $G(T,b)\ne \emptyset$, then the character $\rho_{hyp}$ of the discrete faithful representations of $\pi_1(M)$ into $PSL(2,\mc)$ is encoded by a point of $G(T,b)$. 

More generally, when $G(T,b)$ is non-empty it is related to the variety $X(M)$ of {\it $PSL(2,\mc)$-characters of $M$} as follows (see {\it e.g.} \cite{NZ,BP,Martelli}). Recall that $X(M)$ is the affine algebraic set consisting of representations $\pi_1(M)\ra PSL(2,\mc)$ up to the equivalence relation of having the same trace on each element of $\pi_1(M)$. Denote by $\tilde M$ the universal cover of $M$. Any point $w\in G(T,b)$  determines a {\it pseudo-developing map} $D_w\colon \tilde{M} \ra \mh^3$, unique up to post-composition with an isometry, which sends a lift $\tilde{\Delta}^i$ of the tetrahedron $\Delta^i$ to a hyperbolic ideal tetrahedron with shape parameters $(w_0^i,w_1^i,w_2^i)$. A pseudo-developing map $D_w$ defines a representation $\rho_w\colon \pi_1(M) \to PSL(2,\mc)$ by $D_w(\gamma.x) = \rho_w(\gamma)D_w(x)$ for all $\gamma\in \pi_1(M)$, $x\in \tilde{M}$. Different choices of developing maps $D_w$ yield conjugate representations $\rho_w$. The resulting holonomy map (denoting again by $\rho_w$ the character of a representation $\rho_w$)
\begin{equation}\label{defholmapSept24}\fonc{hol}{G(T,b)}{X(M)}{w}{\rho_w}
\end{equation} is regular, and generically $2:1$. More will be said about it in Section \ref{GhypN}. For the moment, let us note that in general $hol$ is not surjective (for instance, by the construction of pseudo-developing maps it is not hard to see that the image of $hol$ cannot contain any character of reducible representations). 
\smallskip

The Thurston gluing variety $G(T,b)$ and the map $hol$ have good properties when $T$ is a geometric triangulation, in the following sense.
\begin{defi}\label{defgeomtri}{\rm We say that $T$ is a {\it geometric triangulation} of $M$ when:
\begin{itemize}
\item either the edges of $T$ are isotopic to geodesics in the complete hyperbolic structure of $M$, and these geodesics triangulate $M$ into tetrahedra with non zero volume and disjoint interiors; 
\item or $T$ is obtained by subdividing the Epstein-Penner canonical polyhedral decomposition of $M$ (then some tetrahedra may be flat, {\it i.e.}, of zero volume).
\end{itemize}}
\end{defi}
In the former case we say that $T$ is a {\it positive} triangulation, as it can be isotoped to a gluing of positively oriented hyperbolic ideal tetrahedra corresponding to the $\Delta^i$, with classical shape parameters satisfying $\mathfrak{I}(w_{hyp}(E)^{*_E(\Delta)})> 0$ for every abstract edge $E$ of $T$. Such triangulations are produced, e.g., by the software program SnapPy \cite{CDGW}. In general, when $T$ is geometric there exists a point $w_{hyp}\in G(T,b)$ such that $$hol(w_{hyp}) = \rho_{hyp}, \ {\rm and\ } \mathfrak{I}(w_{hyp}(E)^{*_E(\Delta)})\geq 0 \ {\rm for\ every\ abstract\ edge\ }E \ {\rm of\ } T.$$ 

For a geometric triangulation $T$ the following facts hold true  (see \cite{NZ},  \cite[Chapter E.6 and Proposition 7.7]{BP}, \cite{PW} and \cite[Proposition 3.5-3.6]{PP},  \cite{Choi} and \cite[Proposition 15.3.4]{Martelli}):
\begin{itemize}
\item[(i)] ${\rm dim}_\mc(G(T,b))=p$, the number of cusps of $M$ (this is also the dimension of the geometric component $X_{hyp}(M)$ of $X(M)$, see Section \ref{GhypN});
\item[(ii)] there is a unique point $w_{hyp}$ as above, it is a smooth point of $G(T,b)$, and a neighborhood of $w_{hyp}$ parametrizes a space of (non complete) hyperbolic structures on $M$.
 \end{itemize} 

Now let $w:=(w^1,\ldots,w^n)\in  G(T,b)$, where $w^i:=(w_0^i,w_1^i,w_2^i)$ as usual. 
\begin{defi} {\rm We say that a tuple $d:=(d^1,\ldots,d^n)\in \mz^{3n}$ is a {\it system of (edge) colors} for $w$ if: 
\begin{itemize}
\item for every tetrahedron $\Delta^i$ of $T$, $i=1,\ldots,n$, we have 
\begin{equation}\label{tetcol2}
\sum_{k=0,1,2} (\Log(w_k^i) +\pi i d_k^i) = -*_b\! (\Delta)\pi i 
\end{equation}
\item for every edge $e$ of $T$ we have (with notations similar to \eqref{edgerel}):
\begin{equation}\label{edgecol}
\sum_{E\ra e} *_E(\Delta)(\Log(w(E)) +\pi i d(E)) =  -2\pi i.
\end{equation}
\end{itemize}
In such a situation we say that $d$ and $w$ are {\it compatible}. }
\end{defi}
The existence and the structure of the set of compatible pairs $(w,d)$ is discussed in Section \ref{LOGW}. By means of the formula \eqref{coldef} relating $d_k^i$, $w_k^i$ and $\bw_k^i$ for every $i$ and $k$, to any compatible pair $(w,d)$ we can associate a point $\bw:=(\bw^1,\ldots,\bw^n)\in G_N(T,b)$, where $\bw^i:=(\bw_0^i,\bw_1^i,\bw_2^i)$. We say that $(w,d)$ and $\bw$ are compatible.

We can also consider {\it systems of (edge) $N$-colors}, {\it i.e.}, tuples $e:=(e^1,\ldots,e^n)\in (\mz/N)^{3n}$, where $e^i=(e_0^i,e_1^i,e_2^i)$, satisfying the relations obtained from \eqref{tetcol2}-\eqref{edgecol} by replacing $d_k^i$ and $d(E)$ with $2e_k^i$ and $2e(E)$ and working {\rm mod}$(2\pi i N)$; clearly the pairs $(w,e)$ are in one-to-one correspondence with the points of $G_N(T,b)$. 

\begin{remark}\label{defflatcharge}{\rm The log-parameters in \eqref{tetcol} are intrinsic to Neumann's simplicial formulas (\cite{N}) of the Chern-Simons invariant of $PSL(2,\mc)$-characters. Indeed, these simplicial formulas are given by functions of the shape parameters $w=(w_k^i)$ and so-called {\it combinatorial flattenings} $f=(f_k^i)\in \mz^{3n}$, which are defined by the following relations, for every tetrahedron $\Delta^i$ and every edge $e$ of $T$ respectively:
\begin{equation}\label{rellogparoct25}
\sum_{k=0,1,2} (\Log(w_k^i) +\pi i f_k^i) = 0\ ,\ \sum_{E\ra e} *_E(\Delta)(\Log(w(E)) +\pi i f(E)) = 0.
\end{equation}
The colors in \eqref{tetcol} can be written
\begin{equation}\label{expdfc}
d_k^i=f_k^i- *_b (\Delta) c_k^i,
\end{equation}
where $c=(c_k^i)\in \mz^{3n}$ is a topological version of combinatorial flattenings, introduced in \cite{N0} and that we call a charge; by definition, it satisfies for every tetrahedron $\Delta^i$ and every edge $e$ of $T$ the equations} 
$$\sum_{k=0,1,2} c_k^i = 1\ ,\ \sum_{E\ra e} c(E) = 2.$$
\end{remark}

\subsection{Weights}\label{LOGW} The structure of the set of systems of edge colors $d$ compatible with a point $w\in G(T,b)$ follows from Neumann's theory of flattenings \cite{N,N0}. This set is a disjoint union of affine spaces over $\mz$-lattices determined by two pairs of cohomology classes on $M$ and $\partial \bar M$, that we call bulk weights and log weights. We refer to \cite[Section 4C]{AGT} for a complete description of these lattices; examples are given in \cite[Section 6.4]{AGT0}, \cite[Section 9]{AGT}, and \cite[Section 9]{NA} (in \cite[Definition 1.7]{NA} we called the log weight \eqref{formulelkappa} below a ``fused weight''). For systems of $N$-edge colors things go similarly. In this section we describe the weights by limiting our discussion to what is strictly necessary for this paper.

\subsubsection{Log weights and boundary weights} \label{sec:landbweights} As usual denote by $\bar M$ the compact manifold with interior $M$; so $\partial \bar M$ is a disjoint union of $c$ tori. Recall the holonomy map \eqref{defholmapSept24}. 
\smallskip

It is well-known that, besides the holonomy $\rho_w$, a point $w\in G(T,b)$ determines a group homomorphism 
\begin{equation}\label{dilationdef}
\fonc{\delta_w}{H_1(\partial \bar M; \mz)}{\mc^*}{\gamma}{\delta_{w}(\gamma).}
\end{equation}
Thus we have a map $\delta\colon G(T,b)\ra H^1(\partial \bar M; \mc^*)$, $w\mapsto \delta_w$. We  call $\delta$ the {\it dilation factor}. The scalar $\delta_w(\gamma)$ is the square of one among the two (reciprocally inverse) eigenvalues of $\rho_w(\gamma)$, and can be computed as a signed monomial in the classical shape parameters $w^i_k$ (see \cite{NZ}, or the formulas in \cite[Lemma E.6.8]{BP}; these formulas immediately follow from \eqref{formulekappa} below). If $T$ is a geometric triangulation, as defined in Section \ref{sec:tri_GV}, and $w$ is close enough to the point $w_{hyp}\in G(T,b)$, then $\delta_w$ is the holonomy of a similarity structure on $\partial \bar M$ (see, e.g., \cite[Proposition E.6.5]{BP}, \cite[Proposition 2.3]{PP}).
\smallskip

Similarly, by adjoining to $w\in G(T,b)$ a compatible system of edge colors $d$, one can define two  classes $h_2\in H^1(M;\mz/2\mz)$ and $l\kappa\in H^1(\partial \bar M; \mc)$, such that for all $\gamma \in H_1(\partial \bar M; \Z)$ we have (where $j^*\colon H^1(\bar{M},\mz/2)\ra H^1(\partial \bar{M},\mz/2)$ is induced by the inclusion $j\colon \partial \bar{M} \ra \bar{M}$):
\begin{align}
l\kappa(\gamma) & = \Log(\delta_{w}(\gamma)) \quad {\rm mod}(i\pi\mz) \label{lkconstraint0}\\
j^*(h_2)(\gamma) & = (l\kappa(\gamma)-\Log(\delta_{w}(\gamma)))/i\pi\quad{\rm mod} (2\mz).\label{h2class}
\end{align}
Let us explain this. To define the classes $l\kappa$ and $h_2$, first we stipulate that $l\kappa(0):=0$ and $h_2(0):=0$, for the $0$ classes in $H_1(\partial \bar M;\mz)$ or $H_1(M;\mz/2)$ respectively. Given a non-zero class $\gamma\in H_1(\partial \bar M;\mz)$, we represent it by a union of oriented simple closed curves in $\partial \bar M$ which are essential ({\it i.e.}, do not bound a disk) and transverse to the $1$-skeleton of $\partial T$, the triangulation of $\partial \bar M$ induced by $T$. We can assume that these curves have no ``back-tracking'', {\it i.e.}, they intersect each triangle of $\partial T$ in a disjoint union of segments. Call such curves {\it normal curves} in $\partial T$. If a triangle $F$ of $\partial T$ is a section of the tetrahedron $\Delta$ of $T$ near one of its vertices, for every vertex $v$ of $F$ denote by $E_v$ the abstract edge of $\Delta$ containing $v$, by $w(E_v)$ the classical shape parameter of $E_v$, by $d(E_v)$ its color, and by $*_v$ the branching sign of $\Delta$. Given a union of normal curves $C$, we write $F\colon C \rightarrow E_v$ to mean that $C$ intersects the triangle $F$ in subsegments joining the two edges of $F$ having $v$ has endpoint. We count these subsegments algebraically, by using the orientation of $C$: if there are $s_+$ (resp. $s_-$) such segments whose orientation is compatible with (resp. opposite to) the orientation of $\partial \bar M$ as viewed from $v$, then we set $ind(C,v) :=s_+-s_-$. Then, for any non zero class $\gamma\in H_1(\partial \bar M;\mz)$ represented by a union of normal curves $C$ in $\partial T$, we set (the sum is over all the triangles of $\partial T$)
\begin{equation}
  l\kappa(\gamma)  := \sum_{F\colon C \ra E_v} *_v \ ind(C,v)\left(\Log(w(E_v)) +\pi \sqrt{-1}d(E_v)\right) \label{formulelkappa}.
  \end{equation}
Similarly, we define the non oriented normal curves in $T$ as transverse to the $2$-skeleton and intersecting each tetrahedron in a disjoint union of segments. For any normal curve in $T$, we write $\Delta\colon C \rightarrow E$ to mean that $C$ intersects the tetrahedron $\Delta$ in subsegments joining the faces of $\Delta$ adjacent to an abstract edge $E$ of $\Delta$. For any such a tetrahedron, we denote by $ind(C,E)$ the number of these segments, and by $d(E)$ the color of $E$. Then, for any non-zero class $\gamma \in H_1(M;\mz/2)$ represented by a union of normal curves $C$ in $T$ we have (the sum is over all the tetrahedra of $T$)
  \begin{equation}
  h_2(\gamma) =\sum_{\Delta\colon C \ra E} \ ind(C,E)d(E) \quad{\rm mod} (2\mz). \label{formulebulkw}
\end{equation}
We call $l\kappa$ the {\it log weight} of $(w,d)$, and $h_2$ the {\it bulk weight}. In the particular case where $M$ is the complement of a knot $K$ in $S^3$, $H^1(\bar M; \mz/2)\cong \mz/2$ is generated by the meridian of $K$ and therefore $h_2$ is determined by $l\kappa$ (and it vanishes on the canonical longitude).
\begin{remark}\label{cfweight}{\rm One can associate classes $k_f\in H^1(\partial\bar{M};\mc)$ and $k_c\in H^1(\partial\bar{M};\mz)$ (and similarly classes $h_{2,f}$, $h_{2,c} \in H_1(M;\mz/2)$), respectively called {\it flattening weight} and {\it charge weight}, to flattenings $f$ and charge $c$ as in Remark \ref{defflatcharge}. Namely, $k_f$ is defined by the formula obtained from \eqref{formulelkappa} by replacing the color $d(E_v)$ with the flattening $f(E_v)$, and $k_c$ by the formula obtained from \eqref{formulelkappa} by replacing $*_v(\Log(w(E_v))+\pi \sqrt{-1}d(E_v))$ with the charge $c(E_v)$. Note that
\begin{equation}\label{fconstraint0}
k_f(\gamma)  = \Log(\delta_{w}(\gamma)) \quad {\rm mod}(i\pi\mz).
\end{equation}
Clearly
\begin{equation}\label{deflkappa}
	l\kappa= k_f-\pi ik_c.
\end{equation}
}\end{remark}

The following statement follows from results in  \cite{N0,N}. In the present notations, the first claim is Propositions 4.8 (2) and 4.9 (2) of \cite{AGT}, formulated for $l\kappa= k_f-\pi ik_c$. The generators in the second claim follow from the last claim of Lemma 6.1 in \cite{N0}, and the rank is a standard computation in the Neumann-Zagier combinatorics (\cite{NZ}). 
\begin{teo}\label{spaced} Let $w\in G(T,b)$. For any classes $h_2\in H^1(M;\mz/2\mz)$ and $l\kappa\in H^1(\partial \bar M; \mc)$ satisfying \eqref{lkconstraint0} and \eqref{h2class} there is a system of edge colors $d$ satisfying \eqref{formulelkappa} and \eqref{formulebulkw}. Moreover, the set formed by these systems of edge colors is an affine space over a $\mz$-lattice of rank $n-p$ (where $p$ is the number of cusps) generated by vectors in one-to-one correspondence with the edges of $T$.
\end{teo}
In such a situation we say that $(h_2, l\kappa)$ and $(w,d)$, or $(h_2, l\kappa)$ and the corresponding tuple of quantum shape parameters $\bw\in G_N(T,b)$ defined by the formula \eqref{coldef}, are {\it compatible}. For simplicity we will often omit $d$ from the notations, and say that $l\kappa$ and $w$, or $l\kappa$ and $\bw$, or $l\kappa$ and the $PSL(2,\mc)$-character $\rho_w = hol(w)$, are compatible.  For a link complement $M$ the map $j^*\colon H^1(\bar{M},\mz/2)\ra H^1(\partial \bar{M},\mz/2)$ is an isomorphism, so $h_2$ is given by $l\kappa$ and $\delta_w$ via the relation \eqref{h2class}.  
\smallskip

Similarly to \eqref{formulelkappa}, denoting by $\bw(E_v)$ the quantum shape parameter of an abstract edge $E_v$, we can associate to any point $\bw\in G_N(T,b)$ a class $\kappa_\bw\in H^1(\partial \bar M; \mc^*)$ defined by
\begin{equation}
 \kappa_\bw(\gamma)  := \prod_{F\colon C \ra E_v} \bw(E_v)^{*_v ind(C,v)} \label{formulekappa}.
\end{equation}
We call $\kappa_\bw$ the {\it boundary weight} of $\bw$. Clearly $\kappa_\bw(\gamma)^{N}= \delta_{w}(\gamma)$, and if $d\in \mz^{3n}$ is a compatible tuple of edge colors with bulk and log weights $h_2$ and $l\kappa$, for every $\gamma\in H_1(\partial \bar M; \Z)$ we have
\begin{equation}\kappa_\bw(\gamma)  = \exp\left(\frac{l\kappa(\gamma)}{N}+(\pi i) j^*(h_2)(\gamma)\right). \label{kappaborddef}
\end{equation}
Consider the covering map
\begin{equation}\label{coverGN}
\fonc{\pi_N}{G_N(T,b)}{G(T,b)}{\bw = (\bw_k^i)}{w = ((\bw_k^i)^N).}
\end{equation}
Theorem \ref{spaced} (or its proof, using systems of $N$-colors) implies:
\begin{teo}\label{teocoverGN} The map $\pi_N$ has degree $N^{n+p}$. 
\end{teo}
Note that the exponent decomposes as $n+p=n-p+2p$, where $n-p$ follows from Theorem  \ref{spaced}, and $2p$ comes from the degree of freedom in the choice of $\kappa_\bw$. 

\subsubsection{The spaces $G_{hyp,N}(T,b)$ and ${}_NX_{hyp}(M)$}\label{GhypN}{\rm (See \cite[Section 4]{AGT}.) Let $T$ be a geometric triangulation of $M$. We saw in Section \ref{sec:tri_GV} that $w_{hyp}$ is a smooth point of $G(T,b)$. Denote its irreducible component by $G_{hyp}(T,b)$. In this section we discuss some of its remarkable properties, and explain how the map \eqref{coverGN} can be given an intrinsic definition over $G_{hyp}(T,b)$.
\smallskip

We need to recall the following results:
\smallskip

(a) The character $\rho_{hyp}$ of the discrete faithful representations of $\pi_1(M)$ into $PSL(2,\mc)$ is a smooth point of $X(M)$, and thus is contained in a unique irreducible component $X_{hyp}(M)$ of $X(M)$; moreover ${\rm dim}_\mc(X_{hyp}(M)) = p$, the number of cusps of $M$ (see \cite[Chapter 4.5]{Shalen}, \cite[Theorem 8.44]{Kapo}). We call $X_{hyp}(M)$ the {\it geometric component} of $X(M)$.
\smallskip

(b) Dunfield \cite{Dun} showed that the (regular) restriction map $r\colon X_{hyp}(M) \ra X(\partial \bar M)$ is a birational isomorphism onto its image, in the case when $M$ has a single cusp.  This result has been extended by Klaff-Tillmann \cite{KT} to an arbitrary number of cusps. 
\smallskip

(c) The $PSL(2,\mc)$-characters of $M$ can be ``augmented'' (\cite{Champ,BD}): an augmented character $\tilde{\rho}$ consists of a character $\rho\in X(M)$ together with a conjugation invariant choice of a fixed point of $\rho(\pi_1(\partial \bar{M}))$, for each peripheral subgroup $\pi_1(\partial \bar{M})$ of $\pi_1(M)$. Such a choice is equivalent to selecting the square of one among the two (reciprocally inverse) eigenvalues of $\rho(\gamma)$, for each peripheral element $\gamma$ of $\pi_1(M)$; hence it defines a class $\delta_{\tilde \rho}\in H^1(\partial \bar M;\mc^*)$. The set of augmented characters $\tilde{X}(M)$ is algebraic, and has a canonical irreducible component $\tilde{X}_{hyp}(M)$ containing the augmentation of $\rho_{hyp}$ (it is unique, since $\rho_{hyp}$ maps elements of $\pi_1(\partial \bar M)$ to parabolic transformations). 
\smallskip

Now, identify $H^1(\partial \bar M;\mc^*)$ with $(\mc^*\times \mc^*)^p$ by fixing a basis $(\mu_1,\lambda_1,\ldots, \mu_p,\lambda_p)$ of $H_1(\partial \bar M;\mz)$. The map $\tilde{X}_{hyp}(M) \rightarrow X_{hyp}(M)$, $\tilde{\rho}= (\rho,\delta_{\tilde \rho}) \mapsto \rho$, is regular, generically $2^p\colon 1$, and the map
$$\fonc{\delta^X}{\tilde{X}_{hyp}(M)}{H^1(\partial \bar M;\mc^*)\cong (\mc^*\times \mc^*)^p}{\tilde{\rho}}{\delta_{\tilde \rho}}$$ is regular. The map $(\mc^*\times \mc^*)^p\ra X(\partial \bar M)=X(T^2)^p$, which sends the $i$-th entry $(m_i^2,\ell_i^2)$, $i=1,\ldots, p$, to the coset in $X(T^2)$ of the pair
$$\left( \pm \begin{pmatrix} m_i & 0 \\ 0 & m_i^{-1}\end{pmatrix}, \pm \begin{pmatrix} \ell_i & 0 \\ 0 & \ell_i^{-1}\end{pmatrix}\right) \in PSL(2,\mc)^2,$$ is generically $2^p\colon 1$. Since $r\colon X_{hyp}(M) \ra X(\partial \bar M)$ has degree one, it follows that $\delta^X$ has degree one, and is therefore a birational isomorphism onto its image. We thus get the bottom right square of the following commutative diagram, whose description we complete below.
$$\xymatrix{ &  {}_NX_{hyp}(M) \ar[d]^{\pi_N^X} \ar[r]^{\kappa^X} & (\mc^* \times \mc^*)^p \ar[d]^{((.)^N \times (.)^N)^p} \\ G_{hyp,N}(T,b)  \ar[d]_{\pi_N} \ar[r] \ar[ru]^{\widetilde{hol}} & \tilde{X}_{hyp}(M) \ar[r]^{\delta^X} \ar[d]^{2^p:1} & (\mc^* \times \mc^*)^p \ar[d]^{2^p:1}\\G_{hyp}(T,b) \ar[r]^{hol_{\vert}}\ar[ru]^{\tilde{hol}_{\vert}} & X_{hyp}(M) \ar[r]^{r} & X(\partial \bar M)}$$

Let $T$ be a geometric triangulation of $M$, and $G_{hyp}(T,b)\subset G(T,b)$ the irreducible component of $w_{hyp}$. Using the dilation factor \eqref{dilationdef}, the holonomy map $hol$ in \eqref{defholmapSept24} lifts to a regular map $\tilde{hol}\colon G(T,b)\rightarrow \tilde{X}(M)$, $w\mapsto (\rho_w,\delta_w)$, such that that $\delta_w = \delta^X \circ \tilde{hol}(w)$. We have:
\begin{prop} {\rm (\cite[Proposition 4.6]{AGT} and \cite[Remark 1.4]{NA}) The restriction  
\begin{equation}\label{tildeholresmars25}
\tilde{hol}_{\vert}\colon G_{hyp}(T,b)\rightarrow \tilde{X}_{hyp}(M)
\end{equation}
is a birational isomorphism, and the dilation factor $\delta\colon G_{hyp}(T,b)\rightarrow H^1(\partial \bar M;\mc^*)\cong (\mc^*\times \mc^*)^p$ is a birational isomorphism onto its image.}
\end{prop}
The proof combines Thurston's hyperbolic Dehn filling theorem in the form of \cite{PP, Martelli} with Dunfield's arguments showing that $r$ has degree one.

Set
\begin{equation}\label{GhypNdef}
G_{hyp,N}(T,b):= \pi_N^{-1}(G_{hyp}(T,b)),
\end{equation} with $\pi_N$ the map \eqref{coverGN}. Define ${}_NX_{hyp}(M)$ as the set of pairs $(\rho,\kappa)$ where $\rho \in X_{hyp}(M)$, $\kappa\in H^1(\partial \bar M;\mc^*)$, and $\kappa^{N}(\gamma)$ is a squared eigenvalue of $\rho(\gamma)$ for all $\gamma\in H_1(\partial \bar M;\mz)$. Thus ${}_NX_{hyp}(M)$ has the structure of a fibered product $X_{hyp}(M) \times_{X(\partial \bar M)} H^1(\partial \bar M;\mc^*)$. Clearly, the map $\pi_N^X\colon {}_NX_{hyp}(M)\ra \tilde X_{hyp}(M)$, $(\rho,\kappa)\mapsto (\rho, \kappa^N)$, has degree $N^{2p}$. Consider the ``boundary weight map'' (see \eqref{kappaborddef})
\begin{equation}\label{defweightSept24}
\fonc{\kappa}{G_{hyp,N}(T,b)}{H^1(\partial \bar M;\mc^*)}{\bw}{\kappa_\bw.}
\end{equation}
Since $\kappa_\bw(\gamma)^{N}= \delta_w(\gamma)$, we can lift $\tilde{hol}_{\vert}$ to a regular map
\begin{equation}\label{grostildeholmars25}
\fonc{\widetilde{hol}}{G_{hyp,N}(T,b)}{{}_NX_{hyp}(M)}{\bw}{(\rho_w,\kappa_\bw).}
\end{equation}
It has degree $N^{n-p}$ (see Theorem \ref{teocoverGN} and the comment thereafter), and by construction $$\pi_N = (\tilde{hol}_{\vert})^*\pi_N^X.$$ Finally $\kappa^X\colon {}_NX_{hyp}(M) \ra \mc^* \times \mc^*$ is defined by $\kappa^X(\rho,\kappa) = (\kappa(\mu_K), \kappa(\lambda_K))$; so $\kappa = \kappa^X \circ \widetilde{hol}$.

\begin{remark}\label{mostrho}{\rm (see \cite[Proposition 4.6 and pages 2017-2019]{AGT} and \cite[Remark 1.4]{NA}). Denote by EP the set of triangulations obtained by subdividing the Epstein-Penner polyhedral decomposition  of $M$. For every $T\in EP$, fix a branchings $b_T$ of $T$ as at the beginning of Section \ref{sec:tri_GV}, and set $Z_T:=G_{hyp}(T,b_T)\subset G(T,b_T)$. There is a maximal Zariski open subset $\Omega_{Z_T}$ in $Z_T$ containing $w_{hyp}$, such that $hol(\Omega_{Z_T})$ is Zariski open in $X_{hyp}(M)$ and $hol\colon \Omega_{Z_T} \ra hol(\Omega_{Z_T})$ is a homeomorphism. Then $$\Omega := \bigcap_{T\in {\rm EP}} hol(\Omega_{Z_T})$$
is a canonical Zariski open subset $\Omega$ of $X_{hyp}(M)$ which contains $\rho_{hyp}$ and whose points can be (uniquely) encoded by points in the components $G_{hyp}(T,b)$, for every $T\in EP$.}
\end{remark} 

\subsection{The case of the figure-eight knot complement}\label{4_1} Fix an orientation of $M=S^3\setminus 4_1$. Denote by $\lambda_{K}$ and $\mu_{K}$ the canonical longitude and the meridian of $K$, oriented so that the ordered pair $(\mu_K,\lambda_K)$ followed by the outward normal to $\partial \bar{M}$ defines a positive frame of $\bar{M}$. We denote also by $\lambda_K$, $\mu_K$ their classes in $H_1(\partial \bar M; \mz)$.

In this section we implement on $M$ the structures defined in Sections \ref{QGVsec} and \ref{sec:landbweights}. We will use the geometric, positive ideal triangulation $T$ of $M$ shown in Figure 1, with branching $b$ as indicated. The two tetrahedra $\Delta^0$, $\Delta^1$ have opposite $b$-orientations, and $\Delta^0$ is the positively $b$-oriented one; note that they were denoted $U$, $V$ in Section \ref{QHIintro2}.

Denote by $u:=(u_0,u_1,u_2)$ and $\bu:=(\bu_0,\bu_1,\bu_2)$ triples of classical and quantum shape parameters on $\Delta^0$, and by $(a_0,a_1,a_2)$ a triple of edge colors compatible with $\bu$. Denote similarly by $v:=(v_0,v_1,v_2)$, $\bv:=(\bv_0,\bv_1,\bv_2)$ and $(b_0,b_1,b_2)$ the corresponding data for $\Delta^1$. 

The Thurston gluing variety $G(T,b)$ is well-known, and can be described as follows. For the computation of the formulas \eqref{eqedgeu0v0} and \eqref{longmerformdec25} below, we refer to \cite[Section 15.1.4]{Martelli}, where we note that both tetrahedra are given positively oriented branchings, the classical shape parameters are ordered differently (the parameters $w$ and $z$ are $u_2$ and $1-v_2 = v_0^{-1}$ in our notations), and the curves $m$ and $l$ are $\mu_K$ and $\mu_K^{-2}\lambda_K$. Now, $G(T,b)$ can be identified with the irreducible curve in $\mc^2$ with coordinates $(u_0,v_0)$ and defining equation 
$$u_1u_2^2v_0^{-2}v_1^{-1}=1,$$
or equivalently
\begin{equation}\label{eqedgeu0v0}
u_0^2v_0^2=(1-u_0)(1-v_0).
\end{equation}
The point $w_{hyp}\in G(T,b)$ with coordinates $(u_0,v_0)=(e^{\pi i/3}, e^{-\pi i/3})$ realizes the complete structure on $M$, and any point $w\in G(T,b)$ sufficiently close to $w_{hyp}$ determines a (non-complete) hyperbolic structure on $M$ (see (ii) in Section \ref{sec:tri_GV}). 

For every point $w=(u,v)\in G(T,b)$ the dilation factor $\delta_{w}\in H^1(\partial \bar M; \mc^*)$ in \eqref{dilationdef} is given by  
\begin{equation}\label{longmerformdec25}
\delta_{w}(\lambda_K) = u_0^2u_2^{-2} = \frac{u_0^4}{(1-u_0)^2}\ ,\ \delta_{w}(\mu_K) = u_2v_2 = \frac{(u_0-1)(v_0-1)}{u_0v_0}.
\end{equation}

The quantum gluing variety $G_N(T,b)$ is the subspace of $\mc^6$ formed by the pairs $(\bu,\bv)$ of triples of quantum shape parameters $\bu:=(\bu_0,\bu_1,\bu_2)$ and  $\bv:=(\bv_0,\bv_1,\bv_2)$ satisfying the tetrahedral and edge relations
\begin{equation}\label{eqGN}
\begin{array}{ccc} \bu_0\bu_1\bu_2 =  e^{-\frac{\pi i}{N}} & ,& \bv_0\bv_1\bv_2 =  e^{+\frac{\pi i}{N}}\\
\bu_1\bu_2^2\bv_0^{-2}\bv_1^{-1} = e^{-\frac{2\pi i}{N}} & ,&  \bu_1\bu_0^2\bv_2^{-2}\bv_1^{-1}= e^{-\frac{2\pi i}{N}}.
\end{array}
\end{equation}
Using the tetrahedral relations it is easy to see that the two edge relations are equivalent. Thus we can identify $G_N(T,b)$ with the set of tuples $(\bu_0,\bu_1,\bv_0,\bv_1)$ satisfying $\bu_1\bu_0^2\bv_2^{-2}\bv_1^{-1}= e^{-\frac{2\pi i}{N}}$, for instance. Since $G(T,b)$ is an irreducible curve, we have
\begin{equation}\label{Ghypcas41mars25}
G_{hyp}(T,b)=G(T,b)\ {\rm and}\ G_{hyp,N}(T,b)=G_N(T,b).
\end{equation}

Given a point $\bw:=(\bu,\bv)\in G_N(T,b)$, let $w:=(u,v)\in G(T,b)$, and $d:=(a,b)$ be a system of edge colors such that $(w,d)$ is compatible with $\bw$; hence they are related by the formula \eqref{coldef}. Consider the log weight $l\kappa$ associated to $(w,d)$ by \eqref{formulelkappa}, and the boundary weight $\kappa:=\kappa_{\bw}$ associated to $\bw$ by \eqref{formulekappa}. Recall the bulk weight $h_2\in H^1(M; \mz/2)$ in \eqref{formulebulkw}, and the relation \eqref{kappaborddef}. As $M$ is a knot complement we have $h_2(\lambda_K)=0$. Therefore
\begin{equation}\label{compbord}
\kappa_\bw(\lambda_K) = \bu_0^{2}\bu_2^{-2}   = e^{\frac{l\kappa(\lambda_K)}{N}}\ ,\ \kappa_\bw(\mu_K) = \bu_2\bv_2  = (-1)^{h_2(\mu_K)} e^{\frac{l\kappa(\mu_K)}{N}}.
\end{equation}
For $k=0,1,2$, let us write $\bu_k$, $\bv_k$  as in \eqref{coldef}, in terms of edge colors and rescaled log-parameters $\bl_{k}^\bu:=N^{-1}(\Log(u_k)+\pi i a_k)$, $\bl_{k}^\bv:=N^{-1}(\Log(v_k)+\pi i b_k)$; so 
\begin{equation}\label{qshapeuvoct25}
\bu_{k} = \exp(\bl_{k}^\bu+\pi i a_k)\ ,\ \bv_{k} = \exp(\bl_{k}^\bv+ \pi i b_k).
\end{equation}
The relations \eqref{tetcol2}, \eqref{edgecol} and \eqref{formulelkappa} provide logarithms of \eqref{eqGN} and \eqref{compbord} of the form 
\begin{equation}\label{eqlog}
\left\lbrace\begin{array}{rcl}
\bl_{0}^\bu+\bl_{1}^\bu+\bl_{2}^\bu & = &\textstyle -\frac{\pi i}{N}+\pi i(a_0+a_1+a_2)\\
\bl_{0}^\bv+\bl_{1}^\bv+\bl_{2}^\bv & = & \textstyle +\frac{\pi i}{N}+\pi i(b_0+b_1+b_2)\\
\bl_{1}^\bu+2\bl_{2}^\bu-2\bl_{0}^\bv-\bl_{1}^\bv & = & \textstyle -\frac{2\pi i}{N}+\pi i(a_1+2a_2-2b_0-b_1)\\
2\bl_{0}^\bu-2\bl_{2}^\bu & = &\textstyle \frac{l\kappa(\lambda_K)}{N}+\pi i(2a_0-2a_2)\\
\bl_{2}^\bu+\bl_{2}^\bv & = & \textstyle \frac{l\kappa(\mu_K)}{N} +\pi i(a_2+b_2).
\end{array}\right.
\end{equation}
Note that $h_2(\mu_K) = a_2+b_2$ mod$(2)$ by \eqref{formulebulkw}; all other sums of edge colors on the right side of the system \eqref{eqlog} are even integers ({\it e.g.}, by \eqref{sumdeven}). 

The system \eqref{eqlog} may be rewritten in a very convenient form by fixing the point $w\in G(T,b)$, and working solely with edge colors. To see this, assume at first that $w=w_{hyp}$, so we have $(u_0,v_0)=(e^{\frac{\pi i}{3}}, e^{-\frac{\pi i}{3}})$. Then by \eqref{tetcol} and \eqref{edgecol} the tetrahedral and edge relations between edge colors are (they correspond to the first three relations in \eqref{eqlog}):
\begin{align}a_2=-2-a_0-a_1\ ,\ b_2=2-b_0-b_1\label{tetrel41}\\
a_1+2a_2-2b_0-b_1 = -4.\quad \quad \quad \label{edgerel41}
\end{align}
As $w_{hyp}$ realizes the complete hyperbolic structure, we have $\delta_{w_{hyp}}(\mu_K)=\delta_{w_{hyp}}(\lambda_K)=1$, and the identity $\kappa^N = \delta_{w_{hyp}}$ implies $l\kappa(\mu_K)\in \pi i \mz$ and $l\kappa(\lambda_K)\in 2\pi i \mz$. The relations \eqref{tetrel41}, \eqref{edgerel41} and the formula \eqref{formulelkappa} give:
\begin{equation}\label{standardf3} a_1=-2a_0+\frac{l\kappa(\lambda_K)}{2\pi i}-2\ ,\ b_0 = \frac{l\kappa(\mu_K)}{\pi i}-a_0\ ,\ b_1 = 2a_0-\frac{2l\kappa(\mu_K)}{\pi i}-\frac{l\kappa(\lambda_K)}{2\pi i}+2.
\end{equation}
Next consider the case when $w\in G(T,b)$, $w\ne w_{hyp}$. If the coordinates $u_0$, $v_0$ of $w$ satisfy $\mathfrak{I}(u_0)>0$ and $\mathfrak{I}(v_0)<0$, the relations \eqref{tetrel41} and \eqref{edgerel41} are unchanged, and those in \eqref{standardf3} still hold true if one replaces $l\kappa(\cdot)$ with $h_\rho(\cdot) =l\kappa(\cdot )-\Log(\delta_w(\cdot))$ (see \eqref{dconstraintclass}). For an arbitrary point $w\in G(T,b)$ the relations are obtained by doing the same replacement in \eqref{standardf3}, possibly adding even integers to the colors $a_k$, $b_k$ (depending on $w$). 

\subsection{The logarithmic limit set}\label{ASYWEIGHTS} Recall the spaces $G_{hyp,N}(T,b)$ and ${}_NX_{hyp}(M)$, where $N\geq 3$ is odd (see Section \ref{GhypN}); we saw that when $M=S^3\setminus 4_1$ these spaces are respectively the solution set of the equations \eqref{eqGN}, and (up to birational isomorphism) the plane curve whose points are  pairs $(\kappa_\bw(\mu_K),\kappa_\bw(\lambda_K))$ as in \eqref{compbord}. Here we introduce a ``limit'' of the tower of spaces $\{G_{hyp,N}(T,b)\}_N$, that will be used in Section \ref{sec:class_int}. It is related in a natural way (via holonomy maps) to a tropical curve associated to the tower of curves $\{{}_NX_{hyp}(M)\}_N$, and to the $PSL(2,\mc)$ $A$-polynomial of $M$ (see Remark \ref{tropremk}). 

We use the notations of the previous section. As we are going to let $N$ vary and consider corresponding sequences of points in the spaces $G_{hyp,N}(T,b)$, we reintroduce the dependence in $N$ in the notations; thus the points of $G_{hyp,N}(T,b)$ are the tuples $(\bu_{0,N},\bu_{1,N},\bv_{0,N},\bv_{1,N})$ satisfying \eqref{eqGN}. Also for ${\bf x}\in \{\bu,\bv\}$ and $k\in\{0,1,2\}$, setting $u_{k,N} := \bu_{k,N}^N$ we get tuples $(u_{0,N},u_{1,N},v_{0,N},v_{1,N})$ in $G_{hyp}(T,b)$, possibly varying with $N$. As in \eqref{ecoldef} we write
\begin{equation}\label{defseqNcolor}\bl_{k,N}^\bu := \frac{1}{N}\left(\Log(u_{k,N})+2\pi i a_{k,N}' \right),\  \bl_{k,N}^\bv := \frac{1}{N}\left(\Log(v_{k,N})+2\pi i b_{k,N}' \right).
\end{equation}
Therefore $a_{k,N}',b_{k,N}'\in \mz/N$ are $N$-colors, and $\bx_{k,N} = \exp(\bl_{k,N}^\bx)$. Put
\begin{equation}\label{deflimlkmars26}
\bl_{k,\infty}^{\bf x} := \lim_{N\rightarrow \infty} \bl_{k,N}^{\bf x}  \in \mc/2\pi i \mz
\end{equation}
assuming the limit exists. In this case, by the relations \eqref{compbord} the sequence $(\kappa_N)_N$ converges in $H^1(\partial \bar M; \mc^*) \cong (\mc^*)^2$. Then we put
$$\kappa_\infty := \lim_{N\rightarrow \infty}\kappa_N  \in H^1(\partial \bar M; \mc^*).$$
Consider the embedding
$$\Theta_N := {\rm Log}\vert \Theta_N\vert \times \mathcal{A}(\Theta_N)\colon G_{hyp,N}(T,b)\rightarrow \mathbb{R}^{4}\times (\mathbb{S}^1)^4$$
where
\begin{align*}
 & \fleche{{\rm Log}\vert \Theta_N\vert\colon G_{hyp,N}(T,b)}{\mathbb{R}^{4}}{(\bu_{0,N},\bu_{1,N},\bv_{0,N},\bv_{1,N})}{\left(\frac{1}{N} {\rm Log}\vert u_{0,N}\vert,\frac{1}{N} {\rm Log}\vert u_{1,N}\vert,\frac{1}{N} {\rm Log}\vert v_{0,N}\vert,\frac{1}{N} {\rm Log}\vert v_{1,N}\vert\right)}\\
& \fleche{\mathcal{A}(\Theta_N)\colon G_{hyp,N}(T,b)}{(\mathbb{S}^1)^4}{(\bu_{0,N},\bu_{1,N},\bv_{0,N},\bv_{1,N})}{\left(e^{2\pi i \frac{a_{0,N}'}{N}},e^{2\pi i \frac{a_{1,N}'}{N}},e^{2\pi i \frac{b_{0,N}'}{N}}, e^{2\pi i \frac{b_{1,N}'}{N}}\right).}
\end{align*}
The maps ${\rm Log}\vert \Theta_N\vert$ factor to maps defined on $G_{hyp}(T,b)$, for which we keep the same notation. The subspace $\mathbb{R}^{4}\times (\mathbb{S}^1)^4\subset \mathbb{R}^{4}\times\mathbb{R}^{4}$ has the metric $d$ induced from the Euclidean distance, and the set of its subsets inherits the corresponding Hausdorff (pseudo-)distance: $\textstyle d(X,Y) = \max(\sup_{x\in X}(d(x,Y)), \sup_{y\in Y}(d(X,y)))$ for all $X,Y\subset \mathbb{R}^{4}\times (\mathbb{S}^1)^4$. 
\begin{lem}\label{lemYo} The sequence $\left(\Theta_N( G_{hyp,N}(T,b))\right)_N$ converges to $ \mathcal{C}(T,b)\times \mathcal{A}(T,b)\subset \mathbb{R}^{4}\times (\mathbb{S}^1)^4$ as $N\rightarrow +\infty$, $N$ odd, where $\mathcal{A}(T,b):= \{(z_1,z_2,z_3,z_4)\in (S^1)^4\ \vert \ z_1^2z_2z_3^2z_4=1\}$, and $\mathcal{C}(T,b)$ is the union of the following four half lines in $\mr^4$ (parametrized in $\mr^2\times \mr^2$):
$$\begin{array}{ccc} L_1 = \{t((-1,0),2(0,1)), t\geq 0\} & ,& L_2 = \{t(2(0,1),(-1,0)), t\geq 0\}\\
L_3 = \{t((-1,0),2(1,-1)), t\geq 0\} & , & L_4 = \{t(2(1,-1),(-1,0)), t\geq 0\}.
\end{array}$$ 
\end{lem}  
\proof The convergence of ${\rm Log}\vert \Theta_N\vert(G_{hyp}(T,b))$ to $\mathcal{C}(T,b)$ is the only non obvious fact; it follows from the Neumann-Zagier combinatorics \cite{NZ}, and is shown in Example 3.1 of \cite{Yo} (changing notations $z_1$, $z_2$ to $u_0^{-1}$, $v_0^{-1}$, respectively). The vectors $(-1,0)$, $(0,1)$, $(1,-1)$ used to define the half-lines $L_1,\ldots, L_4$ are the directions where $({\rm Log}\vert u_0\vert,{\rm Log}\vert u_1\vert)$ (resp. $({\rm Log}\vert v_0\vert,{\rm Log}\vert v_1\vert)$) tends as $u_0$ (resp. $v_0$) tends to $0$, $1$ or $\infty$. The images of the bounded subsets of $G_{hyp}(T,b)$ by the maps ${\rm Log}\vert \Theta_N\vert$ converge to $\{0\}\subset \mathbb{R}^{4}$ as $N\rightarrow +\infty$, hence the limit of ${\rm Log}\vert \Theta_N\vert(G_{hyp}(T,b))$ is the cone over its asymptotic directions, and the task is to compute these directions. 

Consider now the sets $\mathcal{A}(\Theta_N)(G_{hyp,N}(T,b))$. By rewriting the system \eqref{eqlog} in terms of the parameters $\bl_{k,N}^\bu$, $\bl_{k,N}^\bv$ in \eqref{defseqNcolor} and assuming the limits \eqref{deflimlkmars26} exist, as $N\rightarrow +\infty$ we get
\begin{equation}\label{eqloginfinity}
\left\lbrace\begin{array}{rcl}
\bl_{0,\infty}^{\bu}+\bl_{1,\infty}^{\bu}+\bl_{2,\infty}^{\bu} & = &0 \quad {\rm mod}(2\pi i)\\
\bl_{0,\infty}^{\bv}+\bl_{1,\infty}^{\bv}+\bl_{2,\infty}^{\bv} & = & 0 \quad {\rm mod}(2\pi i)\\
\bl_{1,\infty}^{\bu}+2\bl_{2,\infty}^{\bu}-2\bl_{0,\infty}^{\bv}-\bl_{1,\infty}^{\bv} & = & 0 \quad {\rm mod}(2\pi i)\\
2\bl_{0,\infty}^{\bu}-2\bl_{2,\infty}^{\bu}& = &\Log(\kappa^\infty(\lambda_K))  \quad {\rm mod}(2\pi i)\\
\bl_{2,\infty}^{\bu}+\bl_{2,\infty}^{\bv} & = & \Log(\kappa^\infty(\mu_K))  \quad {\rm mod}(2\pi i).
\end{array}\right.
\end{equation}
Note that the following limits take values in $\mr/2\pi\mz$:
$$\alpha_k^\infty := \mathfrak{I}(\bl_{k,\infty}^{\bu}) =  \textstyle \lim_{N\rightarrow \infty} \frac{2\pi a_{k,N}'}{N}\ ,\ \beta_k^\infty := \mathfrak{I}(\bl_{k,\infty}^{\bv}) =  \textstyle \lim_{N\rightarrow \infty} \frac{2\pi b_{k,N}'}{N}.$$ In particular the classical shape parameters $u_{k,N}$ and $v_{k,N}$ do not contribute to these limits. Taking imaginary parts in the system \eqref{eqloginfinity}, we find that $\alpha_k^\infty, \beta_k^\infty\in \mr/2\pi \mz$, for $k\in \{0,1,2\}$, satisfy the relations:
\begin{equation}\label{eqarginfinity}
\left\lbrace\begin{array}{rcl}
\alpha_0^\infty+\alpha_1^\infty+\alpha_2^\infty & = &0 \quad {\rm mod}(2\pi )\\
\beta_0^\infty+\beta_1^\infty+\beta_2^\infty & = & 0 \quad {\rm mod}(2\pi )\\
\alpha_1^\infty+2\alpha_2^\infty-2\beta_0^\infty-\beta_1^\infty & = & 0 \quad {\rm mod}(2\pi )\\
2\alpha_0^\infty-2\alpha_2^\infty & = &\arg(\kappa^\infty(\lambda_K))  \quad {\rm mod}(2\pi )\\
\alpha_2^\infty+\beta_2^\infty & = & \arg(\kappa^\infty(\mu_K))  \quad {\rm mod}(2\pi ).
\end{array}\right.
\end{equation}
The first three relations are the only constraints on $\alpha_k^\infty, \beta_k^\infty$, where $k\in \{0,1,2\}$ (in particular there is no contribution of the argument of the classical dilation factors $\arg(\delta_{w_N})$ in $\arg(\kappa^\infty(\cdot))$). Clearly then $(e^{i\alpha_0^\infty}, e^{i\alpha_1^\infty},e^{i\beta_0^\infty},e^{i\beta_1^\infty})\in \mathcal{A}(T,b)$, which shows that the sets $\mathcal{A}(\Theta_N)(G_{hyp,N}(T,b))$ degenerate to $\mathcal{A}(T,b)$ as $N\ra+\infty$.\endproof
\medskip

Assigning angles $(\alpha_0^\infty,\alpha_1^\infty,\alpha_2^\infty - \pi)$ and $(\beta_0^\infty,\beta_1^\infty,\beta_2^\infty-\pi)$ to $\Delta^0$ and $\Delta^1$, the system \eqref{eqarginfinity} shows that $\mathcal{A}(T,b)$ is a space of generalized angle structures on $T$, in a sense similar to \cite{LT}. These angle structures determine an asymptotic angular holonomy $\arg(\kappa^\infty(\cdot )) \in H^1(\partial \bar M;\mr/2\pi)$ by the relations (equivalent to the last two ones in \eqref{eqarginfinity}):
\begin{align}
4\alpha_0^\infty+2\alpha_1^\infty & = \arg(\kappa^\infty(\lambda_K)) \quad {\rm mod}(2\pi ),\label{asycurve}\\
\alpha_0^\infty+\alpha_1^\infty+\beta_0^\infty+\beta_1^\infty & =  -\arg(\kappa^\infty(\mu_K))  \quad {\rm mod}(2\pi ).\label{asycurvemu}
\end{align}

\begin{remark}\label{tropremk} {\rm (1) In the terminology of tropical geometry (\cite{V}) the set ${\rm Log}\vert \Theta_1\vert(G_{hyp}(T,b))$ is the {\it amoeba} of $G_{hyp}(T,b)$, and $\mathcal{C}(T,b)$ the {\it tropical curve} of $G_{hyp}(T,b)$. Also, let us identify $\mr^4$ with the interior of its closed unit ball $D^{4}$ by means of the map $x\mapsto x/(1+ \vert \vert \cdot \vert \vert)$, where $\vert \vert \cdot \vert \vert$ is the euclidian norm in $\mr^4$. Then the closure $\overline{\mathcal{C}(T,b)}$ of $\mathcal{C}(T,b)$ in $D^{4}$ has boundary $\partial_\infty \mathcal{C}(T,b) := \overline{\mathcal{C}(T,b)}\setminus \mathcal{C}(T,b)\subset S^3$,  $\mathcal{C}(T,b)$ is the cone over $\partial_\infty \mathcal{C}(T,b)$, and $\partial_\infty \mathcal{C}(T,b)$ is the {\it logarithmic limit set} of ${\rm Log}\vert \Theta_1\vert(G_{hyp}(T,b))$ in the sense of \cite{Berg}.\\
By means of Culler-Shalen theory (\cite{CS,Shalen}) such spaces play an important role in the detection of boundary slopes of incompressible surfaces in cusped manifolds $M$. These slopes are projective classes $[p\mu_K + q\lambda_K]\in \mathbb{P}^1_\mr(H_1(\partial \bar M;\mr))$. In particular, using also the Neumann-Zagier combinatorics \cite{NZ}, Yoshida showed that for $M=S^3\setminus 4_1$ the logarithmic limit set $\partial_\infty \mathcal{C}(T,b)$ determines the slopes (see \cite{Yo}, p. 160):
\begin{equation}\label{bslopes}
{\rm slope}(L_1) = {\rm slope}(L_3) = [4\mu_K - \lambda_K]\ ,\ {\rm slope}(L_2) = {\rm slope}(L_4) = [4\mu_K + \lambda_K].
\end{equation}
The map $\tilde{hol}\colon G_{hyp}(T,b))\ra \tilde{X}_{hyp}(M)$ of Section \ref{GhypN} identifies these boundary slopes with those detected by the Newton polygon of the factor of the $PSL(2,\mc)$ $A$-polynomial of $M$ corresponding to the curve $\tilde{X}_{hyp}(M)$ \cite{CCGLS,BD,Champ}. We refer to \cite{Tillmann0,Tillmann,Kabaya} for further developments and applications of this method to arbitrary cusped manifolds.\\
(2) As in Lemma \ref{lemYo} one can define a limit space $\mathcal{C}(\tilde X_{hyp}(M))\times (\mathbb{S}^1)^2$ of the tower of curves $\{{}_NX_{hyp}(M)\}_N$; the $(\mathbb{S}^1)^2$ factor is realized by the pairs $(e^{i\arg(\kappa^\infty(\lambda_K))},e^{\arg(\kappa^\infty(\mu_K))})$. As $N\to +\infty$, the maps $\widetilde{hol}\colon G_{hyp,N}(T,b) \rightarrow {}_NX_{hyp}(M)$ of Section \ref{GhypN} yields a map $\mathcal{C}(T,b)\times \mathcal{A}(T,b)\longrightarrow \mathcal{C}(\tilde X_{hyp}(M))\times (\mathbb{S}^1)^2$, where $\mathcal{C}(\tilde X_{hyp}(M))$ is the tropical curve assigned to the amoeba of $\tilde X_{hyp}(M)$, as in (1) above.} 
\end{remark}

\section{The QHI} \label{sec:QHI}

 \subsection{Functions in the QHI state sums}\label{FUNCNOT}  Consider the rational functions $\omega_N(.,.|n)$, $n\in \mn$, defined on the curve $\mathcal{F}_N :=\{(x,y)\in \mc^2\vert \ x^N+y^N=1\}$ by
\begin{equation}\label{functomegadef}
\omega_N(x,y|0)=1\ ,\ \omega_N(x,y|n)=\prod_{j=1}^n \frac{y}{1-x\zeta^{j}}\ {\rm if} \ n\geq 1.
\end{equation}
Since $(1-x)(1-x\zeta)\ldots (1-x\zeta^{N-1}) = 1-x^N=y^N$, we have
$$\omega_N(x,y|N+n) = \omega_N(x,y|n).$$
These functions are closely related to Faddeev's \cite{Fad,FKV} {\it non compact quantum dilogarithm} $\Phi_\rmb$, and therefore to the classical Euler dilogarithm ${\rm Li_2}$. The function $\Phi_\rmb$ is defined as follows; it can be regarded as a kind of double Barnes' sine function (\cite{GR}). Let $\rmb\in \mc$ be such that $\mathfrak{R}(\rmb)\ne 0$. Define
$$\sqrt{\hbar} := \frac{1}{\rmb+\rmb^{-1}}.$$
Denote by $\mr +i0$ the real axis directed from $-\infty$ to $\infty$, slightly deformed near $0$ so as to pass above $0$. Then $\Phi_\rmb$ is the holomorphic function defined in the strip $\vert \mathfrak{I}(z)\vert < \vert \mathfrak{I}(i/2\sqrt{\hbar})\vert$ by the integral formula 
\begin{equation}\label{intrepS}
\Phi_{{\rmb}}(z) =\exp\left(\frac{1}{4} \int_{\mr +i0} \frac{e^{-2izw}\ dw}{w\sinh(\rmb w)\sinh(\rmb^{-1} w)}\right).
\end{equation}
Denote by $\Log(\Phi_{{\rmb}}(z))$ the integral on the RHS. It satisfies the functional equation
\begin{equation}\label{eqfuncphibLog}
\Log\left(\Phi_{{\rmb}}\left(z-i\frac{\rmb^{\pm 1}}{2}\right)\right) = \Log\left(\Phi_{{\rmb}}\left(z+i\frac{\rmb^{\pm 1}}{2}\right)\right) + \Log(1+e^{2\pi\rmb^{\pm 1}z}) .
\end{equation}
Therefore $\Phi_{{\rmb}}$ can be extended to a meromorphic function of $z\in \mc$, satisfying the functional equation
\begin{equation}\label{eqfuncphib}
\Phi_{{\rmb}}\left(z-i\frac{\rmb^{\pm 1}}{2}\right) = (1+e^{2\pi\rmb^{\pm 1}z}) \Phi_{{\rmb}}\left(z+i\frac{\rmb^{\pm 1}}{2}\right).
\end{equation}
The poles and zeros of $\Phi_\rmb$ are the points $\textstyle z=\frac{i}{2\sqrt{\hbar}} +mi\rmb+ni\rmb^{-1}$ and $\textstyle z=-(\frac{i}{2\sqrt{\hbar}}+mi\rmb+ni\rmb^{-1})$ respectively, where $m,n\in \mn$.

We will use the following version of $\Phi_\rmb$. Let $\textstyle \rmb = \frac{1}{\sqrt{N}}$, and set 
\begin{equation}\label{defShat}
\textstyle \hat{S}_N(z) := \Phi_{\frac{1}{\sqrt{N}}}\left(\frac{\sqrt{N}}{2\pi}\left(z-i(\pi-\frac{\pi}{N})\right)\right).
\end{equation}
Denote by $\Log(\hat{S}_N(z))$ the integral obtained from the one in \eqref{intrepS} by the change of variable in \eqref{defShat}. Then $\Log(\hat{S}_N(z))$ is a holomorphic function of $z$ in the strip $\textstyle -\frac{2\pi}{N}<\mathfrak{I}(z) < 2\pi$, and it satisfies the functional equation
\begin{equation}\label{eqfunchatSLog} 
\textstyle \Log(\hat{S}_N(z)) = \Log(\hat{S}_N(z+\frac{2i\pi}{N})) + \Log(1-e^{z+\frac{2i\pi}{N}}).
\end{equation}
\noindent Therefore it can be extended to an analytic function on $\mc$ with cuts $\textstyle [0,+\infty[+i(2\pi + \frac{2\pi p}{N})$, $p\in\mn$, and $\textstyle [0,+\infty[-i(l+1)\frac{2\pi}{N}$, $l\in\mn$. All this implies that $\hat{S}_N$ is holomorphic and without zeros in the strip $\textstyle -\frac{2\pi}{N}<\mathfrak{I}(z) < 2\pi$, it satisfies the functional equation
\begin{equation}\label{eqfunchatS} 
\textstyle\hat S_{N}(z) = \hat S_{N}(z+\frac{2i\pi}{N})(1-e^{z+\frac{2i\pi}{N}}),
\end{equation}
and it can be extended to a meromorphic function on $\mc$ with sets of poles and simple poles
\begin{equation}\label{PP1}
\textstyle P=\{2i\pi+i\frac{2p\pi}{N},\ p\in \mn\}\ ,\ P_1=\{2i\pi+i\frac{2p\pi}{N},\ p\in \{0,\ldots,N-1\}\}
\end{equation}
respectively, and zeroes
$$\textstyle Z=\{-(l+1)i\frac{2\pi}{N},\ l\in \mn\}.$$
Moreover the functional equation \eqref{eqfunchatS} yields
\begin{equation}\label{relomS0}
\omega_N(e^{z},y| n) = \prod_{j=1}^{n} \frac{y}{1-e^{z}\zeta^j} = y^n\ \frac{\hat S_N(z+2n i\pi/N)}{\hat S_N(z)}
\end{equation}
for every $y,z\in \mc$ such that $y^N = 1-e^{Nz}$. 
The function $\Phi_\rmb$ has a number of remarkable properties, see \cite{FKV,AK}. In particular, its ''semi-classical limit'', as $\rmb \ra 0^+$, recovers the Euler dilogarithm function (see \cite{Z} and Lemma \ref{estimatehatS} below),
\begin{equation}\label{Li2def}{\rm Li_2}(x) = -\int_0^x\frac{\Log(1-t)}{t}dt,\quad x\in \mc\setminus [1,+\infty[.
\end{equation}
Basically, the connection between the two functions can be guessed by comparing the formula \eqref{intrepS} with the following integral representation. It is well-known, but we recall a proof for it is short and illuminating.
\begin{lem}\label{Li2intrep} For all $y\in \mr+i(-\pi,\pi)$ we have $\frac{1}{2\pi i}{\rm Li_2}(-e^y) = \frac{1}{4} \int_{\mr +i0} \frac{e^{-i\frac{yv}{\pi}}\ dv}{v^2\sinh(v)}.$
\end{lem}
\proof The integral above is a holomorphic function of $y$. Taking the derivative twice we get
$$\partial_{y}^2 \left( \int_{\mr +i0} \frac{e^{-i\frac{yv}{\pi}}\ dv}{v^2\sinh(v)}\right)  = -\frac{1}{\pi^2} \int_{\mr +i0} \frac{e^{-i\frac{yv}{\pi}}\ dv}{\sinh(v)} = -\frac{1}{\pi^2 (1+e^{y})}\ \biggl( \int_{\mr +i0} - \int_{\mr +i0+\pi i} \biggr)\ \frac{e^{-i\frac{yv}{\pi}}\ dv}{\sinh(v)}$$
where, using a simple change of variable, we note that
$$\int_{\mr +i0+\pi i} \frac{e^{-i\frac{yv}{\pi}}\ dv}{\sinh(v)} = -e^{y}\ \int_{\mr +i0} \frac{e^{-i\frac{yv}{\pi}}\ dv}{\sinh(v)}.$$
Now, let $R>0$ and consider the contour $\Gamma_R$ formed by the concatenation of the oriented segments $[-R,R]+i0$, $R+i\pi[0,1]$, $-([-R,R]+i0)+i\pi$ and $-R-i\pi [0,1]$. We have
$$\lim_{R \rightarrow +\infty} \int_{\pm R+i\pi[0,1]} \frac{e^{-i\frac{yv}{\pi}}\ dv}{\sinh(v)} = \lim_{R \rightarrow +\infty} \int_0^1 \frac{e^{-i\frac{y}{\pi}(\pm R + \pi it)}}{\sinh(\pm R + \pi it)} \ dt = 0$$
and therefore
\begin{align*}\biggl( \int_{\mr +i0} - \int_{\mr +i0+\pi i} \biggr)\ \frac{e^{-i\frac{yv}{\pi}}\ dv}{\sinh(v)} & = \lim_{R \rightarrow +\infty} \biggl( \int_{\Gamma_R} - \int_{R+i\pi[0,1]} - \int_{-R-i\pi[0,1]}\biggr) \frac{e^{-i\frac{yv}{\pi}}\ dv}{\sinh(v)} \\ & = \lim_{R \rightarrow +\infty} \int_{\Gamma_R} \frac{e^{-i\frac{yv}{\pi}}\ dv}{\sinh(v)}. 
\end{align*}
This integral is constant in $R$, and equal to $\textstyle 2\pi i .{\rm Res}\left( \frac{e^{-i\frac{yv}{\pi}}}{\sinh(v)}, v=i\pi\right) =-2\pi i e^y$ by the residue theorem. We deduce
$$\partial_{y}^2 \left( \int_{\mr +i0} \frac{e^{-i\frac{yv}{\pi}}\ dv}{v^2\sinh(v)}\right) =\frac{2ie^y}{\pi (1+e^{y})}.$$
Finally, we can integrate twice along a path running from $-\infty$ along the real axis up to $\mathfrak{R}(y)$, and then along a straight segment up to $y\in \mr+i(-\pi,\pi)$, since the integrands are continuous in this strip:
\begin{align*}
\partial_{y} \left( \int_{\mr +i0} \frac{e^{-i\frac{yv}{\pi}}\ dv}{v^2\sinh(v)}\right) & =\frac{2i}{\pi} \int_{-\infty}^y \frac{e^s}{1+e^{s}}ds =\frac{2i}{\pi} \Log(1+e^y)\\ \int_{\mr +i0} \frac{e^{-i\frac{yv}{\pi}}\ dv}{v^2\sinh(v)}& = \frac{2i}{\pi} \int_{-\infty}^y \Log(1+e^s) ds = -\frac{2i}{\pi} {\rm Li}_2(-e^y).
\end{align*}
The result follows.\cvd
\medskip

\subsection{Asymptotics of $\hat{S}_N(z)$ and $g_N(z)$}\label{sec:asySNfev26}

Let us start with a straightforward lemma.

\begin{lem}\label{lem:bound:exp:1-exp}
For all $\delta\in ]0,\pi[$, for all $z \in ]-\infty,M] + i([\delta, 2\pi-\delta]+2\pi k)$, $k\in \mz$, we have:
$$
\left \vert
\dfrac{e^z}{1-e^z}
\right \vert \leq \dfrac{e^{M/2}}{\delta\sqrt{1- \frac{\pi^2}{24}}}.
$$
\end{lem}

\begin{proof} It is enough to do the case $k=0$. For $\delta\in ]0,\pi[$ and $z = x+iy \in ]-\infty,M] + i[\delta, 2\pi-\delta]$, we have
 \begin{align*}
 \left\vert \dfrac{e^z}{1-e^z}
 \right \vert^2 = 
 \dfrac{e^{2x}}{|1-e^z|^2} = 
\dfrac{e^{2x}}{1+e^{2x}-2e^x \cos(y)} & = 
\dfrac{e^{x}}{2\cosh(x)-2\cos(y)} \\ & \leq 
 \dfrac{e^M}{2\cosh(x)-2\cos(y)}.
 \end{align*}
For $y \in [\delta,2\pi-\delta]$ fixed, 
$\dfrac{1}{2\cosh(x)-2\cos(y)}$ is maximal at $x=0$, where it is equal to
$$\dfrac{1}{2-2\cos(y)}
= \dfrac{1}{2-2+4\sin(y/2)^2}
= \dfrac{1}{4 \sin(y/2)^2} .
$$
Moreover, for $y\in [\delta,2\pi-\delta]$ we have $\sin(y/2)\geq \sin(\delta/2)$, and since $0<\delta/2 <\pi/2$, a Taylor-Lagrange expansion with remainder at order $4$ gives $\textstyle \sin(\delta/2) \geq \delta/2 - \frac{1}{6}(\delta/2)^3>\delta/2\left (1-\frac{\pi^2}{24}\right )>0$. Therefore
$$\dfrac{1}{4 \sin(y/2)^2} \leq \dfrac{1}{4(\delta/2 - \frac{1}{6}(\delta/2)^3)^2} \leq \dfrac{1}{\delta^2(1- \frac{\pi^2}{24})}.$$
The lemma follows.
\end{proof}
 
We will use the following {\it uniform} estimate of the relation between ${\rm Li}_2$ and the semi-classical limit of the function $\hat S_N$.
\begin{lem}\label{estimatehatS} There exists constants $B',B''>0$ such that for all  $\delta\in ]0,\pi[$, $M \in \R$,  large enough $N\geq 3$ odd, and $z\in ]-\infty,M] +i[\textstyle \delta-\frac{\pi}{N}, 2\pi-\delta-\frac{\pi}{N}]$ we have
\begin{equation}\label{equalitylemmaasyoct25}\Log(\hat{S}_N(z)) = \frac{N}{2i\pi}{\rm Li_2}(e^{z}) -\frac{1}{2} \Log(1-e^{z}) + \frac{1}{N}\Psi_N(z),
\end{equation}
where $\Psi_N$ is a holomorphic function of $z$, with formula given by \eqref{defXiNjanv26}-\eqref{defPsiNjanv26}, satisfying
\begin{equation}\label{UboundPsi}
\vert \Psi_N(z)\vert \leq B_\delta+\dfrac{\pi e^{M/2}}{2\delta\sqrt{1- \frac{\pi^2}{24}}},\ \mathrm{with}\ B_\delta := B'/\delta+B''.
\end{equation}  
Moreover, setting $\mathrm{c}_M:=\sqrt{1+e^M}$, we have
 $$\left\vert \mathfrak{R}\left(\Log(\hat{S}_N(z)) -\frac{N}{2i\pi}{\rm Li_2}(e^{z})\right) \right\vert \leq \Log(\mathrm{c}_M) + \frac{B_\delta}{N}$$
 or equivalently
$$ \dfrac{1}{\mathrm{c}_M}\exp\left(-\frac{B_\delta}{N}\right) \leq
\left\vert \hat{S}_N(z)\exp\left( -\frac{N}{2i\pi}{\rm Li_2}(e^z)\right) \right\vert \leq \mathrm{c}_M\exp\left(\frac{B_\delta}{N}\right).$$
\end{lem}
\begin{remark}\label{rem:largeN} {\rm (i) The condition ``large enough $N\geq 3$'' is in place to ensure that $\textstyle \delta-\frac{\pi}{N}>0$, so that the half-strip $\textstyle ]-\infty,M] +i[\textstyle \delta-\frac{\pi}{N}, 2\pi-\delta-\frac{\pi}{N}]$ avoids cuts in the domain \eqref{holdomPsiN}. Thus, if $\mathfrak{R}(z)<0$ this condition is unnecessary and $\textstyle \delta-\frac{\pi}{N}\leq 0$ is admissible.\\
(ii) The definition of $\Psi_N$ from $\Xi_N$ (see \eqref{defXiNjanv26}-\eqref{defPsiNjanv26}) shows that the identity \eqref{equalitylemmaasyoct25} holds true for all $z$ in the strip $\mr +i]\textstyle -\frac{\pi}{N}, 2\pi-\frac{\pi}{N}[$.\\
(iii) The lemma generalizes immediately to the situation where one considers distinct constants $\delta^-, \delta^+\in ]0,\pi[$  instead of $\delta$, and takes $z\in ]-\infty,M] +i[\textstyle \delta^--\frac{\pi}{N}, 2\pi-\delta^+-\frac{\pi}{N}]$. Then one has \eqref{UboundPsi} with $\delta:= \min(\delta^-,\delta^+)$. This may be useful when one is willing sharp uniform bounds on specific sets.\\ 
(iv)  We can give explicit numerical values for the constants $B',B''>0$ in Lemma \ref{estimatehatS}, following the computations in \cite[Lemma 7.13]{BGP-N} (see the formula of $B_\delta$ on page 367 of that paper).
}\end{remark}
\proof It follows from Lemmas 7.13 and 7.14 in \cite{BGP-N} by a change of variable and a very small adaptation. Let us give a few details. First we consider the inequalities, since for them the adaptation is simpler to describe.

Lemmas 7.13 and 7.14 in \cite{BGP-N} deal with $\textstyle \Phi_{\rmb}(\frac{y}{2\pi \rmb})$, with the assumptions $\mathfrak{I}(y)\in [\delta,\pi-\delta]$ and $\mathfrak{I}(y)\in [-\pi+\delta,-\delta]$ respectively. From the proof of these lemmas, it is easy to see that these assumptions, keeping the same conclusions, can be extended to $\mathfrak{I}(y)\in [0,\pi-\delta]$ and $\mathfrak{I}(y)\in [-\pi+\delta,0]$ respectively. By setting $\rmb:= 1/\sqrt{N}$ and $y=z-i\pi+i\pi/N$, these conclusions imply immediately that, given $\delta>0$, $N$, and $\textstyle z\in \mr + i[\delta-\frac{\pi}{N}, 2\pi-\frac{\pi}{N}-\delta]$, there exists a constant $B_\delta>0$ as in the statement such that
 $$\left\vert \mathfrak{R}\left(\Log(\hat{S}_N(z)) -\frac{N}{2i\pi}{\rm Li_2}(e^{z+\frac{i\pi}{N}})\right) \right\vert \leq \frac{B_\delta}{N}.$$
 More precisely, denoting 
$$\epsilon(z):= \dfrac{1}{z^2}\left (\dfrac{z}{\sinh(z)}-1\right )$$ (the Taylor remainder at $0$ of order $2$ of the function $z \mapsto z/\sinh(z)$), and $\Omega\subset \C$ a contour equal to $\R$ except near $0$ where it avoids $0$ upwards while staying in a disk $D(0,R)$ with center $0$ and radius $R\in ]0,\pi[$, we have
\begin{equation}\label{eq:BAGPN:Xi}
\Log(\hat{S}_N(z)) -\frac{N}{2i\pi}{\rm Li_2}(e^{z+\frac{i\pi}{N}})
= \dfrac{1}{N} \Xi_N(z)
\end{equation}
for $\textstyle z\in \mr + i[\delta-\frac{\pi}{N}, 2\pi-\frac{\pi}{N}-\delta]$, where 
\begin{equation}\label{defXiNjanv26}
\Xi_N(z) :=
 \int_{v\in \Omega} \epsilon\left (\frac{v}{N}\right ) 
\dfrac{\exp(-i\frac{z v}{\pi}-v+\frac{v}{N})}{4 \sinh(v)}dv
\end{equation}
is holomorphic in the open strip $\mr +i]\textstyle -\frac{\pi}{N}, 2\pi-\frac{\pi}{N}[$ and satisfies $\vert \Xi_N(z)\vert \leq B_\delta$ on the closed substrip $\mr +i[\textstyle \delta-\frac{\pi}{N}, 2\pi-\delta-\frac{\pi}{N}]$. (We stress that in Lemma 7.13 of  \cite{BGP-N} the stated upper bound is $\vert \mathfrak{R}(\Xi_N(z))\vert \leq B_\delta$, where the constant $B_\delta:= B'/\delta+B''$ is described explicitly on page 367, but the conclusion is deduced from the inequality $\vert \mathfrak{R}(\Xi_N(z))\vert \leq\vert \Xi_N(z)\vert$ and an upper bound for $\vert \Xi_N(z)\vert$).

We want to replace the term $\textstyle \frac{N}{2i\pi}{\rm Li_2}(e^{z+\frac{i\pi}{N}})$ with $\textstyle \frac{N}{2i\pi}{\rm Li_2}(e^{z})$. For this, we will use a Taylor expansion.

Recall that we take $z\in ]-\infty,M] +i[\textstyle \delta-\frac{\pi}{N}, 2\pi-\delta-\frac{\pi}{N}]$, with $N$ large enough so that $\textstyle \delta-\frac{\pi}{N}>0$. In particular, the arguments of the logarithms and dilogarithms in the following formulas will avoid the cuts of these functions. Since
\begin{equation}\label{Taylor0}
{\rm Li_2}(e^{z+\frac{i\pi}{N}}) - {\rm Li_2}(e^{z}) = - \frac{i\pi}{N} \int_0^1 \Log(1-e^{z+\frac{i\pi}{N}t}) dt,
\end{equation}
we have
$$ \mathfrak{R}\left(\frac{N}{2i\pi}{\rm Li_2}(e^{z+\frac{i\pi}{N}})
- \frac{N}{2i\pi}{\rm Li_2}(e^{z})\right) = - \frac{1}{2} \int_0^1 \Log\left\vert 1-e^{z+\frac{i\pi}{N}t}\right\vert dt
= -  \int_0^1 \Log\left\vert 1-e^{z+\frac{i\pi}{N}t}\right\vert^{\frac{1}{2}} dt. $$

Now $\vert 1-e^z\vert^{\frac{1}{2}}\leq \sqrt{1+ \vert e^z\vert} \leq \sqrt{1+e^M}$ for all $\textstyle z\in ]-\infty,M] +i[\delta-\frac{\pi}{N}, 2\pi-\delta]$. Define $\mathrm{c}_M:=\sqrt{1+e^M}$. We get, for every $z\in ]-\infty,M] +i[\textstyle \delta-\frac{\pi}{N}, 2\pi-\delta-\frac{\pi}{N}]$:

$$\left\vert \mathfrak{R}\left(\Log(\hat{S}_N(z)) -\frac{N}{2i\pi}{\rm Li_2}(e^{z})\right) \right\vert \leq \Log(\mathrm{c}_M) + \frac{B_\delta}{N}.$$
The pair of inequalities in the statement immediately follows.

Let us now prove the equality \eqref{equalitylemmaasyoct25} and the upper bound \eqref{UboundPsi}.
From (\ref{eq:BAGPN:Xi}) and (\ref{Taylor0}), we obtain
\begin{align*}
\Log(\hat{S}_N(z)) 
&= \frac{N}{2i\pi}{\rm Li_2}(e^{z+\frac{i\pi}{N}})
+ \dfrac{1}{N} \Xi_N(z)\\
&= \frac{N}{2i\pi}{\rm Li_2}(e^{z})  - \frac{1}{2} \int_0^1 \Log(1-e^{z+\frac{i\pi}{N}t}) dt
+ \dfrac{1}{N} \Xi_N(z).
\end{align*}
Another Taylor expansion gives us
\begin{equation}\label{eq:taylor:log:exp}
\Log(1-e^{z+\frac{i\pi}{N}t}) = \Log(1-e^{z}) - \frac{1}{N}\mathcal{E}_N(z,t),
\end{equation}
where $\mathcal{E}_N(z,t):= (i\pi t) \int_0^1 \dfrac{e^{z+\frac{i\pi}{N}ts}}{1-e^{z+\frac{i\pi}{N}ts}}ds$. Hence we obtain
$$
\Log(\hat{S}_N(z)) 
= \frac{N}{2i\pi}{\rm Li_2}(e^{z})  - \frac{1}{2}  \Log(1-e^{z}) 
+ \frac{1}{2N}\int_0^1\mathcal{E}_N(z,t)dt
 + \dfrac{1}{N} \Xi_N(z).
$$
This gives (\ref{equalitylemmaasyoct25}), by setting
\begin{equation}\label{defPsiNjanv26}
\Psi_N(z):= \Xi_N(z)+\dfrac{1}{2}\int_0^1\mathcal{E}_N(z,t)dt
=\Xi_N(z)+\dfrac{i\pi}{2}\int_{t=0}^1 t \int_{s=0}^1\dfrac{e^{z+\frac{i\pi}{N}ts}}{1-e^{z+\frac{i\pi}{N}ts}}dsdt.
\end{equation}
Consider the strip $]-\infty,M] +i[\textstyle \delta-\frac{\pi}{N}, 2\pi-\delta-\frac{\pi}{N}]$. We know that $|\Xi_N(z)|\leq B_\delta$, and from Lemma \ref{lem:bound:exp:1-exp} the second term gives:
$$\left |
\dfrac{i\pi}{2}\int_{t=0}^1 t \int_{s=0}^1\dfrac{e^{z+\frac{i\pi}{N}ts}}{1-e^{z+\frac{i\pi}{N}ts}}dsdt
\right | \leq \dfrac{\pi}{2 } \int_{s=0}^1\left |
 \dfrac{e^{z+\frac{i\pi}{N}ts}}{1-e^{z+\frac{i\pi}{N}ts}} \right | ds
 \leq \dfrac{\pi e^{M/2}}{2\delta\sqrt{1- \frac{\pi^2}{24}}}.$$
By the triangle inequality, \eqref{UboundPsi} follows at once.\cvd
\medskip

Let us draw the following important though easy consequence of Lemma \ref{estimatehatS} and the functional equation \eqref{eqfunchatSLog}. It will be used in Lemma \ref{lem:reformintqjui25}.

By the domains where $\Log(\hat{S}_N(z))$ (deduced after \eqref{eqfunchatSLog}) and ${\rm Li_2}(e^{z})$  and $\Log(1-e^{z})$ are holomorphic functions (here, $\textstyle \mc\setminus \left(\cup_{k\in \mz} \left(2i\pi k +[0,+\infty[\right)\right)$), the identity (\ref{equalitylemmaasyoct25}) immediately shows that $\Psi_N$ has a holomorphic extension on the domain 
\begin{equation}\label{holdomPsiN}
\mc \setminus \left(\bigcup_{p\in \mn} \left([0,+\infty[+i\left(2\pi + \frac{2\pi p}{N}\right)\right) \bigcup \left(\bigcup_{l\in \mn}\left( [0,+\infty[-i(l+1)\frac{2\pi}{N}\right)\right) \bigcup [0,+\infty[ \right).
\end{equation}
In particular $\Psi_N$ is holomorphic on the domain (see Figure \ref{fig:Uet+})
$$U:=\left \lbrace z\in \mc\ \vert\ \mathfrak{R}(z) <0 \right \rbrace \bigcup \left \lbrace z\in \mc\ \vert \ 0 < \mathfrak{I}(z) < 2\pi\right\rbrace.$$
In the following lemma, for every $M\in \mr$, $\alpha\in ]0,1[$, $k\in \mz_{\leq 0}$, $K>2\pi$, and $c>0$ we consider the sets $U_{M,\alpha,k,K,c}\subset \mc$ formed by the points $z$ such that (see Figure \ref{fig:Uet+}):
\begin{itemize}
\item $\mathfrak{R}(z)\leq M$ and $\textstyle (\alpha-1)\frac{\pi}{N}< \mathfrak{I}(z)< 2\pi-(\alpha+1)\frac{\pi}{N}$, or 
\item $\mathfrak{R}(z)<0$ and $\textstyle 2\pi k + (\alpha-1)\frac{\pi}{N}\leq \mathfrak{I}(z)\leq (\alpha-1)\frac{\pi}{N}$, or 
\item $\textstyle \mathfrak{R}(z)< -c$ and $\textstyle 2\pi-(\alpha+1)\frac{\pi}{N}\leq \mathfrak{I}(z)\leq K$.
\end{itemize}
\medskip

\begin{figure}[!h]
\begin{tikzpicture}
\draw[->] (-6,0)--(6,0);
\draw[->] (0,-6)--(0,6);

\draw (0,0-.03) node {\Huge $\cdot$};
\draw (0+.3,0+.3) node {$0$};
\draw (0,2-.03) node {\Huge $\cdot$};
\draw (0+.5,2-.03+.3) node {$2i\pi$};

\draw[color=blue,thick] (0,6)--(0,2)--(6,2);
\draw[color=blue,thick] (0,-6)--(0,0)--(6,0);

	\draw[white,fill=blue!30,opacity=0.5, ] (-6,6)--(0,6)--(0,2)--(6,2)--(6,0)--(0,0)--(0,-6)--(-6,-6);
	
\draw (-2,5) node {\Huge \color{blue} $U$};

	\draw[red ,fill=red!50,opacity=0.5, ] (-6,3.5)--(-2,3.5)--(-2,1.7)--(4,1.65)--(4,-0.15)--(0,-0.15)--(0,-4-0.15)--(-6,-4.15);
	
\draw[color=red] (-2,3.5-.03) node {\Huge $\cdot$};
\draw[color=black] (-2,3.5+.3) node {\scriptsize $-c+i K$};

\draw[color=red] (-2,1.7-.03) node {\Huge $\cdot$};
\draw[color=black] (-3.8,1.6) node { \scriptsize $-c+i \left (2\pi-(\alpha+1)\frac{\pi}{N}\right )$};

\draw[color=red] (4,1.7-.07) node {\Huge $\cdot$};
\draw[color=black] (5.7,1.6) node { \scriptsize $M+i \left (2\pi-(\alpha+1)\frac{\pi}{N}\right )$};

\draw[color=red] (4,-.19) node {\Huge $\cdot$};
\draw[color=black] (5.2,-.5) node { \scriptsize $M+i (\alpha-1)\frac{\pi}{N}$};

\draw[color=red] (0,-.19) node {\Huge $\cdot$};
\draw[color=black] (.9,-.55) node { \scriptsize $i(\alpha-1)\frac{\pi}{N}$};

\draw[color=red] (0,-4-.19) node {\Huge $\cdot$};
\draw[color=black] (1.5,-4-.2) node { \scriptsize $i\left ((\alpha-1)\frac{\pi}{N}+2k\pi\right )$};

\draw (-3,-2) node {\Huge \color{red!40!black} $U_{M,\alpha,k,K,c}$};

\end{tikzpicture}
\caption{The domain $U$ and the sets $U_{M,\alpha,k,K,c}$.}\label{fig:Uet+}
\end{figure}
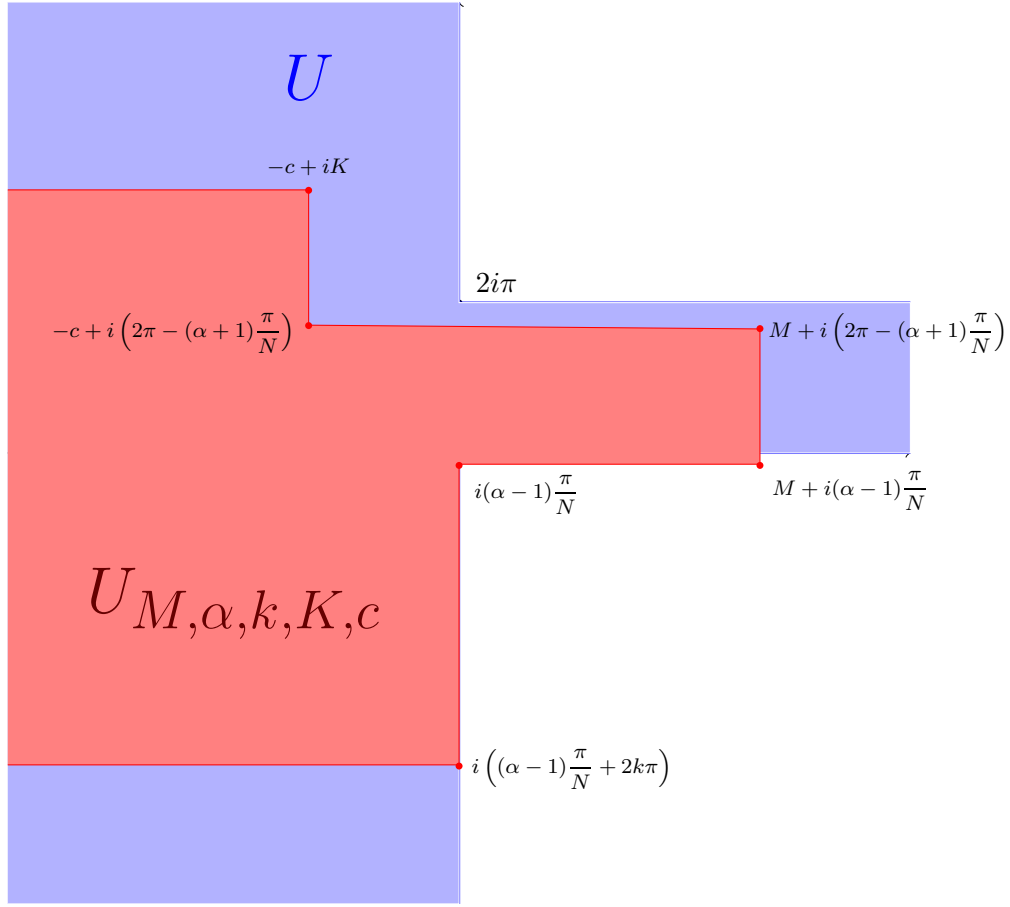
\begin{lem}\label{extboundPsiNjanv26} (i) The sequence of functions $\Psi_N$ converges simply to a holomorphic function on the strip $0 < \mathfrak{I}(z) < 2\pi$.\\
(ii) The function $\textstyle \vert \exp(\frac{1}{N} \Psi_N)\vert$ is bounded from above by a constant independent of $N$ on each set $U_{M,\alpha,k,K,c}$.\\
(iii) The sequence of functions $\textstyle \exp(\frac{1}{N} \Psi_N)$ converges uniformly to the constant function with value $1$ on  compact subsets of $U$.
\end{lem}
\begin{remark}{\rm (i) The condition $\textstyle \mathfrak{R}(z)< -c$ for some $c>0$ when $\textstyle 2\pi-(\alpha+1)\frac{\pi}{N}\leq \mathfrak{I}(z)\leq K$ is in place to ensure that the half-space stays at distance from the singular points $2i\pi+ 2i\pi p/N$, $p\in \mn$, of $\Psi_N$.\\ 
(ii) The proof below provides an explicit upper bound of $\textstyle \vert \exp(\frac{1}{N} \Psi_N)\vert$ on $U_{M,\alpha,k,K,c}$.} \end{remark} 

\begin{proof} (i) Any point of $\textstyle \mr +i]0, 2\pi[$ is contained in $\textstyle \mr +i]-\frac{\pi}{N}, 2\pi-\frac{\pi}{N}[$  for $N$ large enough, where $\Psi_N$ has the integral formula given by \eqref{defPsiNjanv26} and \eqref{defXiNjanv26}. By using Lemma \ref{lem:bound:exp:1-exp} and Lebesgue's dominated convergence theorem it is clear that the second summand in the formula \eqref{defPsiNjanv26} converges simply on $\textstyle \mr +i]0, 2\pi[$. Thus it remains to prove it also for $\Xi_N$. 

Recall the function $\epsilon$ in \eqref{defXiNjanv26}. It is easy to see that $\epsilon$ is holomorphic on $\mr\cup D(0,R)$, that $\textstyle \lim_{x\to 0}\epsilon(x) = -\frac{1}{6}$, and that there is a scalar $\sigma_R>0$ such that $\vert \epsilon(x)\vert \leq \sigma_R$ for all $x\in  \Omega \cup D(0,R)$ (see \cite[pp. 364-365]{BGP-N} for details). Let us split $\Xi_N(z)$ as
\begin{multline} \Xi_N(z) = \int_{v\in \Omega\cap D(0,R) } \epsilon\left (\frac{v}{N}\right ) \dfrac{\exp(-i\frac{z v}{\pi}-v+\frac{v}{N})}{4 \sinh(v)}dv \\ + \int_{v\in ]-\infty,-R]} \epsilon\left (\frac{v}{N}\right ) \dfrac{\exp(-i\frac{z v}{\pi}-v+\frac{v}{N})}{4 \sinh(v)}dv + \int_{v\in [R,+\infty[} \epsilon\left (\frac{v}{N}\right ) \dfrac{\exp(-i\frac{z v}{\pi}-v+\frac{v}{N})}{4 \sinh(v)}dv.
\end{multline}
Note that for all $N$ the integrand is a holomorphic function of $z$ and $v$. It has modulus less than $\sigma_R  \textstyle \vert \frac{\exp(-i\frac{z v}{\pi}-v+R)}{4 \sinh(v)} \vert$ for $v\in \Omega\cap D(0,R)$, where this function is continuous and hence integrable. Now set $z=x+iy$, so  that $0<y<2\pi$. For a real $v\ne 0$ we have 
\begin{align*}\left\vert \epsilon\left (\frac{v}{N}\right ) \dfrac{\exp(-i\frac{z v}{\pi}-v+\frac{v}{N})}{4 \sinh(v)} \right\vert & \leq \sigma_R  \left \vert \frac{\exp(-i\frac{z v}{\pi}-v+\frac{v}{N})}{4 \sinh(v)} \right\vert \\ & = \sigma_R  \frac{\exp(v(-1 + \frac{y}{\pi}+\frac{1}{N}))}{4 \sinh(v)}.
\end{align*}
Let $c^-, c^+$ be positive real numbers such that $\textstyle 0<c^-<\frac{y}{\pi}<c^+<2$, and assume that $N$ is large enough so that $\textstyle \frac{y}{\pi}+\frac{1}{N}<c^+$.  Then the RHS is $\leq \textstyle \sigma_R  \frac{\exp(v(c^+ -1))}{4 \sinh(v)}$ on $[+R,+\infty[$ and $\leq \textstyle \sigma_R  \frac{\exp(v(c^- -1))}{4 \sinh(v)}$ on $[-\infty,-R[$, where these functions are integrable on these intervals. Then we can apply Lebesgue's dominated convergence theorem to the three integrals. By summing the results we get
$$\lim_{N\rightarrow +\infty}  \Xi_N(z) = -\frac{1}{6} \int_{v\in \Omega} \dfrac{\exp(-i\frac{z v}{\pi}-v)}{4 \sinh(v)}dv,$$
The arguments above show also that the integral on the RHS is a holomorphic function of $z$ in the strip $\mr+i]0, 2\pi[$, which concludes the proof of (i).\\ 
(ii) First consider $\textstyle \vert \exp(\frac{1}{N} \Psi_N)\vert$ on $\textstyle ]-\infty,M] +i]-\frac{\pi}{N}, 2\pi-\frac{\pi}{N}[$. Let $\alpha>0$ be fixed, and apply \eqref{UboundPsi} with $\delta$ of the form $\delta_N:=\textstyle \alpha\frac{\pi}{N}$. For any $z\in \textstyle ]-\infty,M] +i[\delta_N-\frac{\pi}{N}, 2\pi-\delta_N-\frac{\pi}{N}]$, we find 
\begin{equation}\label{boundPsiNseqstripjanv26}
\frac{1}{N}\vert \Psi_N(z)\vert\leq \frac{B'}{\alpha\pi} +\dfrac{\pi e^{M/2}}{2\alpha\pi\sqrt{1- \frac{\pi^2}{24}}}+\frac{B''}{N},
\end{equation}
which is bounded from above by a constant independent of $N$. This proves the claim on the strip $\textstyle (\alpha-1)\frac{\pi}{N}< \mathfrak{I}(z)< 2\pi-(\alpha+1)\frac{\pi}{N}$.\\
Assume now that $\mathfrak{R}(z)<0$. Denote by $C_{M,\alpha,B',B''}$ the constant on the RHS of \eqref{boundPsiNseqstripjanv26}. By repeated use of \eqref{eqfunchatSLog} we can shift $z$ by $-2i\pi$, thus getting $\textstyle \mathfrak{R}(\Log (\hat S_{N}(z-2i\pi))) = \mathfrak{R}(\Log (\hat S_{N}(z))) + \sum_{k=1}^N \Log\vert 1-e^{z+\frac{2i\pi}{N}}\vert = \mathfrak{R}(\Log (\hat S_{N}(z))) + \Log\vert 1-e^{Nz}\vert$. Then the difference of \eqref{equalitylemmaasyoct25} with the same identity where $z$ is replaced by $z-2i\pi$ has real part 
\begin{equation}\label{shiftby2pi}
\mathfrak{R}(\Psi_N(z)) - \mathfrak{R}(\Psi_N(z-2i\pi)) = -N\Log\vert 1-e^{Nz}\vert.
\end{equation}
Therefore $\mathfrak{R}\left(\Psi_N(z-2i\pi)\right) \leq \mathfrak{R}\left(\Psi_N(z)\right)+N\Log(1+\vert e^{Nz}\vert)\leq \mathfrak{R}\left(\Psi_N(z)\right)+N\Log(2)$, and hence $\textstyle \vert \exp(\frac{1}{N} \Psi_N(z))\vert = \exp(\frac{1}{N} \mathfrak{R}\left(\Psi_N(z)\right))\leq \exp(C_{M,\alpha,B',B''}+k\Log(2))$ for any $z$ in a  half-strip $\textstyle ]-\infty,0[ +i(2\pi k'+[\delta_N-\frac{\pi}{N}, 2\pi-\delta_N-\frac{\pi}{N}])$, where $k\leq k'\leq 0$.

To fix the ideas take now for instance $\textstyle \alpha:= \frac{1}{2}$ (namely, in the following discussion $\alpha$ must be chosen so that the width of the ``narrow'' half-strips defined below is $\textstyle <\frac{2\pi}{N}$, hence the width $2\pi - 2\delta_N$ of the above half-strips is $> \textstyle 2\pi - \frac{2\pi}{N}$, whence $\alpha\in ]0,1[$). Denote by $U'$ the union of the above half-strips, $\textstyle ]-\infty,0[ +i(2\pi k'+[-\frac{\pi}{2N}, 2\pi-\frac{3\pi}{2N}])$, where $k\leq k'\leq 0$. The set $U_{M,\alpha,k,K,c}$ is the union of $U'$ with the ``narrow'' half-strips $\textstyle ]-\infty,0[ +i\left( 2\pi k'+ [-\frac{3\pi}{2N},-\frac{\pi}{2N}]\right)$, $k-1\leq k'\leq 0$, and $\textstyle ]-\infty,-c[ +i[2\pi-\frac{3\pi}{2N}, K]$.

Thus it remains to prove the bound when $z$ lies within these half-strips. Consider first the ``narrow'' ones, with $k-1\leq k'\leq 0$. For any such a point $z$, we have $z+\textstyle{\frac{2i\pi}{N}}\in U'$. Therefore we consider the identity \eqref{eqfunchatSLog}. By means of \eqref{equalitylemmaasyoct25} it yields
\begin{multline*}
\frac{1}{N}  \left(\Psi_N(z-\textstyle{\frac{2i\pi}{N}}) - \Psi_N(z)\right)   \\ = \frac{N}{2i\pi}\left({\rm Li_2}(e^{z})-{\rm Li_2}(e^{z-\frac{2i\pi}{N}}) \right) +\frac{1}{2}\left(\Log(1-e^{z-\frac{2i\pi}{N}})-\Log(1-e^{z}) \right)+ \Log(1-e^{z}).
\end{multline*}
Using the Taylor expansion $\textstyle {\rm Li_2}(e^{z-\frac{2i\pi}{N}}) - {\rm Li_2}(e^{z})  = - \frac{2i\pi}{N} \int_0^1 \Log(1-e^{z-\frac{2i\pi}{N}t}) dt$ and changing $z$ to $\textstyle  z+\frac{2i\pi}{N}$ this can be written
$$\frac{1}{N}  \left(\Psi_N(z) - \Psi_N\left((z+\frac{2i\pi}{N}\right)\right) = \int_0^1 \Log(1-e^{z+\frac{2i\pi}{N}t}) dt +\frac{1}{2}\left(\Log(1-e^{z+\frac{2i\pi}{N}})+\Log(1-e^{z}) \right).$$
Since we assume $\mathfrak{R}(z)<0$, the real part of the RHS is bounded from above by $2\Log(2)$, and hence
\begin{multline*}\left\vert \exp\left(\frac{1}{N} \Psi_N(z)\right)\right\vert \leq \exp\left(\frac{1}{N} \mathfrak{R}\left(\Psi_N\left(z+\frac{2i\pi}{N}\right)\right)+2\log(2)\right)\\ \leq \exp(C_{M,\alpha,B',B''}+(k+2)\Log(2)).
\end{multline*}
Finally, assume that $\mathfrak{R}(z)<-c$, $c>0$. Since $\Log\vert 1-e^{z}\vert\ge \Log\vert 1-\vert e^{z}\vert \vert \geq \Log( 1-e^{-c})$, similarly as above we obtain $\textstyle \vert \exp(\frac{1}{N} \Psi_N(z))\vert \leq \exp(\frac{1}{N} \mathfrak{R}(\Psi_N(z-\frac{2i\pi}{N}))- 2\Log( 1-e^{-c}))\leq \exp(C_{M,\alpha,B',B''}- 2\Log( 1-e^{-c}))$ for $z\in \textstyle ]-\infty,-c[ +i[2\pi-\frac{3\pi}{2N}, 2\pi]$. When $2\pi \leq \mathfrak{I}(z)\leq K$, the bound of $\textstyle \vert \exp(\frac{1}{N} \Psi_N(z))\vert$ is obtained similarly by using \eqref{shiftby2pi}. This concludes the proof of (ii).\\ 
(iii) First note that any compact subset of $U$ is contained in $U_{M,\alpha,k,K,c}$ for some $M,\alpha,k,K$ and $c$, for $N$ large enough. By (ii) above and Montel's theorem, it follows that the sequence of functions $\textstyle \exp(\frac{1}{N} \Psi_N(z))$ is a {\it normal family} (that is, it has a subsequence which converges uniformly on compact subsets) on $U$. Since by (i) it converges simply to $1$ on a subset of $U$ having a limit point, by Vitali's Theorem on the convergence of holomorphic functions, the convergence is uniform on all compact subsets of $U$. This concludes the proof.
\end{proof}
 
\color{black}

Recall that $y^{1/N}:=\exp(\Log(y)/N)$ for any $y\in \mc^*$. Put 
\begin{equation}\label{gfunctiondef}
g_N(z):=\prod_{j=1}^{N-1} (1-z\zeta^{-j})^{j/N}.
\end{equation}
In order to prove Theorem \ref{asyinvteo}, we will need to study asymptotics of $g_N(e^z)$ for $z=\bl_{0,N}$ of the form \eqref{defseqNcolor}. First note that, by collecting conjugate terms it is immediate that $\vert g_N(1)^N\vert =N^{\frac{N}{2}}$, and factorizing by $\zeta^{-j/2}$ each term in the product yields
$$ g_N(1)^N = e^{\frac{\pi i}{12}(N-2)(N-1)(2N-1)} \prod_{j=1}^{N-1}2^j\sin\left(\frac{\pi j}{N}\right)^j$$
where the product is real positive. Therefore 
\begin{equation}\label{cstBBP}
g_N(1)^N  = N^{\frac{N}{2}}e^{\frac{\pi i}{12}(N-2)(N-1)(2N-1)}.
\end{equation}

Before considering $g_N(e^z)$ for more general $z$, let us do some rewriting. We have, for any odd integer $N$ and any $z \notin (2 i \pi/N) \Z$ (so that the following quotients are defined):
\begin{align}  g_N(e^z) & = \prod_{k=1}^{N-1} \bigl(1-e^z\zeta^{-k}\bigr)^{k/N}\nonumber  \\ & =  \prod_{k=1}^{N-1} \biggl(\frac{\hat S_N(z + 2(N-(k+1))i\pi/N)}{\hat S_N(z+ 2(N-k)i\pi/N)}\biggr)^{k/N}\nonumber\\  & =   \hat S_N(z)^{(N-1)/N}\ \prod_{k=1}^{N-1} \hat S_N\left(z+2(N-k)i\pi/N\right)^{-1/N} \nonumber \\ & =  \hat S_N(z) \ \prod_{k=1}^{N} \hat S_N\left(z+2ki\pi/N\right)^{-1/N}.\label{prodformgN} \end{align}

Now, it is tempting to use Lemma \ref{estimatehatS} to compare $g_N(e^z)$ to
\begin{equation}\label{ratiogNclass0mars25}
\exp\left( \frac{N}{2i\pi}{\rm Li_2}(e^z)\right) \cdot 
\prod_{k=1}^{N} \exp\left( \frac{N}{2i\pi}{\rm Li_2}(e^{z+2ki\pi/N})\right)^{-1/N}, 
\end{equation}
an expression that can be simplified thanks to the following factorization formula (see for example \cite[page 7]{Z}):

\begin{prop}
For any $x \in \C$, and any positive integer $N$, we have:
$${\rm Li_2}(x)=N \sum_{y \in \C, \ y^N=x} {\rm Li_2}(y).$$
\end{prop}

In particular, for any $z \in \C$ and any positive integer $N$, we have:
$$ \sum_{k=1}^N {\rm Li_2}(e^{z+2ki\pi/N})=\dfrac{1}{N} {\rm Li_2}(e^{Nz}).$$

It immediately follows that $g_N(e^z)$ should be ``comparable" (again, we are speaking non-rigorously for now) to 
\begin{equation}\label{ratiogNclassmars25}
\exp\left( \frac{N}{2i\pi}{\rm Li_2}(e^z)- \frac{1}{2i\pi N}{\rm Li_2}(e^{Nz})\right),
\end{equation}
whose asymptotics look much simpler to study.

Let us now turn back to rigor. We need to be careful about both the requirements that $z \notin (2 i \pi/N) \Z$ and that all the points  $z+2ki\pi/N$ live in a domain where Lemma \ref{estimatehatS} can be applied (and if not, we need to keep track of correction terms).

\begin{lem}\label{factorest} Let $u_0\in \mc\setminus \mr$, and for any odd integer $N$, let $\bl_{0,N}$ be such that $\bu_{0,N}:= e^{\bl_{0,N}}$ satisfies $\bu_{0,N}^N = u_0$, and $\bl_{0,N} = \bl_{0,\infty} + O_{N\to \infty}\left (N^{-1}\right )$, where $\bl_{0,\infty}\in i\mr$. We have:

(i) $\vert g_N(\bu_{0,N})\vert$ is bounded from below and above by the product of $e^{\frac{N}{2\pi} \mathfrak{I}({\rm Li_2}(e^{\bl_{0,\infty}}))}$ with a positive constant independent of $N$. In particular, if ${\rm Li_2}(e^{\bl_{0,\infty}})\in \mr$, then $\vert g_N(\bu_{0,N})\vert$ is bounded from below and above by a positive constant independent of $N$.

(ii) the same result holds true by replacing $g_N(\bu_{0,N})$ with $\hat{S}_N(\bl_{0,N})$. 		
\end{lem}

\proof First assume that $\Im(u_0) > 0$. Set
$$\textstyle \bl_{0,N} = \dfrac{1}{N}(\Log(u_{0})+2\pi i a_{0,N}' )$$
where $a_{0,N}'\in \{0, 1, \ldots, N-1\}$. Define
\begin{equation}\label{defdeltaM}
\delta_N:= \dfrac{1}{N}\min\{\Im(\Log(u_0)), \pi - \Im(\Log(u_0))\} \ \ \text{and} \ \ M:= |\Re(\Log(u_0))|.
\end{equation}
We want to apply the reasoning preceding Lemma \ref{factorest} to the terms $\hat S_N\left(z+2ki\pi/N\right)$ on the right side of \eqref{prodformgN}, taking $z=\bl_{0,N}$. We first notice that $\delta_N>0$, and $\bl_{0,N} \notin (2 i \pi/N) \Z$ as required (since $u_0\ne 1$). We then immediately see that
\begin{equation}\label{chaine}
\bl_{0,N}, \bl_{0,N} + \frac{2i\pi}{N}, \ldots,  \bl_{0,N} + \frac{2i\pi(N-1-a'_{0,N})}{N} \in [-M,M]+i[\delta_N-\frac{\pi}{N},2\pi - \delta_N-\frac{\pi}{N}]
\end{equation}
and
$$ \bl_{0,N} + \frac{2i\pi(N-a'_{0,N})}{N}, \ldots,  \bl_{0,N} + \frac{2i\pi(N-1)}{N} \in [-M,M]+i[2\pi+\delta_N-\frac{\pi}{N},4\pi - \delta_N-\frac{\pi}{N}].$$
These two sequences of scalars realize all the arguments on the right side of \eqref{prodformgN} but the last of the product, $ \bl_{0,N} + 2i\pi$. We can directly apply the inequalities of Lemma \ref{estimatehatS} for $z$ being any one among the $N-a'_{0,N}$ values of the sequence \eqref{chaine}. The error bound is 
\begin{equation}\label{errorboundgNoct25}
\Log(\mathrm{c}_M)+ \dfrac{B_{\delta_N}}{N} = \Log(\mathrm{c}_M) + \dfrac{B'}{\delta_N N} + \dfrac{B''}{N},
\end{equation}
which by \eqref{defdeltaM} is bounded from above by a constant independent of $N$. For $\bl_{0,N} + 2i\pi$ or any of the $a'_{0,N}$ values of the sequence below \eqref{chaine}, we need to modify the corresponding terms $\hat S_N\left(z+2ki\pi/N\right)$ on the right side of \eqref{prodformgN}, by re-shifting $z_k:=z+2ki\pi/N$ using $N$ times the formula \eqref{eqfunchatS}. We have
$$\hat{S}_N(z_k) = \dfrac{\hat{S}_N(z_k-2i\pi)}{1-e^{Nz_k}},$$
each factor $1-e^{Nz_k} = 1-u_0$ is raised at the power $-1/N$ in \eqref{prodformgN}, and these factors are no more numerous than $N$. So the re-shiftings add at worst a bounded non zero term to the error bound \eqref{errorboundgNoct25}, which changes only $\mathrm{c}_M$. Eventually, the ratio of \eqref{prodformgN} over \eqref{ratiogNclass0mars25}, or equivalently \eqref{ratiogNclassmars25}, is bounded from above and below by the product of $N+1$ positive constants independent of $N$, with $N$ of them raised at the power $1/N$. Hence the ratio \eqref{prodformgN}/\eqref{ratiogNclassmars25} is bounded by a constant. To conclude, it therefore suffices to prove that
\begin{equation}\label{ratiojuil25}
\exp\left( \frac{N}{2i\pi}{\rm Li_2}(e^{\bl_{0,N}})- \frac{1}{2i\pi N}{\rm Li_2}(e^{N\bl_{0,N}})
\right)
\end{equation}
is bounded from above and below by constants independent of $N$ multiplied with $e^{\frac{N}{2\pi} \mathfrak{I}({\rm Li_2}(e^{\bl_{0,\infty}}))}$. But ${\rm Li_2}(e^{N\bl_{0,N}})={\rm Li_2}(u_0)$, and $\bl_{0,N} = \bl_{0,\infty} + O_{N \to \infty}\left (N^{-1}\right )$ implies
$$ {\rm Li_2}(e^{\bl_{0,N}}) ={\rm Li_2}(e^{\bl_{0,\infty}})+O_{N \to \infty}(1/N)$$
by a Taylor expansion as in \eqref{Taylor0} (which is valid, since ${\rm Li_2}$ is differentiable on $\mc \setminus [1,+\infty[$, and $e^{N\bl_{0,N}}=u_0\notin ]1,+\infty[$ implies $e^{\bl_{0,N}} \notin ]1,+\infty[$). It follows
$$\exp\left( \frac{N}{2i\pi}{\rm Li_2}(e^{\bl_{0,N}})\right )= \exp\left(\frac{N}{2\pi} \mathfrak{I}({\rm Li_2}(e^{\bl_{0,\infty}}))\right)\exp (O_{N \to \infty}(1)).$$
This proves the last claim, and concludes the proof of the lemma when $\Im(u_0) > 0$. Similar arguments apply when $\Im(u_0) < 0$: take $\delta_N=\pi/N$, and start the sequence \eqref{chaine} at $\bl_{0,N}+ 2i\pi/N$. Only the arguments $\bl_{0,N}$ and $\bl_{0,N} + 2i\pi$ on the right side of \eqref{prodformgN} are missing from this sequence and the following one. With these changes one can conclude by reasoning as when $\Im(u_0) > 0$. \cvd

\subsection{The QHI of cusped manifolds}\label{sec:QHIcusped} The aim of this section is to introduce some general features of the QHI, provide links to the literature (thus explaining notations), and to explain the splitting formula of the QHI into the symmetry defects and the reduced QHI. We refer to \cite{GT} for the definition of the QHI of cusped manifolds by means of state sums over geometric branched ideal triangulations, and to \cite{AGT} for the general situation of non necessarily branched triangulations. The QHI of cusped manifolds used in this paper are those of \cite[Theorem 1.1]{AGT}, in the case where the bulk weights $h_f,h_c$ are zero, or are determined by the boundary weights $k_f,k_c$ (as for link complements). (By \cite[Corollary 1.8 (1)]{NA}, another choice of bulk weight would at worst modify these QHI by multiplication with a $4N$-th root of unity, which has no consequence on the results of this paper.)

Let $M$ be a cusped manifold. As usual we denote by $\bar{M}$ the compact manifold with interior $M$, $X(M)$ is the variety of {\it $PSL(2,\mc)$-characters of $M$}, and $X_{hyp}(M) \subset X(M)$ is the ``geometric'' irreducible component, containing the character $\rho_{hyp}$ of the complete hyperbolic holonomy (see Section \ref{GhypN}).

The QHI of $M$ are sequences of complex numbers $\Hh_N(M, \rho,k_c,k_f)$ modulo multiplication by $2N$-th roots of unity, indexed by the odd positive integers $N\geq 3$. The arguments of $\Hh_N(M, \rho,k_c,k_f)$ are:
\begin{itemize}
	\item[(i)] a character $\rho\in X_{hyp}(M)$, 
	\item[(ii)] classes $k_c\in H^1(\partial \bar{M},\mz)$ and $k_f\in H^1(\partial \bar{M},\mc)$ satisfying \eqref{fconstraint0}. 
\end{itemize}
The invariant $\Hh_N(M, \rho,k_c,k_f)$ is {\it computed} by the total contractions, denoted $\Hh_{N, (T,b,c)}$ and called {\it QHI state sums},  of certain tensor networks associated to the ideal triangulations $T$ of $M$, similarly as in the Turaev-Viro TQFTs. The QHI state sums are functions of combinatorial data on $T$ encoding $\rho\in X(M)$ and $k_c,k_f$ as in (i)-(ii): $\rho$ is encoded by a point $w_\rho$ of the gluing variety $G(T,b)$ (for some $b$), which is possible for most characters $\rho$ in $X_{hyp}(M)$ (see Remark \ref{mostrho}), and the classes $k_f$ and $k_c$ are a flattening weight and a charge weight, encoded by flattenings $f$ and charges $c$ as in Remark \ref{defflatcharge}-\ref{cfweight}. When $\rho$ is contained in $X_{hyp}(M)$ all the choices of $(T,b,w_\rho,c,f)$ eventually give the QHI state sums the same value, which by definition is $\Hh_N(M, \rho,k_c,k_f)$.

It is shown in \cite{NA} that, up to a possible stronger ambiguity modulo multiplication by $4N$th-roots of unity, which we denote by $=_{\mu_{4N}}$, we have a factorization into refined invariants (recall that $l\kappa := k_f-i\pi k_c$)
\begin{equation}\label{factQHIalphared}
	\Hh_N(M,\rho,k_f,k_c) =_{\mu_{4N}} \alpha_N(M,\rho,k_c;\mathfrak{s})\Hh_N^{red}(M,\rho,l\kappa;\mathfrak{s}),
\end{equation}
called respectively the {\it symmetry defect} and the {\it reduced} QHI. These refined invariants depend on an additional structure $\mathfrak{s}$ on $M$, called a {\it non ambiguous} structure, which is of topological nature ($\mathfrak{s}$ is an equivalence class of ideal triangulations under a restricted set of branched Pachner $2-3$ moves). The argument $l\kappa$ on the right side is the log weight, see \eqref{formulelkappa} and \eqref{deflkappa} (in \cite{NA} it was denoted $\kappa$, and called a fused weight). 

\begin{remark}\label{NAdefectmars25} {\rm When $M$ fibers over $S^1$, {\it any} fibration $\psi\colon M\ra S^1$ (with oriented fibers) defines a canonical non ambiguous structure $\mathfrak{s}_\psi$ on $M$, carried by the ideal triangulations of $M$ which are layered triangulations for the fibration (\cite[Prop. 6.3(1)]{NA}). Such triangulations are taut, which allows to associate to them a canonical charge, as in Remark \ref{defflatcharge}, whence a canonical charge weight $k_\psi$, as in Remark \ref{cfweight}, to the fibration (\cite[Prop. 4.3(2) and 4.8(2)]{NA}).  From the definition of the symmetry defect (\cite[page 808]{NA}), it is quite immediate to check that $\alpha_N(M,\rho,k_\psi;\mathfrak{s}_\psi)=_{4N} 1$ for every $\rho$, whence $\Hh_N(M,\rho,k_f,k_\psi) =_{\mu_{4N}} \Hh_N^{red}(M,\rho,k_f-i\pi k_\psi;\mathfrak{s}_\psi)$. Note, however, that the classes $k_\psi$ are very peculiar, and a layered triangulation can carry several (possibly taut) non ambiguous structures, that the symmetry defect $\alpha_N(M,\rho,k_c;\mathfrak{s})$ can distinguish (see the examples in \cite[Section 9]{NA}).} 
\end{remark}

Recall the relations \eqref{lkconstraint0}-\eqref{h2class}. By setting 
\begin{equation}\label{dconstraintclass} h_\rho(\gamma) := l\kappa(\gamma)-\Log(\delta_{w_\rho}(\gamma)) \in i\pi\mz
\end{equation}
and using that $l\kappa = k_f-i\pi k_c$ (see \eqref{deflkappa}), we can explain the notations of Theorem \ref{asyinvteo}:
\begin{equation}\label{HNdefh}
	\Hh_{N}(M,\rho,h_\rho,k_c) := \Hh_N(M,\rho,k_f,k_c).
\end{equation}

Now, for every class $a \in H^1(\partial \bar M; \pi i N \mz)$ we have $\Hh_N^{red}(M,\rho,l\kappa + a;\mathfrak{s}) = \Hh_N^{red}(M,\rho,l\kappa;\mathfrak{s})$ (see \cite[Theorem 1.6 (3)]{NA}). From this and \eqref{h2class} it follows that $\Hh_N^{red}(M,\rho,l\kappa;\mathfrak{s})$, regarded as a function of $l\kappa$, eventually depends only the boundary weight $\kappa\in H^1(\partial \bar M; \mc^*)$, defined in \eqref{kappaborddef}. We shall therefore denote
\begin{equation}\label{HNredkappamars25}
\Hh_N^{red}(M,\rho,\kappa;\mathfrak{s}):=\Hh_N^{red}(M,\rho,l\kappa;\mathfrak{s}).
\end{equation}
Recall that, since $\kappa$ is given by its values on a basis of $H_1(\partial \bar{M};\mz)$, the relations \eqref{lkconstraint0}-\eqref{h2class} show that $\kappa$ is a choice of $N$-th roots of the values of $\delta_{w_\rho}$ on these basis elements. Combining \eqref{factQHIalphared}, \eqref{HNdefh} and \eqref{HNredkappamars25} we can write
\begin{equation}\label{splitHNteo1.1}
\Hh_{N}(M,\rho,h_\rho,k_c) =_{\mu_{4N}} \alpha_N(M,\rho,k_c;\mathfrak{s})\Hh_N^{red}(M,\rho,\kappa;\mathfrak{s}).
\end{equation}
The invariant $\Hh_N(M,\rho,h_\rho,k_c)$ or $\Hh_N^{red}(M,\rho,\kappa;\mathfrak{s})$, for fixed $k_c$ and varying couple $(\rho,k_f)$, or fixed $\mathfrak{s}$ and varying couple $(\rho,\kappa)$ respectively, defines a rational function on the variety ${}_NX_{hyp}(M)$ of Section \ref{GhypN}  (\cite[Theorem 1.1]{AGT}, \cite[Corollary 1.8]{NA}). This function is obtained from the QHI state sums $\Hh_{N,(T,b,c)}$: these are rational functions on the space $G_{hyp,N}(T,b)$ defined in \eqref{GhypNdef}, and factor through the map $\widetilde{hol}$ in \eqref{grostildeholmars25}. 

\subsubsection{The case of $M=S^3\setminus 4_1$}\label{caseMsplitmars25} It follows immediately from definitions (see \cite[Section 8]{NA}) that $\alpha_N(M,\rho,k_c;\mathfrak{s})$ is in $PM(N)$, as defined in Section \ref{sec:sketchproof}, for every fixed $\rho\in X_{hyp}(M)$, $k_c\in H^1(\partial \bar M;\mz)$ and $\mathfrak{s}$. Because of the splitting \eqref{splitHNteo1.1}, the proof of Theorem \ref{asyinvteo} is therefore reduced to the asymptotic analysis of $\Hh_N^{red}(M,\rho,\kappa;\mathfrak{s})$ as $N\ra +\infty$.

Let us give the details of this argument for $M=S^3\setminus 4_1$. The invariant $\Hh_N^{red}(M,\rho,\kappa;\mathfrak{s})$ has a simple state sum $\Hh_{N}^{red}(\bu_{0,N},\bu_{1,N},\bv_{0,N},\bv_{1,N})$ on the branched triangulation $(T,b)$ of Figure \ref{Teight}, given in \eqref{form1} below. In the notations of Section \ref{4_1} the couple $(\rho,\kappa)$ is encoded by the parameters 
$$\bu_{k,N} = \exp\left(\frac{1}{N}\left(\Log(u_k)+ \pi i (N+1)a_k\right)\right)\ ,\ \bv_{k,N} = \exp\left(\frac{1}{N}\left(\Log(v_k) + \pi i (N+1)b_k\right)\right)$$
for $k\in \{0,1\}$, where $u_k,v_k\in \mc_*$ are fixed shape parameters of the tetrahedra of $T$. The non ambiguous structure $\mathfrak{s}$ is defined by the branching $b$ of $T$, and was studied in \cite[Section 9.2]{NA} (where it was denoted $[(T,\omega_1)]$). The symmetry defect $\alpha_N(M,\rho,k_c;\mathfrak{s})$ is computed on $(T,b)$ by the monomial $(\bu_{0,N}^{-1}\bv_{0,N}^{-1})^{\frac{N-1}{2}}$ for a certain charge $c$.

From \eqref{splitHNteo1.1} and this discussion it follows that the state sum expression of $\Hh_{N}(M,\rho,h_\rho,k_c)$ on $(T,b)$ is \begin{equation}\label{ssumrednonred}
	\Hh_{N,(T,b,c)}(\bu_N,\bv_N) := (\bu_{0,N}^{-1}\bv_{0,N}^{-1})^{\frac{N-1}{2}}\Hh_{N}^{red}(\bu_{0,N},\bu_{1,N},\bv_{0,N},\bv_{1,N}).
\end{equation}

\begin{lem}\label{lem:lim:sym:defect} We have:
$$\lim_{N\ra +\infty} \dfrac{2\pi}{N}\Log \left \vert (\bu_{0,N}^{-1}\bv_{0,N}^{-1})^{\frac{N-1}{2}} \right \vert =0.$$
\end{lem}

\begin{proof} Indeed
\begin{align*}
	\dfrac{2\pi}{N}\Log \left \vert (\bu_{0,N}^{-1}\bv_{0,N}^{-1})^{\frac{N-1}{2}} \right \vert 
	&= \dfrac{2\pi}{N} \frac{N-1}{2}\Log \left \vert (\bu_{0,N}^{-1}\bv_{0,N}^{-1}) \right \vert 
	\\
	&= -\pi \frac{N-1}{N^2}\mathfrak{R}(\Log (u_{0})+\Log (v_{0}))
\end{align*}
using the above expressions of $\bu_{k,N}$, $\bv_{k,N}$ in the last equality. The result follows.\end{proof}

\subsection{The QHI state sum of $M=S^3\setminus 4_1$}\label{sec:statesum} In this section we describe the formula of $\Hh_{N}^{red}(M,\rho,\kappa; \mathfrak{s})$ in the case of the figure-eight knot complement $M$, using the branched ideal triangulation $(T,b)$ of Figure \ref{Teight}.

Recall the lifted holonomy map $\widetilde{hol}$ in \eqref{grostildeholmars25}; by \eqref{Ghypcas41mars25} we have $G_{hyp,N}(T,b)=G_N(T,b)$. For any $\rho \in hol(G(T,b))\subset X_{hyp}(M)$ and compatible boundary weight $\kappa$, let $(\bu_0,\bu_1,\bv_0,\bv_1)\in G_N(T,b)$ (as in Section \ref{4_1}) be such that $\widetilde{hol}(\bu_0,\bu_1,\bv_0,\bv_1)=(\rho,\kappa)$.  Denote by $\mathfrak{s}$ the non ambiguous structure carried by $(T,b)$. 

By \cite[Section 6.4, page 912]{AGT0} and the formula \eqref{splitHNteo1.1}, the reduced QHI of $(M,\rho,\kappa;\mathfrak{s})$ can be expressed by the state sum 
\begin{align}
\Hh_{N}^{red}(M,\rho,\kappa; \mathfrak{s}) & =  \Hh_{N}^{red}(\bu_0,\bu_1,\bv_0,\bv_1)\notag \\
& =  \frac{g_N(\bu_0)g_N(\bv_0^*)^*}{|g(1)|^2}\sum_{\alpha,\beta=0}^{N-1}
  \zeta^{\beta^2-\alpha^2} \omega_N(\bu_0,\bu_1^{-1}\vert \beta)\
  \omega_N(\bv_0^*,\bv_1^*{}^{-1}\vert \alpha)^*\nonumber \\
& = \frac{g_N(\bu_0)g_N(\bv_0^*)^*}{|g(1)|^2}\ \Sigma_N(\bu_0,\bu_1)\ \Sigma_N(\bv_0^*,\bv_1^*)^*  \label{form1} .
\end{align}
where $^*$ denotes the complex conjugation, the functions $\omega_N$ and $g_N$ were defined in \eqref{functomegadef} and \eqref{gfunctiondef}, and we put
\begin{equation}\label{sumS}
\Sigma_N(\bu_0,\bu_1) = \sum_{\beta=0}^{N-1} \zeta^{\beta^2}
\omega_N(\bu_0,\bu_1^{-1}\vert \beta).
\end{equation} 
 
\subsection{Integral representation} \label{INTREPEST} As we are going to consider the functions in \eqref{form1} when $N\rightarrow +\infty$, we re-introduce now the hidden dependence on $N$ in the notations, thus letting $(\bu_N)_N:=((\bu_{0,N},\bu_{1,N},\bu_{2,N}))_N$ be a sequence of triples of quantum shape parameters on $\Delta^0$ as in Section \ref{4_1}. Contrary to the more general situation considered in Section \ref{ASYWEIGHTS}, here we assume that $\bu_{k,N}^N$ is constant. Thus, as in \eqref{qshapeuvoct25} we write \begin{equation}\label{recapuvarmars25} \bu_{k,N} = \exp(\bl_{k,N}),\ \bl_{k,N}:= \frac{1}{N}(\Log(u_{k}) +\pi ia_{k}) + \pi i a_k, a_{k}\in \mz.\end{equation}
Using the formula \eqref{relomS0} one can write \eqref{sumS} as
\begin{equation}
\Sigma_N(\bu_{0,N},\bu_{1,N}) = 1 + \frac{1}{\hat S_N(\bl_{0,N})}\sum_{\beta=1}^{N-1} e^{\frac{2\pi i}{N} \beta^2 -\bl_{1,N}\beta} \textstyle \hat S_N(\bl_{0,N}+2\beta i\pi/N).\label{int01}
\end{equation}
Note that for $\beta\in \mz$ we have $e^{\bl_{1,N}\beta} = e^{(\bl_{1,N}+2i\pi k)\beta}$ for any $k \in \mz$, so the formula \eqref{int01} is invariant under translation of $\bl_{1,N}$ by $2i\pi$. In fact, because the QHI solely depend on the quantum shape parameters $\bu_{k,N}$ and $\bv_{k,N}$, by adding even integers to $a_0$ or $a_1$ if necessary we can assume that in the above formula we have
\begin{equation}\label{situation2}
\forall N,\ \mathfrak{I}(\bl_{0,N})\in ]-2\pi,0]\ ,\ \mathfrak{I}(\bl_{1,N})\in ]0,2\pi].
\end{equation}

Recall that the poles of $\hat S_N(z)$ are the points $z=2i\pi+2pi\pi/N$, $p\in \mn$, the simple poles being those with $p\in \{0,\ldots,N-1\}$ (see \eqref{PP1}). On another hand the poles of $\coth(Nz/2)$ are the points $z=2i\pi \beta/N$, $\beta\in \mz$, and they are all simple poles, with residue $2/N$. Because $\mathfrak{I}(\bl_{0,N})\subset ]-2\pi,0]$, it follows that the function
$$\textstyle\phi\colon z\mapsto \hat S_N(\bl_{0,N}+z)\coth(\frac{Nz}{2})\frac{N}{2}$$
is meromorphic in the half-space $\{\textstyle \mathfrak{I}(z) < 2\pi\}$, and has poles $\textstyle z_\beta:=2i\pi \beta/N$ where $\beta\in\mz$, $\beta\leq N-1$. These poles are simple, and ${\rm Res}(\phi,z_\beta) =\hat S_N(\bl_{0,N}+2\beta i\pi/N)$. 

Let $C_N$ be the boundary of a rectangle with counter clockwise orientation, sides parallel to the real and imaginary axes, containing the segment $\textstyle i[\frac{2\pi}{N},2\pi-\frac{2\pi}{N}]$ in its interior, and close to it so that we have constants $\varepsilon>0$ and $s_N^-,s_N^+\in ]0,\textstyle \frac{2\pi}{N}[$ satisfying  
\begin{equation}\label{choiceCN}
C_N \subset \{\vert \mathfrak{R}(z) \vert \leq \varepsilon\}\cap \{s_N^-\leq \mathfrak{I}(z) \leq 2\pi-s_N^+\}\ . \end{equation}
See Figure \ref{fig:contours:CN+-}. An immediate application of Cauchy's residue theorem gives:
\begin{lem}\label{lemintrep} Assuming \eqref{situation2} and \eqref{choiceCN}, we have
\begin{equation}\label{int1}
\Sigma_N(\bu_{0,N},\bu_{1,N}) =1+ \frac{1}{\hat S_N(\bl_{0,N})} \frac{N}{4i\pi}\int_{C_N} e^{\frac{N}{2i\pi}(z^2 - \bl_{1,N}z)}\hat S_N(\bl_{0,N} + z)\textstyle \coth(\frac{Nz}{2})dz.
\end{equation}
\end{lem}

\subsection{Some logarithmic limits of log-parameters}\label{loglimmars25} In this section we focus on the peculiar logarithmic limits of log-parameters \eqref{deflimlkmars26} selected by the two situations (a)-(b) of Theorem \ref{asyinvteo}. These log-parameters have the form \eqref{recapuvarmars25} (at least for $N$ large enough); that is, they have constant shape parameters $u_k\in \mc_*$ and constant edge colors $a_k\in\mz$. In this situation, and assuming \eqref{situation2}, as usual we put, for $k\in\{0,1\}$:
$$\bl_{k,\infty}^\bu := \lim_{N\ra +\infty} \bl_{k,N}.$$
The edge colors $a_0,a_1,a_2$ solve the linear system \eqref{tetrel41}-\eqref{edgerel41}-\eqref{standardf3}, where $a_0$ is a free variable. Then, taking $a_0$ even and $\geq 2$ if necessary (so that $\textstyle \mathfrak{I}(\Log(u_{0}) +\pi ia_{0})>0$), the assumption \eqref{situation2} implies
\begin{equation}\label{l0inftymars25}
\bl_{0,\infty}^\bu = -2i\pi.
\end{equation}
By the equations \eqref{standardf3} and the comments thereafter, when $a_0$ is positive and big enough, $a_1$ is negative. Then, again by \eqref{situation2}, we have the following alternative:
\begin{equation}\label{l1inftymars25}
\bl_{1,\infty}^\bu = \left \{\begin{matrix}
i\pi & \text{(case} (a): a_1 \ \text{is\ odd}),\\
2i\pi & \text{(case} (b): a_1 \ \text{is\ even}).
\end{matrix}
\right.
\end{equation}
Note that \eqref{l0inftymars25}-\eqref{l1inftymars25} describe the possible limits of the parameters of the integral representation of $\Sigma_N(\bu_{0,N},\bu_{1,N})$ in \eqref{int1}. 

Next consider $\Sigma_N(\bv_{0,N}^*,\bv_{1,N}^*)$. As usual set $\bv_{k,N} := \exp(\bl_{k,N}^\bv)$, $\textstyle \bl_{k,N}^\bv:= \frac{1}{N}(\Log(v_{k}) +\pi ib_{k}) + \pi i b_k$, $\bl_{k,\infty}^\bv := \textstyle \lim_{N\ra +\infty} \bl_{k,N}^\bv$. Then $\textstyle \bl_{k,N}^{\bv^*}= \frac{1}{N}(\Log(v_{k})^* -\pi ib_{k}) - \pi i b_k$, and we consider $\bl_{k,\infty}^{\bv^*} := \textstyle \lim_{N\ra +\infty} \bl_{k,N}^{\bv^*}$ assuming $\mathfrak{I}((\bl_{k,N}^{\bv^*}))$ satisfies \eqref{situation2}. By using the expressions of $b_0$, $b_1$ in the equations \eqref{standardf3}, the same analysis as above shows that, again for $a_0$ even, positive, and big enough, we have
\begin{equation}\label{l01inftymars25}
\bl_{0,\infty}^{\bv^*} = -2i\pi
\quad ,\quad \bl_{1,\infty}^{\bv^*} = \left \{\begin{matrix}
i\pi & \text{(case} (a): a_1 \ \text{is\ odd}),\\
2i\pi & \text{(case} (b): a_1 \ \text{is\ even}).
\end{matrix}
\right.
\end{equation}
These values give the possible limits of the parameters of the integral representation of $\Sigma_N(\bv_{0,N}^*,\bv_{1,N}^*)$.

\begin{lem}\label{rem:reduction:sigma:integral} The two cases (a) and (b) of Theorem \ref{asyinvteo} correspond respectively to the cases (a) and (b) of Theorem \ref{thm:main_triang} and to the cases (a):  $\bl_{0,\infty}^\bx = -2i\pi$, $\bl_{1,\infty}^\bx = i\pi$ and (b): $\bl_{0,\infty}^\bx = -2i\pi$, $\bl_{1,\infty}^\bx = 2i\pi$, for both $\bx = \bu$ and $\bv^*$. 
\end{lem}
\begin{proof} By the equations \eqref{standardf3} the parity of $a_1$ and $b_1$ is the parity of $h_\rho(\lambda_K)/2i\pi$. The conclusion then follows from \eqref{l0inftymars25}-\eqref{l1inftymars25}-\eqref{l01inftymars25}.
\end{proof}

\section{Asymptotics of classical integrals on a vertical contour}\label{sec:SPM}

The analysis of the integral on the RHS of \eqref{int1},  with the logarithmic limits $\bl_{0,\infty}$, $\bl_{1,\infty}$ specified by Lemma \ref{rem:reduction:sigma:integral}, will be done in Section \ref{sec:rect:contour}. Here, as a warm-up situation we focus on simpler integrals $I_{\pm,N}(\bl_{0,\infty},\bl_{1,\infty})$, that we call ``classical'' integrals, with integrands given by an exponential term taken from the asymptotics of the integrand in \eqref{int1} as $N\to +\infty$. These classical integrals are defined in Section \ref{sec:class_int}; there, we provide also motivations for the rest of Section \ref{sec:SPM}. In particular, in Lemma \ref{lemcritpts} we will see that the critical sets of the integrands recover a neighborhood of the geometric point $w_{hyp}$ of the gluing variety $G(T,b)$. The rest of Section \ref{sec:SPM} is devoted to the asymptotics of $I_{\pm,N}(\bl_{0,\infty},\bl_{1,\infty})$, again with the logarithmic limits specified by Lemma \ref{rem:reduction:sigma:integral}; the proofs will make no use of Lemma \ref{lemcritpts} (all related facts are checked in the given situations).

\subsection{Classical integrals and their critical sets}\label{sec:class_int} Recall the contour $C_N$ from  \eqref{choiceCN}. By the assumption \eqref{situation2}, when $z\in C_N$ both cases $0< \mathfrak{I}(\bl_{0,\infty}+z) < 2\pi$ and $-2\pi < \mathfrak{I}(\bl_{0,\infty}+z) < 0$ may happen. If $0< \mathfrak{I}(\bl_{0,\infty}+z) < 2\pi$, then $\hat S_N(\bl_{0,\infty}+z)$ is given by the formulas \eqref{intrepS}-\eqref{defShat}, and Lemma \ref{estimatehatS} yields
\begin{equation}\label{equiv1}
\hat S_N(\bl_{0,\infty} + z)= e^{\frac{N}{2i\pi} {\rm Li}_2(e^{\bl_{0,\infty} + z})}\mathcal{O}_{N\ra \infty}(1).
\end{equation}
If $-2\pi < \mathfrak{I}(\bl_{0,\infty}+z) < 0$, we get $\hat S_N(\bl_{0,\infty}+z) = (1-e^{N(\bl_{0,\infty}+z)})\hat S_N(\bl_{0,\infty}+z+2\pi i)$ by using $N$ times the functional equation \eqref{eqfunchatS}; so in this case we have 
\begin{equation}\label{equiv2}
\hat S_N(\bl_{0,\infty} + z)= \left\lbrace \begin{array}{rl} e^{N\bl_{0,\infty}}e^{\frac{N}{2i\pi} ({\rm Li}_2(e^{\bl_{0,\infty} + z})+2i\pi z)}\mathcal{O}_{N\ra \infty}(1) & {\rm if}\ \mathfrak{R}(\bl_{0,\infty} + z)>0\\ e^{\frac{N}{2i\pi} {\rm Li}_2(e^{\bl_{0,\infty} + z})}\mathcal{O}_{N\ra \infty}(1) & {\rm if}\ \mathfrak{R}(\bl_{0,\infty} + z)<0. \end{array}\right.
\end{equation}
It is easily checked that $\textstyle \coth(\frac{Nz}{2})$ is bounded on $C_N$, and therefore the integrand on the RHS of \eqref{int1} satisfies, for $z\in C_N$: 
\begin{multline}\label{equiv3mars26}
e^{\frac{N}{2i\pi}(z^2 - \bl_{1,N}z)}\hat S_N(\bl_{0,N} + z)\textstyle \coth(\frac{Nz}{2}) =\\ \left\lbrace \begin{array}{rl} e^{N\bl_{0,\infty}}e^{\frac{N}{2i\pi} \mathcal{L}_+(z;\bl_{0,\infty},\bl_{1,\infty})}\mathcal{O}_{N\ra \infty}(1) & {\rm if}\ \mathfrak{R}(\bl_{0,\infty} + z)>0 \ {\rm and}\\
& -2\pi < \mathfrak{I}(\bl_{0,\infty}+z) < 0,\\ e^{\frac{N}{2i\pi} \mathcal{L}_-(z;\bl_{0,\infty},\bl_{1,\infty})}\mathcal{O}_{N\ra \infty}(1) & {\rm otherwise}, \end{array}\right. 
\end{multline}
where 
\begin{equation}\label{defL+-}
\left\lbrace\begin{array}{ll}
\mathcal{L}_-(z;\bl_{0,\infty},\bl_{1,\infty}) := {\rm Li}_2(e^{\bl_{0,\infty} + z})+z^2- \bl_{1,\infty} z,& \\
\mathcal{L}_+(z;\bl_{0,\infty},\bl_{1,\infty}) := {\rm Li}_2(e^{\bl_{0,\infty} + z})+z^2 -(\bl_{1,\infty}-2i\pi) z. &
\end{array}\right.
\end{equation}
Note that the two halves of $C_N$, in the left and right half spaces of $\mc$, are homotopic to $i[s_N^-,2\pi-s_N^+]$ (with opposite orientations) relatively to their endpoints. We call 
\begin{equation}
I_{\pm,N}(\bl_{0,\infty},\bl_{1,\infty}) = \int_{i[s_N^-,2\pi-s_N^+]} e^{\frac{N}{2i\pi}\mathcal{L}_\pm(z;\bl_{0,\infty},\bl_{1,\infty})}dz\label{I+}
\end{equation}
the {\it classical integrals}.\\
Because of \eqref{equiv3mars26}, it is reasonable to expect that the asymptotics of the integrals $I_{\pm,N}(\bl_{0,\infty},\bl_{1,\infty})$ have a lot to do with those of the integral on the RHS of \eqref{int1} (see Lemma \ref{qcintjui25} for a rigourous statement). Then, in view of the saddle point method (see Section \ref{sub:SPM:PV}), we are naturally let to consider the critical sets of the functions $z\mapsto \mathcal{L}_\pm(z;\bl_{0,\infty}^\bx,\bl_{1,\infty}^\bx)$, where the parameters $\bl_{0,\infty}^\bx$, $\bl_{1,\infty}^\bx$, with $\bx=\bu$ and $\bv$, are associated to the tetrehedra $\Delta^0$ and $\Delta^1$, as usual. In particular, by Lemma \ref{lemYo} it is interesting to consider the situations where $(e^{\bl_{0,\infty}^{\bu}},e^{\bl_{1,\infty}^{\bu}},e^{\bl_{0,\infty}^{\bv}},e^{\bl_{1,\infty}^{\bv}}) \in \{0\}\times \mathcal{A}(T,b)\subset\mathcal{C}(T,b)\times \mathcal{A}(T,b)$. By \eqref{eqloginfinity} and \eqref{asycurvemu}, such a tuple $(\bl_{0,\infty}^{\bu}, \bl_{1,\infty}^{\bu},\bl_{0,\infty}^{\bv}, \bl_{1,\infty}^{\bv})$ satisfies the sole relation
\begin{equation}\label{asycurvelog}
\bl_{1,\infty}^{\bu}+2\bl_{2,\infty}^{\bu}-2\bl_{0,\infty}^{\bv}-\bl_{1,\infty}^{\bv} = 0 \quad {\rm mod}(2\pi i),
\end{equation}
and it determines an asymptotic angular holonomy by the system 
\begin{equation}\label{asycurvelog2}
\left\lbrace\begin{array}{ll} 
\bl_{0,\infty}^{\bu}+\bl_{1,\infty}^{\bu}+\bl_{0,\infty}^{\bv}+\bl_{1,\infty}^{\bv}  & = - i\arg(\kappa^\infty(\mu_K))  \quad {\rm mod}(2\pi i)\\
4\bl_{0,\infty}^{\bu}+2\bl_{1,\infty}^{\bu}  & = i\arg(\kappa^\infty(\lambda_K)) \quad {\rm mod}(2\pi i).\end{array}\right.
\end{equation} 
Denote by $\mathcal{A}'(T,b)$ the set of solutions of the equation \eqref{asycurvelog}. Let us define a sign
\begin{equation}\label{signcriteqhol}
\star:=\left\lbrace\begin{array}{ll} +1 & {\rm if} \ \ \textstyle 2\bl_{0}^{\bu,\infty}+\bl_{1}^{\bu,\infty} = \frac{i\arg(\kappa^\infty(\lambda_K))}{2}\quad  {\rm mod}(2\pi i)\\ -1 & {\rm  if} \ \ \textstyle 2\bl_{0}^{\bu,\infty}+\bl_{1}^{\bu,\infty}  = \frac{i\arg(\kappa^\infty(\lambda_K))}{2} +\pi i\quad {\rm mod}(2\pi i).\end{array}\right.
\end{equation}
\begin{lem}\label{lemcritpts} The function $\mathcal{L}_\pm(\cdot ;\bl_{0,\infty},\bl_{1,\infty})$ has a unique critical point $z_0:=z_0^\pm(\bl_{0,\infty},\bl_{1,\infty})$, and this point satisfies the equation:
\begin{equation}\label{criteqhol}
\star e^{-\frac{i\arg(\kappa^\infty(\lambda_K))}{2}}e^{2(\bl_{0,\infty}+z_0)} +e^{\bl_{0,\infty}+z_0}-1=0.
\end{equation}
Moreover, the image of the map 
$$\fonc{\Xi}{\mathcal{A}'(T,b)}{\mc^2}{(\bl_{0,\infty}^{\bu}, \bl_{1,\infty}^{\bu},\bl_{0,\infty}^{\bv}, \bl_{1,\infty}^{\bv})}{\left(e^{\bl_{0,\infty}^{\bu}+z_0^-(\bl_{0,\infty}^{\bu}, \bl_{1,\infty}^{\bu})},e^{\bl_{0,\infty}^{\bv}+z_0^+(\bl_{0,\infty}^{\bv}, \bl_{1,\infty}^{\bv})}\right)}$$
is contained in $G(T,b)$ and covers a neighborhood of $w_{hyp}$, the geometric solution to the gluing equations.
\end{lem}
\proof We have
\begin{equation}\label{crit1} 
\begin{array}{l}
(\partial_z\mathcal{L}_-)(z_0;\bl_{0,\infty},\bl_{1,\infty}) =0\Longleftrightarrow 2z_0-\Log(1-e^{\bl_{0,\infty}+z_0})- \bl_{1,\infty}=0 ,\\
(\partial_z\mathcal{L}_+)(z_0;\bl_{0,\infty},\bl_{1,\infty}) =0\Longleftrightarrow 2z_0-\Log(1-e^{\bl_{0,\infty}+z_0})+2i\pi-\bl_{1,\infty}=0 .
\end{array}
\end{equation}
Passing to the exponential, both equations yield the following degree two equation in $e^{z_0}$:
\begin{equation}\label{criteq}
e^{-\bl_{1,\infty}}e^{2z_0} +e^{\bl_{0,\infty}}e^{z_0}-1=0.
\end{equation}
The two roots have product $-e^{\bl_{1,\infty}}$, hence their arguments sum to $\mathfrak{I}(\bl_{1,\infty})+\pi$ mod$(2\pi)$. On the other hand, the first equation of \eqref{crit1} imposes $2\mathfrak{I}(z_0)\in ]-\pi+\mathfrak{I}(\bl_{1,\infty});\pi+\mathfrak{I}(\bl_{1,\infty})]$, and the second imposes $2\mathfrak{I}(z_0)\in ]-3\pi+\mathfrak{I}(\bl_{1,\infty});-\pi+\mathfrak{I}(\bl_{1,\infty})]$. Combining the two facts, we see that $\mathcal{L}_+(z;\bl_{0,\infty},\bl_{1,\infty})$ and $\mathcal{L}_-(z;\bl_{0,\infty},\bl_{1,\infty})$ have exactly one critical point for every fixed values of $\bl_{0,\infty}$ and $\bl_{1,\infty}$. The identity \eqref{criteqhol} follows immediately from \eqref{criteq} and the second relation in \eqref{asycurvelog2}. This proves the first claim. \\
From \eqref{criteq} we get
\begin{equation}\label{criteq2oct24}
e^{-\bl_{1,\infty}-2\bl_{0,\infty}} = e^{-2(z_0+\bl_{0,\infty})}(1-e^{z_0+\bl_{0,\infty}}).
\end{equation}
It is also an immediate check that this identity and the first of \eqref{asycurvelog} imply that $(u_0,v_0):=(e^{\bl_0^{\bu,\infty}+z_0^\pm(\bl_0^{\bu,\infty}, \bl_1^{\bu,\infty})},e^{\bl_0^{\bv,\infty}+z_0^\pm(\bl_0^{\bv,\infty}, \bl_1^{\bv,\infty})})$ is a solution of the defining equation \eqref{eqedgeu0v0} of $G(T,b)$. This proves the second claim.
 
 Finally note that $\Xi(-2\pi i, \pi i,2\pi i, -\pi i) = (e^{\pi i/3}, e^{- \pi i/3})$. A straightforward computation shows that the Jacobian of $\Xi$ at $(-2\pi i, \pi i,2\pi i, -\pi i)$ has real rank $2$. Since $G(T,b)$ is a curve and $(e^{\pi i/3}, e^{- \pi i/3})$ is a smooth point, this shows that $\Xi$ is a submersion, and the last claim follows.\cvd
\medskip

We note the following simple caracterization of critical points $z_0$ which are saddle points of order $1$:
$$\partial_z^2\mathcal{L}_\pm(z_0;\bl_{0,\infty},\bl_{1,\infty}) = 1+(1-e^{{\rm l}_0 +z_0})^{-1} \ne 0.$$

\subsection{The saddle-point method}\label{sub:SPM:PV} Let $\lambda$ denote a piecewise continuously differentiable path contained in an open
set $\Delta \subset \mathbb{C}$, and let $z \mapsto g(z)$ and $z \mapsto f(z)$ denote two holomorphic functions on $\Delta$.
We consider the asymptotics of the integral
\begin{equation}\label{intIndec25}
I_n := \int_{\lambda} g(z) e^{n (-i) f(z)} dz\quad , \ n\to +\infty.
\end{equation}
Let us recall a concise version of the one-dimensional saddle point method theorem (see also \cite{O'Sul}). Note that usually the condition (iii) below is written as a maximum of the real part of the potential function. Here we consider $n(-i)f(z)$ instead of $nf(z)$, so we look for a maximum of the imaginary part instead (this is mostly in order to work with ${\rm Li}_2$ instead of $i{\rm Li}_2$ in the following sections).

\begin{prop}[\cite{PV}, Section 2]\label{prop:SPM:PV} Assume there exists an interior point $z_0 \in \lambda$ such that 
	\begin{enumerate}[label=(\roman*)]
\item $f'(z_0)=0$,
\item $f''(z_0) \neq 0$,
\item $\Im f(z) < \Im f(z_0)$ for all $z \neq z_0$ on the closure of $\lambda$ in $\mathbb{C}\cup\{\infty\}$,
	\end{enumerate}
and $g(z_0)\neq 0$. Then we have:
$$ I_n = \dfrac{\sqrt{2\pi}}{\sqrt{n}} e^{i \theta} g(z_0) \dfrac{1}{\sqrt{|f''(z_0)|}}
e^{n (-i) f(z_0)} \left ( 1+ o_{n \to \infty}(1) \right ),
$$
for some explicit $\theta \in \mathbb{R}$. In particular 
$$ \lim_{n\ra +\infty}  \dfrac{1}{n}\Log \left \vert I_n \right \vert = \Im f(z_0).
$$
\end{prop}

In the following sections we will use this result to study the integrals $I_{\pm,N}(\bl_{0,\infty},\bl_{1,\infty})$ in \eqref{I+} for the particular values $\bl_{0,\infty}=-2i \pi$ and $\bl_{1,\infty}\in \{i \pi,2i\pi\}$ obtained in Section \ref{loglimmars25}. Therefore we set $$n:=\frac{N}{2\pi}$$
(as usual with $N\geq 3$ odd); $g$ will be the constant map equal to $1$, and $f$ one of the following two functions, introduced in \eqref{defL+-}:
\begin{equation}\label{defL+-modif}
	f(z) = \left\lbrace
	\begin{array}{ll}
		f_-(z):= \mathcal{L}_-(z;\bl_{0,\infty},\bl_{1,\infty}) := {\rm Li}_2(e^{\bl_{0,\infty} + z})+z^2- \bl_{1,\infty} z& 
		\\
		f_+(z) := \mathcal{L}_+(z;\bl_{0,\infty},\bl_{1,\infty}) := {\rm Li}_2(e^{\bl_{0,\infty} + z})+z^2 -(\bl_{1,\infty}-2i\pi) z . 
		\end{array}\right.
\end{equation}
With these notations, and for $\textstyle s_N^-,s_N^+\in ]0,\frac{2\pi}{N}[$, we are going to describe some asymptotics of the integrals
\begin{equation}\label{notInjui25} I_{n,\pm} :=  \int_{i[s_N^-,2\pi-s_N^+]}  e^{n(-i) f_{\pm}(z)} dz.
\end{equation}
Note that the endpoints $s_N^-,s_N^+$ of the contour are moving to $0$ or $2\pi i$ as $N\rightarrow +\infty$, which lie on branch cuts of the functions $f_\pm$. Eventually, our analysis will be achieved by taking $s_N^-$ not too small, namely (see the end of Section \ref{smallIsN})
\begin{equation}\label{sNvalue}
s_N^- := \alpha \frac{\pi}{N},\ \alpha\in ]0,1[.
\end{equation}
We begin in Section \ref{sub:small:vert:int} with the asymptotics of integrals along small contours in $i\mr$, near $0$ or $2i\pi$. These asymptotics will be used several times in the following sections.

\subsection{Asymptotics of integrals $\textstyle \int e^{\scriptscriptstyle{\frac{N}{2\pi} {\rm Cl}_2(t)}}dt$ }\label{sub:small:vert:int}

\subsubsection{Along $[\pi,2\pi]+2k\pi$, $k\in \mz$.}\label{smallpijuil25}

Recall that Clausen's function
\begin{equation}\label{Clausenf} {\rm Cl}_2(t) = -\int_0^t \Log\vert 2\sin(t/2)\vert dt
\end{equation}
is $2\pi$-periodic and odd (\cite{Z}). Therefore it is enough to make our analysis over $[-\pi,0]$. Since ${\rm Cl}'_2(t)= -\Log\vert 2\sin(t/2)\vert$ for all $t \notin 2\pi \Z$, we have ${\rm Cl}''_2(t)= -\cot(t/2)/2$ for the same $t$, which implies that the Clausen function is convex on $[-\pi,0]$. It vanishes at $-\pi$ and $0$, is negative on $]-\pi,0[$, and has a strict minimum ${\rm Cl}_2\left (-\pi/3\right )= -1.10149...$ that we will denote $-m$.

Since the graph of a convex function lies under its chords, it follows
$${\rm Cl}_2(t)\leqslant -\dfrac{3 m}{2\pi}t-\dfrac{3m}{2}\ \text{for\ all}\ t \in \left [-\pi,-\frac{\pi}{3}\right ], \ \text{and}\ 
{\rm Cl}_2(t)\leqslant \dfrac{3 m}{\pi}t \ \text{for\ all}\ t \in \left [-\frac{\pi}{3},0\right ].$$

Hence we can compute the upper bound:
\begin{align*}
0 < \int_{-\pi}^0 e^{\frac{N}{2\pi} {\rm Cl}_2(t)}dt 
&= \int_{-\pi}^{-\pi/3} e^{\frac{N}{2\pi} {\rm Cl}_2(t)}dt 
+\int_{-\pi/3}^0 e^{\frac{N}{2\pi} {\rm Cl}_2(t)}dt \\
&\leqslant \int_{-\pi}^{-\pi/3} e^{\frac{N}{2\pi} \left (-\frac{3 m}{2\pi}t-\frac{3m}{2}\right )}dt 
+\int_{-\pi/3}^0 e^{\frac{N}{2\pi} \frac{3 m}{\pi}t}dt \\
&= 
e^{-\frac{3Nm}{4\pi}}
\frac{(-4)\pi^2}{3 m N}
\left (e^{\frac{Nm}{4\pi}}- e^{\frac{3Nm}{4\pi}}\right )
+ \frac{2\pi^2}{3 m N}\left (1- e^{-\frac{Nm}{2\pi}}\right )\\
&= \frac{2\pi^2}{ m N}\left (1+o_{N\to \infty}(1)\right )\\
&= o_{N\to \infty}\left (\frac{1}{\sqrt{N}}\right ).
\end{align*}

Since $e^{\frac{N}{2\pi} {\rm Cl}_2}\geqslant 0$, this generalizes to
$$\int_{I} e^{\frac{N}{2\pi} {\rm Cl}_2(t)}dt= o_{N\to \infty}\left (\frac{1}{\sqrt{N}}\right )$$
for any sub-part $I$ of $[-\pi,0]$ or $[\pi,2\pi]
$.

\subsubsection{Along $[0,s_N^-]$} \label{smallIsN} The Clausen function is positive and increasing on $[0,\pi/3]$, so to get an upper bound it is enough to discuss the asymptotical behaviour of the constant integral 
$$\int_{[0, s_N^-]} e^{\frac{N}{2\pi} {\rm Cl}_2(s_N^-)} = s_N^- \cdot e^{\frac{N}{2\pi} {\rm Cl}_2(s_N^-)}. $$ 
We want to prove that $s_N^- \cdot e^{\frac{N}{2\pi} {\rm Cl}_2(s_N^-)} = o_{N \to \infty}\left ( \frac{1}{\sqrt{N}}\right )$.

Recall that $\textstyle s_N^-\in ]0,\frac{2\pi}{N}[$. For all $x\in (0,1)$ we have
$$ \frac{x}{2} < x-\frac{x}{3} < x - \frac{x^3}{3} < \sin(x) < x.$$
Thus, for any $N\geqslant 7$ and any $t\in (0,s_N^-]$ we have
$$ \frac{t}{2} = \frac{2(\frac{t}{2})}{2} < 2 \sin\left (\frac{t}{2}\right )< 2\left (\frac{t}{2}\right ) = t,$$
and therefore $\Log(t)-\Log(2) < \Log \left (2 \sin\left (t/2\right )\right ) < \Log(t)$. By integrating the opposite of this inequality, we obtain for all $t \in (0,s_N^-]$,
$$t -t\Log(t) <  {\rm Cl}_2(t) < \Log(2) t + t - t \Log(t).$$
(This notably implies that ${\rm Cl}_2(t) \sim_{t \to 0^+} -t\Log(t)$ but we will use the more precise double inequality in what follows.) By evaluating the previous double inequality at $t=s_N^-$ we get
$$-s_N\Log(s_N^-) <  {\rm Cl}_2(s_N^-) < \Log(2) s_N^- + s_N^- - s_N^- \Log(s_N^-)$$
and thus 
$$ s_N^- \cdot e^{\frac{N}{2\pi} {\rm Cl}_2(s_N^-)} <  s_N^-\cdot e^{\frac{N}{2\pi}s_N\left (\Log(2)+1-\Log(s_N^-)\right )}.$$
In particular, if $s_N^-=\alpha \frac{\pi}{N}$, where $\alpha \in ]0,1[$, then we get
$$ s_N^- \cdot e^{\frac{N}{2\pi} {\rm Cl}_2(s_N^-)} <  \alpha \frac{\pi}{N} \cdot e^{\frac{\alpha}{2}\left (\Log(2)+1+\Log(N) - \Log(\pi)-\Log(\alpha)\right )}= \dfrac{\text{Constant}}{N^{1-\frac{\alpha}{2}}}=o_{N\to \infty}\left (\frac{1}{\sqrt{N}}\right ).$$

\subsection{The limit in case (a)}\label{sub:SPM:a} Recall \eqref{l0inftymars25}-\eqref{l1inftymars25} and \eqref{defL+-modif}-\eqref{notInjui25}. In case (a) we have $\bl_{0,\infty}=-2i\pi$ and $\bl_{1,\infty}=i\pi$, so
\begin{equation}\label{defL+-modifcasea}
\left\lbrace
	\begin{array}{ll}
		f_-(z)= \mathcal{L}_-(z;-2i\pi,i\pi) := {\rm Li}_2(e^{z})+z^2 - i\pi z,& 
		\\
		f_+(z) = \mathcal{L}_+(z;-2i\pi,i\pi) := {\rm Li}_2(e^{z})+z^2 +i\pi  z.
		\end{array}\right.
\end{equation}

We will prove, using Proposition \ref{prop:SPM:PV} for $I_{n,-}$, that
\begin{equation}\label{asyIjui25}
\lim_{n\ra +\infty}  \dfrac{1}{n}\Log \left \vert I_{n,-} \right \vert = \dfrac{1}{2} {\rm Vol}(M)
  \ \ \text{and} \ \ 
I_{n,+} = O_{N\ra \infty}(1). 
\end{equation}
By Lemma \ref{lem:PM:SE} this will imply asymptotics of linear combinations of $I_{n,-}$ and $I_{n,+}$, such as
\begin{equation}\label{asyIjui25b}
\lim_{n\ra +\infty}  \dfrac{1}{n}\Log \left \vert I_{n,-} \pm I_{n,+} \right \vert = \dfrac{1}{2} {\rm Vol}(M).
\end{equation}

\subsubsection{Case (a), integral $I_{n,-}$}\label{sub:a:-}
Let us prove that $\lim_{n\ra +\infty}  \dfrac{1}{n}\Log \left \vert I_{n,-} \right \vert = \dfrac{1}{2} {\rm Vol}(M)$.

First we show that $f_-(z)=  {\rm Li}_2(e^{z})+z^2 - i\pi z$ and $\textstyle z_0:= {i \frac{\pi}{3}}$ satisfy the properties of Proposition \ref{prop:SPM:PV} for the contour $\lambda=i[\eta,2\pi-\eta']$, where $\eta,\eta' \in(0,1)$.

We have $\partial_z {\rm Li}_2(z) = -\Log(1-z)/z$, so $\partial_z {\rm Li}_2(e^{z}) = -\Log(1-e^{z})e^{-z}.e^{z}= -\Log(1-e^{z})$, and therefore
$$(f_-)'(z)= 2z -\Log(1-e^{z}) - i \pi,\ \text{and}\ (f_-)''(z)= 2 + \frac{e^z}{1-e^z}= 1 + \frac{1}{1-e^z}.$$ 
It follows $(f_-)'(z_0)=  (2i\pi/3)  -\Log(1-e^{i\pi/3}) - i \pi= (-i \pi/3) - \Log(e^{-i \frac{\pi}{3}})=0$, which checks (i). Also $(f_-)''(z_0)= 1 + (1-e^{i \frac{\pi}{3}})^{-1}=1+e^{i \frac{\pi}{3}} \neq 0$, which checks (ii).

Let us check property (iii). We parametrize the contour $\lambda=i[\eta,2\pi-\eta']$ by $z=i t$, where $t \in [\eta,2\pi-\eta']$. Let us prove that $[0,2\pi] \ni t \mapsto \Im(f_-)(it)$ has a strict maximum at $\textstyle\frac{\pi}{3}$.

Recall (\cite{Z}) that
\begin{equation}\label{BWdilogmars25}
\forall y \in \C, \ \mathfrak{I}({\rm Li}_2(y))   = -\arg(1-y)\Log\vert y\vert + D(y),
\end{equation}
where $\arg$ is the branch of the argument with values in $]-\pi,\pi]$, and $D(y)$ is the Bloch-Wigner dilogarithm, which computes the algebraic volume of an oriented hyperbolic ideal tetrahedron with cross--ratio modulus $y$. Hence we have
$$\Im(f_-)(it) = \Im\left ({\rm Li}_2(e^{it})+(it)^2 - i\pi it\right )
=\Im\left ({\rm Li}_2(e^{it})\right ) = D(e^{it})= {\rm Cl}_2(t),$$
where ${\rm Cl}_2$ is Clausen's function \eqref{Clausenf}. Since it achieves its strict local maxima at the points $\pi/3 + 2\pi n$, $n\in \mz$, (iii) follows.

We can now apply Proposition \ref{prop:SPM:PV}
and conclude that
 $$\lim_{n\ra +\infty}  \dfrac{1}{n}\Log \left \vert \int_{[i\eta,2i\pi-i\eta']}  e^{n(-i) f_{-}(z)} dz \right \vert 
 = \Im (f_-)(z_0) = D\left (e^{i\frac{\pi}{3}}\right )
 = \dfrac{1}{2} {\rm Vol}(M),$$
as the volume of the figure-eight knot complement is well-known to be twice the volume of a regular ideal tetrahedron, with angles $\textstyle \frac{\pi}{3}$.

Finally, consider the decomposition
\begin{align*}
	I_{n,-} &=  \int_{i[s_N^-,2\pi-s_N^+]}  e^{n(-i) f_{-}(z)} dz\\
	&=
	\int_{i[ s_N^-,\eta ]}  e^{n(-i) f_{-}(z)} dz
	+\int_{i[\eta,2\pi-\eta']}  e^{n(-i) f_{-}(z)} dz
	+\int_{i[2\pi-\eta',2\pi- s_N^+]}  e^{n(-i) f_{-}(z)} dz.
\end{align*}
By  the results of Section \ref{sub:small:vert:int}, an immediate triangle inequality shows that
the first and third integrals are $O_{n\to\infty}(\exp(n K))$ where $K>0$ is a value of ${\rm Cl}_2(t)$ outside of its strict maximum, and thus $ K<{\rm Vol}(M)/2$. Hence
 $$\lim_{n\ra +\infty}  \dfrac{1}{n}\Log \left \vert I_{n,-} \right \vert 
= \dfrac{1}{2} {\rm Vol}(M).$$

\subsubsection{Case (a), integral $I_{n,+}$}\label{sec:case(a)In+}
Let us prove that $I_{n,+} = O_{N\ra \infty}(1)$.

One can easily compute that the function $f_+(z)=  {\rm Li}_2(e^{z})+z^2 +i\pi  z$ admits a unique critical point at $\textstyle z=-i \frac{\pi}{3}$ (on which $\Im f_+ |_{i\R}$  has a minimum, not a maximum), but we will not use this.

We deform the contour $i[s_N^-,2\pi-s_N^+]$. Let $x$ be a large negative real number, and define $C^\infty$ as the path from $is_N^-$ to $2i\pi -is_N^+$ formed by the concatenation of the open oriented intervals $C_{s_N^-}:=i(s_N^-,0)$, $C^\infty_0:=(0,x)\subset \mr_-$, $C^\infty_x:=x+ i(0,2\pi)$, $C^\infty_{2\pi}:=(x,0)+i2\pi$, and $C'_{s_N^+}:=(2i\pi,2i\pi-is_N^+)$ (see Figure \ref{fig:contour:a:+}). Since $C^\infty$ and $i[s_N^-,2\pi-s_N^+]$ are homotopic relatively to their endpoints we have 
\begin{equation}\label{defneg}
	I_{n,+} = \int_{C^\infty} e^{\frac{N}{2\pi i}\mathcal{L}_+(z;-2\pi i,\pi i)}dz.
\end{equation}
\begin{figure}[!h]
\begin{tikzpicture}
\draw[color=lightgray,->] (-5,0)--(2,0);
\draw[color=lightgray,->] (0,-2)--(0,4);
\draw (-4,-.05) node {\Huge $\cdot$};
\draw (-4,-.25) node {$x$};
\draw (-4,-.05+2) node {\Huge $\cdot$};
\draw (-4,+.25+2) node {$x+2i\pi$};
\draw (+.025,-.05) node {\Huge $\cdot$};
\draw (+.25,-.25) node {$0$};
\draw (+.025,.4-.05) node {\Huge $\cdot$};
\draw (+.5,.4) node {$is_N^-$};
\draw (+.025,2-.4-.05) node {\Huge $\cdot$};
\draw (+1,2-.4) node {$2i\pi-is_N^+$};
\draw (+.025,2-.05) node {\Huge $\cdot$};
\draw (+.25+.2,2) node {$2i \pi$};
\draw[color=blue,thick] (0,0)--(-4,0);
\draw[color=blue] (-2,0-.3) node {$C^\infty_0$};
\draw[color=blue] (-2,0-.3-.5) node {$\Im(f_+)<0$};
\draw[color=blue,thick] (0,0+2)--(-4,0+2);
\draw[color=blue] (-2,0+2+.3) node {$C^\infty_{2\pi}$};
\draw[color=blue] (-2,0+2+.3+.5) node {$\Im(f_+)<0$};
\draw[color=blue,thick] (-4,0)--(-4,2);
\draw[color=blue] (-4-.5,1.2) node {$C^\infty_x$};
\draw[color=blue] (-4-1,.7) node {$\Im(f_+)<0$};
\draw[color=blue,thick] (0,2)--(0,2-.4);
\draw[color=blue] (-.5,1.6) node {$C'_{s_N^+}$};
\draw[color=red,thick] (0,0)--(0,.4);
\draw[color=red] (-.5,.4) node {$C_{s_N^-}$};
\end{tikzpicture}
\caption{The deformed contour $C^\infty=C_{s_N^-} \cup C^\infty_0 \cup C^\infty_x \cup C^\infty_{2\pi} \cup C'_{s_N^+}$ (blue parts are where $\Im(f_+)<0$ and the red part is where $\Im(f_+)$ is positive but very small)
	}\label{fig:contour:a:+}
\end{figure}
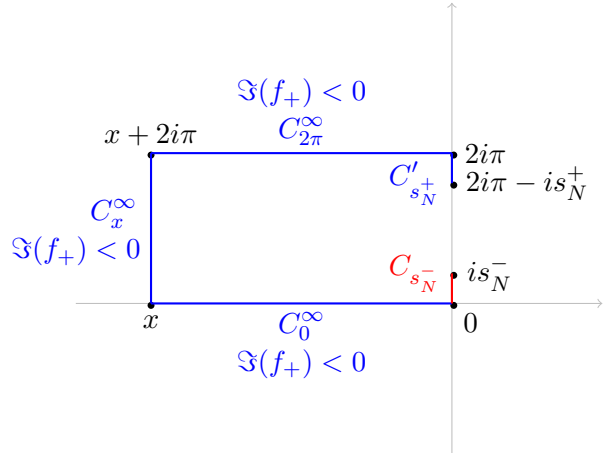

From the formula $f_+(z)=\mathcal{L}_+(z;-2\pi i,\pi i) = {\rm Li}_2(e^{z})+z^2+\pi iz$ we compute
$$\Im(f_+(z))= \Im({\rm Li}_2(e^{z}))+2 \Re(z)\Im(z) + \pi \Re(z).$$

For $z \in C_0^\infty \subset \R_-$ we thus have $\Im(f_+(z))= \Im({\rm Li}_2(e^{z}))+ \pi \Re(z)= \pi \Re(z)<0$ since $\Re(z)<0$ and the dilogarithm of a real number $<1$ is real.

For $z \in C_{2\pi}^\infty \subset \R_-+ i2\pi$ we have $\Im(f_+(z))= \Im({\rm Li}_2(e^{z}))+5\pi\Re(z)=5\pi\Re(z) <0$ since again $\Re(z)<0$ and the dilogarithm of  $e^z \in \exp(]-\infty;0[ + 2i \pi) \subset \R\setminus ]1;+\infty[$ is real.

For $z=x+it \in C^\infty_x$ we have $\Im(f_+(x+it))= \Im({\rm Li}_2(e^{x+it}))+2 x t + \pi x.$ Moreover ${\rm Li}_2(y)\sim y$ as $y\ra 0$ in $\mc$, and 
 therefore there exists $x_0 <0$ such that
  for every $t\in [0,2\pi]$ and $x\leq x_0$ we have
  $\textstyle |{\rm Li}_2(e^{x+it})|<\frac{1}{2}$.
  Hence, taking for instance $x \leqslant x_0-1$, for every $t\in [0,2\pi]$ we have
$$\Im(f_+(x+it))= \Im({\rm Li}_2(e^{x+it}))+2 x t + \pi x
\leqslant \dfrac{1}{2} + \pi(-1) <0.
$$
We conclude that for $|x|$ large enough, $\Im(f_+)<0$ on the contour $C^\infty_x$.

Now we remark that, having fixed for example $x=x_0-1<0$, and denoting by $\gamma\colon [0,1]\to \mc$ a parametrization of the path $C^\infty_0 \cup C^\infty_x \cup C^\infty_{2\pi} $, we have
\begin{align*}
	\left \vert \int_{C^\infty} e^{\frac{N}{2\pi i}f_+(z)}dz\right \vert
& \leqslant \int_{0}^1 |e^{\frac{N}{2\pi i}f_+(\gamma(t))} \gamma'(t)|dt \\
& \leqslant \int_{0}^1 e^{\frac{N}{2\pi}\Im f_+(\gamma(t))} |\gamma'(t)| dt\\
& \leqslant \int_{0}^1 e^{\frac{N}{2\pi}0}|\gamma'(t)| dt=2|x|+2\pi = O_{N\ra \infty}(1).
\end{align*}

Finally, the remaining two integrals on the small vertical contours $C_{s_N^-}, C'_{s_N^+}$ are also $O_{N\ra \infty}(1)$ from the conclusions of Section \ref{sub:small:vert:int}. Hence $I_{n,+} = O_{N\ra \infty}(1)$.

\subsection{The limit in case (b)}\label{sub:SPM:b}
For case (b) we have $\bl_{0,\infty}=-2i\pi$ and $\bl_{1,\infty}=2i\pi$. Hence
\begin{equation}\label{defL+-modifcaseb}
	\left\lbrace
	\begin{array}{ll}
		f_-(z)= \mathcal{L}_-(z;-2i\pi,2i\pi) := {\rm Li}_2(e^{z})+z^2 - 2i\pi z,& 
		\\
		f_+(z) = \mathcal{L}_+(z;-2i\pi,2i\pi) := {\rm Li}_2(e^{z})+z^2.& 
	\end{array}\right.
\end{equation}

We will again use Proposition \ref{prop:SPM:PV} to prove that $I_{n,-}$ and $I_{n,+}$ are 
of the form
\begin{equation}\label{asyIjui25c}
I_{n,-} = \dfrac{\text{Constant}_-}{\sqrt{N}}(e^{i r_+ N}+o_{N\to\infty}(1)), \ \
I_{n,+} = \dfrac{\text{Constant}_+}{\sqrt{N}}(e^{i r_- N}+o_{N\to\infty}(1)),
\end{equation}
where $\text{Constant}_-$ and $\text{Constant}_+$ are distinct non-zero complex numbers and $r_+,r_-\in \R$ are distinct as well (see \eqref{Cst+},\eqref{Cst-}).
In particular, this will imply that, e.g.,
\begin{equation}\label{asyIjui25d}
\lim_{n\ra +\infty}  \dfrac{1}{n}\Log \left \vert I_{n,-} \right \vert =
\lim_{n\ra +\infty}  \dfrac{1}{n}\Log \left \vert I_{n,+} \right \vert = 
\lim_{n\ra +\infty}  \dfrac{1}{n}\Log \left \vert I_{n,-} \pm I_{n,+} \right \vert = 0.
\end{equation}

\subsubsection{Case (b), integral $I_{n,+}$}\label{sub:b:+:perron}
Let us study the asymptotics of $I_{n,+}=  \int_{i[s_N^-,2\pi-s_N^+]}  e^{\frac{N}{2i\pi} f_{+}(z)} dz$.

From the formula $f_+(z)=\mathcal{L}_+(z;-2\pi i,2\pi i) = {\rm Li}_2(e^{z})+z^2$
we compute
\begin{align*}\Im(f_+(z))&= \Im({\rm Li}_2(e^{z}))+2 \Re(z)\Im(z),\\
(f_+)'(z)& =2z- \Log(1-e^z),\\
(f_+)''(z)& =2 + \dfrac{e^z}{1-e^z} = 1 + \dfrac{1}{1-e^z},\\
f_+(x+it)& = {\rm Li}_2(e^{x+it})+(x+it)^2 ,\\
\dfrac{\partial}{\partial x} \Im (f_+(x+it)) & = \Im \left (\dfrac{\partial}{\partial x}f_+(x+it)\right ) = 2t - \Im \left (\Log(1-e^{x+it})\right )
= 2t - \arg\left (1-e^{x+it}\right ),\\
\dfrac{\partial}{\partial t} \Im (f_+(x+it)) & = \Im \left (\dfrac{\partial}{\partial t}f_+(x+it)\right ) = 2x - \Re \left (\Log(1-e^{x+it})\right )
= 2x - \Log \vert 1-e^{x+it} \vert.
\end{align*}
One can easily compute that $(f_+)'(z_0)=0$ if and only if 
$$z_0 = a:= \Log ((-1+\sqrt{5})/{2}) \approx -0.48.$$
We will deform the contour $i[s_N^-,2\pi-s_N^+]$ (in blue in Figure \ref{fig:contour:b:+:perron}) to a new contour $C'$ (in dotted blue in Figure \ref{fig:contour:b:+:perron}) on which $\Im(f_+)$ is negative everywhere but in the point $a$ where  $\Im(f_+)(a)=0$ and on the small subcontour $[0,is_N]$ where it is nonnegative. Let us prove this is possible.

We construct $C'$ in six parts:
\begin{itemize}
\item $C'_1$ is vertical, going from $i s_N^-$ to $0$.
\item $C'_2$ is vertical, going from $0$ to $-i \eta$.
\item $C'_3$ is almost horizontal, going from $-i \eta$ to $a$.
\item $C'_4$ is vertical, going from $a$ to $a+2i\pi-i\eta'$.
\item $C'_5$ is horizontal, going from $a+2i\pi-i\eta'$ to $2i\pi-i\eta'$.
\item $C'_6$ is vertical, going from $2i\pi-i\eta'$ to $2i\pi-is_N^+$.
\end{itemize}

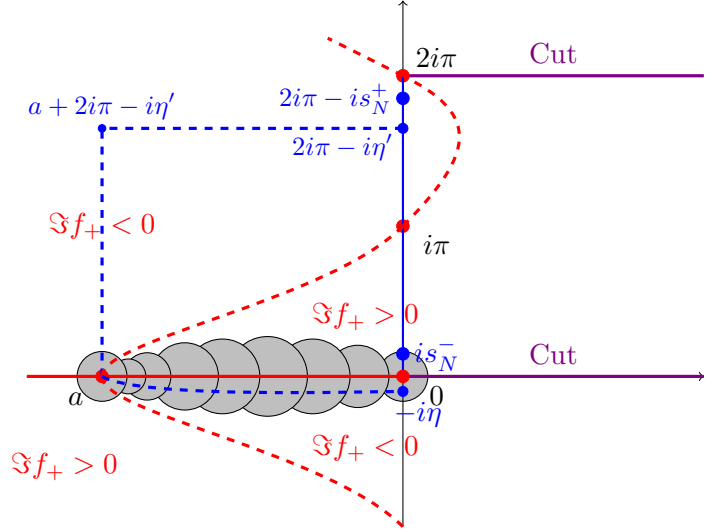
\begin{figure}[!h]
	\begin{tikzpicture}
		
		\draw[black,fill=gray!50,opacity=0.5, ] (0,0) circle (2ex); 		
		\draw[black,fill=gray!50,opacity=0.5, ] (-.6,0) circle (2.5ex);  
		\draw[black,fill=gray!50,opacity=0.5, ] (-1.2,0) circle (3ex);  
		\draw[black,fill=gray!50,opacity=0.5, ] (-1.8,0) circle (3.2ex);  
		\draw[black,fill=gray!50,opacity=0.5, ] (-2.4,0) circle (3ex);  
		\draw[black,fill=gray!50,opacity=0.5, ] (-2.9,0) circle (2.7ex);  
		\draw[black,fill=gray!50,opacity=0.5, ] (-3.4,0) circle (1.9ex);  
		\draw[black,fill=gray!50,opacity=0.5, ] (-3.65,0) circle (1.4ex);  
		\draw[black,fill=gray!50,opacity=0.5, ] (-4,0) circle (2ex);

		\draw[red] (-.5,.8) node {$\Im f_+ >0$};
		\draw[red] (-.5,-.95) node {$\Im f_+ <0$};		
		\draw[red] (-4.5,-1.2) node {$\Im f_+ >0$};		
		\draw[red] (-4,2) node {$\Im f_+ <0$};		
		
		\draw[color=black,->] (-5,0)--(4,0);		
		\draw[color=black,->] (0,-2)--(0,5);
		
		\draw[color=violet , very thick ] (0,0)--(4,0);		
		\draw[color=violet ] (2,.3) node {Cut};
		
		\draw[color=violet , very thick ] (0,0+4)--(4,0+4);		
		\draw[color=violet ] (2,.3+4) node {Cut};		
		
		\draw (-4-.35,-.3) node {$a$};
		\draw[red,fill=red] (-4,0) circle (.5ex); 
		\draw (-4,+.25+4-.65) node {\small \color{blue} $a+2i\pi-i\eta'$};
		\draw[blue,fill=blue] (-4,4-.7) circle (.3ex); 		
		\draw (+.25+.2,-.25) node {$0$};
		\draw[red,fill=red] (0,2) circle (.5ex); 
		\draw (+.25+.2,0.3) node {\color{blue}$i s_N^-$};
		\draw[blue,fill=blue] (0,.3) circle (.5ex); 		
		\draw (+.25+.2,2-.25) node {$i \pi$};
		\draw[red,fill=red] (0,4) circle (.5ex); 
		\draw (-.9,4-.3) node {\small \color{blue}$2i\pi - i s_N^+$};
		\draw[blue,fill=blue] (0,4-.3) circle (.5ex); 		
		\draw (-.8,4-.95) node {\small \color{blue}$2i\pi - i \eta'$};
		\draw[blue,fill=blue] (0,4-.7) circle (.4ex); 		
		\draw (+.25+.2,4+.25) node {$2i \pi$};
		\draw[color=red,very thick] (0,0)--(-5,0);
		\draw[color=blue,very thick,dashed] (0,0+4-.7)--(-4,0+4-.7);
		\draw[color=blue,very thick,dashed ] (-4,0)--(-4,4-.7);
		\draw[color=blue,thick] (0,0)--(0,4);
		\draw[color=red, very thick, dashed] (-4,0) ..controls +(0,-.5) and +(-1,1).. (0,-2);
		\draw[color=red,very thick, dashed] (-4,0) ..controls +(0,.5) and +(-1,-1).. (0,2);
		\draw[color=red,very thick, dashed] (0,2) ..controls +(1,1) and +(1,-.5).. (0,4);
		\draw[color=red, very thick,dashed] (0,4) ..controls +(-1,.5) and +(1,-.5).. (-1,4.5);
		
		\draw[color=blue , very thick, dashed] (-4,0) ..controls +(0.5,-.3) and +(-.5,0).. (0,-.2);
		\draw[blue,fill=blue] (0,-.2) circle (.4ex); 
		\draw (+.2,-0.5) node {\color{blue}$-i \eta $};		
		
		\draw[red,fill=red] (0,0) circle (.5ex);

	\end{tikzpicture}
	\caption{The deformed contour (in dashed blue) through zones where $\Im f_+<0$ and the point $a$ (in red) where $\Im f_+=0$.}\label{fig:contour:b:+:perron}
\end{figure}

First,  $\Im(f_+)$ is zero on the half-line $(-\infty,0]$ since ${\rm Li}_2$ takes real values on $[0,1]$. Moreover,  $\Im(f_+)(z)$ vanishes at $z=2i\pi$ (and also at $z=i\pi$ but we will not use this). The set of points where $\Im(f_+)$ vanishes is drawn in red in Figure \ref{fig:contour:b:+:perron}; the red dotted line's general form was computed via \textit{Mathematica}, but we will not use its exact position in our arguments.
\smallskip

\underline{Defining $C'_1$, and asymptotics:} We define $C'_1:=[0,i s_N^-]$ oriented downwards. Recall that for all $t\in \R$ and $z:=it$ we have $\Im({\rm Li}_2(e^z))={\rm Cl}_2(t)$, thus $\Im(f_+(it))={\rm Cl}_2(t)$. From Section \ref{smallIsN} we have:
$$ \int_{C'_1} e^{\frac{N}{2 i \pi}f_+(z)}dz =
\int_{0}^{s_N^-} e^{\frac{N}{2 \pi}\Im(f_+(it))}dt=
\int_{0}^{s_N^-}e^{\frac{N}{2 \pi}{\rm Cl}_2(t)}dt = O_{N\to \infty}\left (\frac{1}{N^{1-\frac{\alpha}{2}}}\right ) = o_{N\to \infty}\left (\frac{1}{\sqrt{N}}\right ).$$

\underline{Defining $C'_2$, and asymptotics:} We define $C'_2:=[-i\eta,0]$ oriented downwards. From Section \ref{smallpijuil25}, since $[-\eta,0] \subset [-\pi,0]$ we have:
$$ \left |\int_{C'_2} e^{\frac{N}{2 i \pi}f_+(z)}dz\right | \leqslant
\int_{-\eta}^{0} e^{\frac{N}{2 \pi}\Im(f_+(it))}dt=
\int_{-\eta}^{0}e^{\frac{N}{2 \pi}{\rm Cl}_2(t)}dt = o_{N\to \infty}\left (\frac{1}{\sqrt{N}}\right ).$$

\underline{Defining $C'_3$:}
The critical point $z=a$ of $f_+$ is simple, since $(f_+)''(z)$ only vanishes at $z=\Log(2)$. Thus, by standard complex analysis, there is a disk neighbourhood $D_a$ of $a$ (drawn as a grey disk) split into four zones by two transverse level lines where $\Im(f_+)=0$, and $\Im(f_+)$ is alternatively positive and negative in these four zones by circling around $a$. We already know that one level line is the horizontal line (drawn in full red), thus the other level line (part of the dotted red curve) crosses transversely this horizontal line.

Now, for any point $x \in (a,0]$ we have
$$\dfrac{\partial}{\partial t} \Im (f_+(x+it))\vert_{t=0} =  2x - \Log \vert 1-e^{x} \vert
=  2x - \Log (1-e^{x} )>0,$$
by definition of $a$ and the fact that
$ 2x - \Log \vert 1-e^{x} \vert>0 \Leftrightarrow (e^x)^2+(e^x)-1>0$.
Hence, $\Im(f_+)$ is increasing on every small enough vertical line (oriented upwards) above every $x\in(a,0]$. This confirms that a negative zone for $\Im(f_+)$ is below the red horizontal line and a positive zone is above it.

We can now define the part $C'_3$ of the new contour $C'$ (in Figure \ref{fig:contour:b:+:perron} this is the part which is between $-i\eta$ and $a$, oriented from right to left, where the value of $\eta>0$ will be decided in what follows). For any $x \in (a,0]$, define a small enough disk $D_x$ around $x$ so that the real line cuts $D_x$ into two half-disks, where $\Im(f_+)$ is negative on the bottom half-disk and positive on the upper half-disk. Such a small disk exists since $\textstyle \frac{\partial}{\partial t} \Im (f_+(x+it))\vert_{t=0}>0$.

The union $\textstyle D_a \cup \bigcup_{a<x \leqslant 0} D_x$ covers the compact set $[a,0]$, so there is a finite family of points $x_0=a<x_1<x_2<\ldots<x_p=0$ such that $\textstyle D_a \cup \bigcup_{1 \leqslant i \leqslant p} D_{x_i}$ covers $[a,0]$. Take $\eta>0$ small enough to be smaller than the radius of $D_{x_p}=D_0$.
 We can now define the part $C'_3$ as a smooth curve going from $-i\eta$ to $a$ while remaining in all bottom half-disks; it can even be chosen outside of $D_a$ as an horizontal line of imaginary part $-\eta$ (thus higher than the opposite of half the minimum of all radii of $D_{x_i}$), and smoothed out afterwards.

\smallskip

\underline{Defining $C'_4$:}
We define $C'_4:= [a,a+2i\pi-i\eta']$, oriented upwards. Let us prove that $\Im(f_+)<0$ on the vertical line $(a,a+2i\pi]$.
To do this we look for all $t \in (0,2\pi)$ at the real derivative
$$\dfrac{\partial}{\partial t} \Im (f_+(a+it))
=  2a - \Log \vert 1-e^{a+it} \vert
=  \Log (1-e^{a}) - \Log \vert 1-e^{a+it} \vert.
$$
We remark that this quantity is negative for all $t \in (0,2\pi)$, by a simple geometric argument: $1-e^a$ is the distance between $1$ and $e^a \in (0,1)$, which is smaller than the distance $\vert 1-e^{a+it} \vert$ between $1$ and any other point $e^{a+it}$ of the  circle of center $0$ and radius $e^a$ (since $1$ lives outside of the closed disk $\bar{D}(0,e^a)$). Hence, $\Im(f_+)$ decreases on $[a,a+2i\pi]$ oriented upwards, and thus $\Im(f_+)<0$ on the vertical line $(a,a+2i\pi]$.
\smallskip

\underline{Asymptotics on $C'_3 \cup C'_4$:}
From what precedes, we have  all the tools to apply the saddle point method (see Section \ref{sub:SPM:PV}) on the contour $\lambda=C'_3\cup C'_4$ for $g(z)=1, f(z)=f_+(z)$ and $z_0=a$. Since $N = 2\pi n$, we obtain:
$$ \int_{C'_3\cup C'_4} e^{\frac{N}{2 i \pi}f_+(z)}dz = \dfrac{{2\pi}}{\sqrt{N}} e^{i \theta}\cdot 1 \cdot \dfrac{1}{\sqrt{|f_+''(a)|}}
e^{\frac{N}{2i\pi} f_+(a)} \left ( 1+ o_{N \to \infty}(1) \right ).$$
For the curious reader, we remark that (\cite{Z}) $${\rm Li}_2(e^a)={\rm Li}_2\left (\frac{\sqrt{5}-1}{2}\right )= \frac{\pi^2}{10}-\Log^2\left (\frac{\sqrt{5}+1}{2}\right ),$$
and thus $f_+(a)={\rm Li}_2(e^a) +a^2$ has the unexpectedly simple expression
$\textstyle f_+(a)= \frac{\pi^2}{10} $. For our purposes, we only need to know that this is a real number, i.e. $\Im(f_+(a))=0$ which we knew already.

Moreover, $(f_+)''(a)=1+\frac{1}{1-e^a}= \frac{5+\sqrt{5}}{2}$. Hence we have
$$ \int_{C'_3\cup C'_4} e^{\frac{N}{2 i \pi}f_+(z)}dz = \dfrac{\text{Constant}_+}{\sqrt{N}}(e^{i r_+ N}+o_{N\to\infty}(1)),$$
where $\text{Constant}_+ \in \C^*$ has modulus $2\pi \left (\frac{5+\sqrt{5}}{2}\right )^{-\frac{1}{2}}=3.303...$ and $r_+=-\frac{\pi}{20} \in \R$.
\smallskip

\underline{Defining $C'_5$, and asymptotics:} We define $C'_5:=[a+2i\pi-i\eta',2i\pi-i\eta']$ oriented from left to right.
Since $\partial_x \Im (f_+(x+it)) = 2t - \arg\left (1-e^{x+it}\right )$ for all $x+it$ in the domain of definition of $f_+$, and since $\arg$ lies in $(-\pi,\pi]$, we deduce that for any fixed $t>\pi/2$ the function $\left (x\mapsto \Im (f_+(x+it))\right )$ is strictly increasing on the horizontal line of altitude $t$ oriented from left to right.

In particular for $t=2\pi - \eta'$, $\Im (f_+)$ is maximal on $C'_5$ at its right endpoint $2i\pi-i\eta'$ where it takes the value ${\rm Cl}_2(2\pi-\eta')<0$.
Hence 
$$ \left |\int_{C'_5} e^{\frac{N}{2 i \pi}f_+(z)}dz\right | \leqslant
|a|e^{\frac{N}{2 \pi}{\rm Cl}_2(2\pi-\eta')} = o_{N\to \infty}\left (\frac{1}{\sqrt{N}}\right ).$$

\underline{Defining $C'_6$, and asymptotics:} We define $C'_6:=i[2\pi-\eta',2\pi-s_N^+]$ oriented upwards. From Section \ref{sub:small:vert:int}, since $[2\pi-\eta',2\pi-s_N^+] \subset [\pi,2\pi]$ we have:
$$ \left |\int_{C'_6} e^{\frac{N}{2 i \pi}f_+(z)}dz\right | \leqslant
\int_{2\pi-\eta'}^{2\pi-s_N^+} e^{\frac{N}{2 \pi}\Im(f_+(it))}dt=
\int_{2\pi-\eta'}^{2\pi-s_N^+} e^{\frac{N}{2 \pi}{\rm Cl}_2(t)}dt = o_{N\to \infty}\left (\frac{1}{\sqrt{N}}\right ).$$

This concludes the construction of the new contour $C'$, and by summing the integrals on each part, we obtain
$$ I_{n,+}=\int_{C'} e^{\frac{N}{2 i \pi}f_+(z)}dz = \dfrac{\text{Constant}_+}{\sqrt{N}}(e^{i r_+ N}+o_{N\to\infty}(1)),$$
where
\begin{equation}\label{Cst+}
\vert \mathrm{Constant}_+ \vert  = 2\pi \left (\frac{5+\sqrt{5}}{2}\right )^{-\frac{1}{2}}=3.303...,\ r_+=-\frac{\pi}{20}\in \R.
\end{equation} 

\subsubsection{Case (b), integral $I_{n,-}$}\label{sec:bIn-janv26} From the formula $f_-(z)=\mathcal{L}_-(z;-2\pi i,2i\pi) = {\rm Li}_2(e^{z})+z^2 -2i\pi z$ we compute
\begin{align*}\Im(f_-(z))&= \Im({\rm Li}_2(e^{z}))+2 \Re(z)\Im(z)-2\pi\Re(z),\\
(f_-)'(z)&=2z- \Log(1-e^z)-2i\pi,\\
(f_-)''(z)& =2 + \dfrac{e^z}{1-e^z} = 1 + \dfrac{1}{1-e^z},\\
f_-(x+it)&= {\rm Li}_2(e^{x+it})+(x+it)^2 -2i\pi(x+it),\\
\dfrac{\partial}{\partial x} \Im (f_-(x+it)) & = \Im \left (\dfrac{\partial}{\partial x}f_-(x+it)\right ) = 2(t-\pi) - \Im \left (\Log(1-e^{x+it})\right )
\\ &= 2(t-\pi) - \arg\left (1-e^{x+it}\right ),\\
\dfrac{\partial}{\partial t} \Im (f_-(x+it)) & = \Im \left (\dfrac{\partial}{\partial t}f_-(x+it)\right ) = 2x- \Re \left (\Log(1-e^{x+it})\right )
= 2x - \Log \vert 1-e^{x+it} \vert.
\end{align*}
One can easily compute that $(f_-)'(z_0)=0$ if and only if $$z_0 = b+i\pi:= \Log((1+\sqrt{5})/2)+i\pi \approx 0.48+i\pi.$$

As in the previous section, we will deform the contour $i[s_N^-,2\pi-s_N^+]$ (in blue in Figure \ref{fig:contourC':b:+}) to a new contour $C'$ (in dotted blue in Figure \ref{fig:contourC':b:+}). In Figure \ref{fig:contourC':b:+}, red dots and lines are points where one can easily prove that $\Im(f_-)=0$, and the red dotted curve is suggested by \textit{Mathematica}.

We construct $C'$ in five parts:
\begin{itemize}
	\item $C'_1$ is horizontal, going from $i s_N^-$ to $1+is_N^-$.
	\item $C'_2$ is vertical, going from $1+is_N^-$ to $1+i(\pi-\epsilon)$, where $\epsilon>0$ is small enough.
	\item $C'_3$ is almost horizontal, going from $1+i(\pi-\epsilon)$ to the saddle point $b+i\pi$.
	\item $C'_4$ is almost horizontal, going from the saddle point $b+i\pi$ to $i(\pi+\epsilon)$.
	\item $C'_5$ is vertical, going from $i(\pi+\epsilon)$ to $2i\pi-is_N^+$.
\end{itemize}

\begin{figure}[!h]
	\begin{tikzpicture}
	\begin{scope}[xshift=4cm,yshift=2cm]
		\draw[black,fill=gray!50,opacity=0.5, ] (0,0) circle (2ex); 		
		\draw[black,fill=gray!50,opacity=0.5, ] (-.6,0) circle (2.5ex);  
		\draw[black,fill=gray!50,opacity=0.5, ] (-1.2,0) circle (3ex);  
		\draw[black,fill=gray!50,opacity=0.5, ] (-1.8,0) circle (3.2ex);  
		\draw[black,fill=gray!50,opacity=0.5, ] (-2.4,0) circle (3ex);  
		\draw[black,fill=gray!50,opacity=0.5, ] (-2.9,0) circle (2.7ex);  
		\draw[black,fill=gray!50,opacity=0.5, ] (-3.4,0) circle (1.9ex);  
		\draw[black,fill=gray!50,opacity=0.5, ] (-3.65,0) circle (1.4ex);  
		\draw[black,fill=gray!50,opacity=0.5, ] (-4,0) circle (2ex);
	\end{scope}

		\draw[red] (1,3) node {$\Im f_- <0$};		
		\draw[red] (3.2,1.2) node {$\Im f_- <0$};				
		\draw[red] (3,3) node {$\Im f_- >0$};		
		\draw[red] (1,1) node {$\Im f_- >0$};				
%
%

		\draw[color=black,->] (-2,0)--(8,0);		
		\draw[color=black,->] (0,-2)--(0,5);
		
		\draw[color=violet , very thick ] (0,0)--(8,0);		
		\draw[color=violet ] (2,-.3) node {Cut};
		
		\draw[color=violet , very thick ] (0,0+4)--(8,0+4);		
		\draw[color=violet ] (2,.3+4) node {Cut};		
		
		\draw (2.25+.75,2+.3) node {$b+i\pi$};
		\draw[red,fill=red] (2.25,2) circle (.5ex); 
		\draw (+.25+.2,-.25) node {$0$};
		\draw[red,fill=red] (0,2) circle (.5ex); 
		\draw (-.25-.2,0.3) node {\color{blue}$i s_N^-$};
		\draw[blue,fill=blue] (0,.3) circle (.5ex); 		
		\draw (-.25-.2,2-.25) node {$i \pi$};
		\draw[red,fill=red] (0,4) circle (.5ex); 
		\draw (-.9,4-.3) node {\small \color{blue}$2i\pi - i s_N^+$};
		\draw[blue,fill=blue] (0,4-.3) circle (.5ex); 		
		\draw (+.25+.2,4+.25) node {$2i \pi$};
		\draw[color=red,very thick] (0,2)--(8,2);
		\draw[color=blue,thick] (0,0)--(0,4);
		\draw[color=red, very thick, dashed] (0,0) ..controls +(3,0) and +(3,0).. (0,4);

		
		\draw[red,fill=red] (0,0) circle (.5ex); 
		
		\draw[color=blue,very thick,dashed] (0,0.3)--(4,.3);
		
		\draw (4+.9,0.3) node {\color{blue}$1+i s_N^-$};
		\draw[blue,fill=blue] (4,.3) circle (.5ex); 	

		\draw[color=blue,very thick,dashed] (4,.3)--(4,2-.15);		
		\draw[blue,fill=blue] (4,2-.15) circle (.3ex); 			
		\draw (4+1.1,2-.3) node {\color{blue}\small $1+i (\pi-\epsilon)$};		
		\draw[color=blue,very thick,dashed] (4,2-.15)--(2.4,2-.15)--(2.15,2.15)--(0,2.15);		
		\draw[blue,fill=blue] (0,2+.15) circle (.3ex); 					
		\draw (-.95,2+.3) node {\color{blue}\small $i (\pi+\epsilon)$};		
		
	\end{tikzpicture}
	\caption{The deformed contour $C'$}\label{fig:contourC':b:+}	
\end{figure}
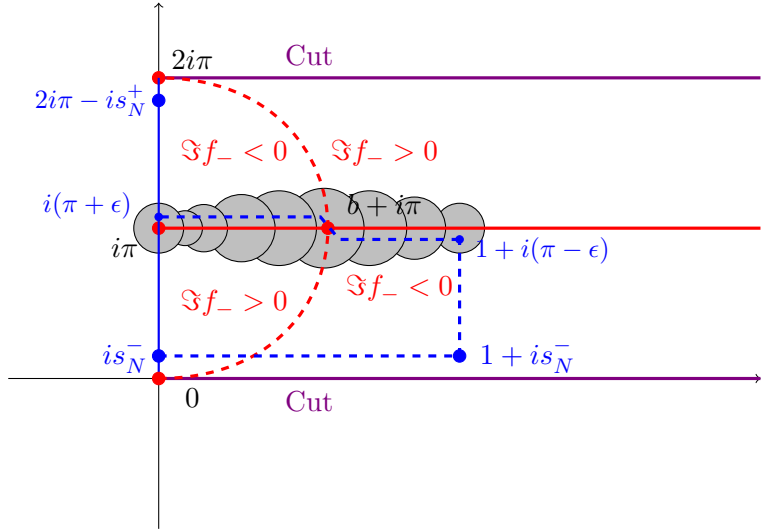
\smallskip

\underline{Asymptotics on $C'_1$:} On the horizontal line of altitude $s_N^-$, for all $x\geqslant 0$, we have
$$\dfrac{\partial}{\partial x} \Im (f_-(x+is_N^-)) = 2s_N^--2\pi - \arg\left (1-e^{x+is_N^-}\right )<-\frac{\pi}{2},$$
where the last inequality comes from $|\arg|\leqslant \pi$ and the fact that $s_N^-$ is small. Thus for all  $x> 0$ we have
$$\Im (f_-(x+is_N^-)) < \Im (f_-(0+is_N^-))- \frac{\pi}{2}x = {\rm Cl}_2(s_N^-)-\frac{\pi}{2}x.$$
Since ${\rm Cl}_2(s_N^-)< {\rm Cl}_2(\pi/3) \approx 1.10149$, it follows that for all $x \geqslant 1$ we have $\Im (f_-(x+is_N^-))\leqslant 
\Im (f_-(1+is_N^-))< 0$. Moreover, on $C'_1=[is_N^-,1+is_N^-]$ we have:
\begin{align*}
 \left |\int_{C'_1} e^{\frac{N}{2 i \pi}f_-(z)}dz\right | 
 &\leqslant
\int_0^{1}e^{\frac{N}{2  \pi}\Im (f_-(x+is_N^-))}dx\\
&\leqslant  \int_0^{1}e^{\frac{N}{2  \pi}\left ({\rm Cl}_2(s_N^-)-\frac{\pi}{2}x\right )}dx\\
&= e^{\frac{N}{2  \pi}{\rm Cl}_2(s_N^-)} \cdot  \int_0^{1}e^{\frac{N}{2 
		 \pi}\left (-\frac{\pi}{2}x\right )}dx \\
	 &= e^{\frac{N}{2  \pi}{\rm Cl}_2(s_N^-)} \cdot 
	 \dfrac{4}{N}\left (1-e^{-\frac{N}{4}}\right )\\
&=o_{N\to \infty}\left (\frac{1}{\sqrt{N}}\right ),\end{align*}
where the last asymptotics comes from $s_N^- e^{\frac{N}{2\pi}{\rm Cl}_2(s_N^-)}=
o_{N\to \infty}(1/\sqrt{N})$ established in Section \ref{smallIsN}.
\smallskip

\underline{Asymptotics on $C'_3 \cup C'_4$:} We proceed similarly as we did in the previous section for $C'_3$, constructing this time two almost horizontal curves $C'_3, C'_4$, going respectively from $1+i\pi-i\epsilon$ to $b+i\pi$ and from $b+i\pi$ to $i\pi+i\epsilon$. For this we construct a small disk $D_x$ around each point $x+i\pi$ of $[i\pi,1+i\pi]$, where $\Im(f_-)$ is negative on the bottom half-disk and
positive on the upper half-disk if $x>b$, and conversely if $x<b$. This is possible, because for any $x\geqslant0$ the vertical derivative at $x+i\pi$ is
$$\dfrac{\partial}{\partial t} \Im (f_-(x+it)) \vert_{t=\pi}
= 2x - \Log \vert 1-e^{x+i\pi} \vert=2x-\Log(1+e^x).$$
The reader will recognize the golden ratio, and agree that $2x-\Log(1+e^x)<0$ for $x \in [0,b)$ and $2x-\Log(1+e^x)>0$ for $x>b$. This justifies the location of the zones of signs of $\Im(f_-)$ in Figure \ref{fig:contourC':b:+} (the {\it valleys} and the {\it hills} of $f_-$ at $b$), whence the construction of the disks $D_x$. By compacity we can reduce to a finite family of such disks which cover $[i\pi,1+i\pi]$; they are drawn in grey in Figure \ref{fig:contourC':b:+}. Then we define $\epsilon>0$ small enough to be smaller than the radii of all grey disks.

Finally, in order to apply the saddle point method of Proposition \ref{prop:SPM:PV} for the contour $\lambda=C'_3\cup C'_4$, with $g=1$, $f=f_-$, and $z_0=b+i\pi$, we only need to check that $f''(z_0) \neq 0$. We have 
$$(f_-)''(b+i\pi)=1+\frac{1}{1-e^{b+i\pi}}= 1+ \frac{1}{1+\frac{1+\sqrt{5}}{2}}=\frac{5-\sqrt{5}}{2} =1.38... \neq 0$$ so we are good to go. For the sake of completeness, we can compute $f_-(b+i\pi) = 9\pi^2/10$, once again from the special values of the dilogarithm. Hence we get
$$ \int_{C'_3\cup C'_4} e^{\frac{N}{2 i \pi}f_-(z)}dz = \dfrac{\text{Constant}_-}{\sqrt{N}}(e^{i r_- N}+o_{N\to\infty}(1)),$$
where $\text{Constant}_- \in \C^*$ has modulus $2\pi \left (\frac{5-\sqrt{5}}{2}\right )^{-\frac{1}{2}}=5.34...$ and $r_-=-\frac{9 \pi}{20} \in \R$.
\smallskip

\underline{Asymptotics on $C'_2$:} Now that $\epsilon>0$ has been fixed, we define $C'_2$ as the vertical contour from $1+is_N^-$ to $1+i\pi-i\epsilon$.
To see that $\Im(f_-)$ is strictly increasing on $C'_2$ going upwards, one can for example prove that the vertical derivative is positive:
$$\dfrac{\partial}{\partial t} \Im (f_-(x+it)) \vert_{x=1}
= 2 - \Log \vert 1-e^{1+it} \vert \geqslant 2-\Log(1+e)>0,$$
where the last inequality stems from the fact that $1>b$. Hence the negative number $\Im \left (f_-\left (1+i\pi-i\epsilon\right )\right )$ is an upper bound for $\Im(f_-)$ on $C'_2$, and
$$ \left |\int_{C'_2} e^{\frac{N}{2 i \pi}f_-(z)}dz\right | \leqslant
\pi \cdot e^{\frac{N}{2 \pi}\Im \left (f_-\left (c+i\pi-i\epsilon\right )\right )} = o_{N\to \infty}\left (\frac{1}{\sqrt{N}}\right ).$$

\underline{Asymptotics on $C'_5$:} We define $C'_5:=i[\pi+\epsilon,2\pi-s_N^+]$, oriented upwards. From Section \ref{sub:small:vert:int}, since $i[\pi+\epsilon,2\pi-s_N^+] \subset [\pi,2\pi]$, we have:
$$ \left |\int_{C'_5} e^{\frac{N}{2 i \pi}f_-(z)}dz\right | \leqslant
\int_{\pi+\epsilon}^{2\pi-s_N^+} e^{\frac{N}{2 \pi}\Im(f_-(it))}dt=
\int_{\pi+\epsilon}^{2\pi-s_N^+} e^{\frac{N}{2 \pi}{\rm Cl}_2(t)}dt = o_{N\to \infty}\left (\frac{1}{\sqrt{N}}\right ).$$

This concludes the construction of the new contour $C'$, and by summing the integrals on each part, we obtain
$$ I_{n,-}=\int_{C'} e^{\frac{N}{2 i \pi}f_-(z)}dz = \dfrac{\text{Constant}_-}{\sqrt{N}}(e^{i r_- N}+o_{N\to\infty}(1)),$$
where
\begin{equation}\label{Cst-}
\vert \mathrm{Constant}_- \vert = 2\pi \left (\frac{5-\sqrt{5}}{2}\right )^{-\frac{1}{2}}=5.34..., \ r_-=-\frac{9 \pi}{20} \in \R.
\end{equation}

\section{Asymptotics of the quantum integral}\label{sec:rect:contour}

Recall that in Section \ref{sec:SPM} we defined the functions
\begin{itemize}
\item (case (a), function $f_-$) $\mathcal{L}_-(z;-2i\pi,i\pi) = {\rm Li}_2(e^{z})+z^2 - i\pi z$,
\item (case (a), function $f_+$)  $\mathcal{L}_+(z;-2i\pi,i\pi) = {\rm Li}_2(e^{z})+z^2 +i\pi  z$,
\item (case (b), function $f_-$) $\mathcal{L}_-(z;-2i\pi,2i\pi) = {\rm Li}_2(e^{z})+z^2 - 2i\pi z$,
\item (case (b), function $f_+$) $ \mathcal{L}_+(z;-2i\pi,2i\pi) = {\rm Li}_2(e^{z})+z^2$,
\end{itemize}
and we studied the asymptotics of the integrals $\textstyle \int_{[is_N^-,2i\pi-is_N^+]} e^{\frac{N}{2i\pi} \mathcal{L}_{\pm}(z,-2i\pi,\ell)}dz$ for $\ell=i\pi$ (case (a)) or $\ell=2i\pi$ (case (b)).

In this section we extend this study to obtain asymptotics of the integral on the right side of \eqref{int1}, $\textstyle \int_{C_N} e^{\frac{N}{2i\pi}(z^2 - \bl_{1,N}z)}\hat S_N(\bl_{0,N} + z)\textstyle \coth(\frac{Nz}{2}) dz$, thus finishing the proof of Theorem \ref{asyinvteo} (see the steps IX and X in Section \ref{sec:sketchproof}).

Because the integrand depends on $N$ in a complicated way, we need a preliminary result, Lemma \ref{qcintjui25} below, to approximate the integral. 

To state this lemma, denote by $C^+_N$ and $C^-_N$ the intersections of the contour $C_N$ (defined in \eqref{choiceCN}) with the right and left half-spaces; see Figure \ref{fig:contours:CN+-}. Recall that we assume $\textstyle s_N^- := \alpha \frac{\pi}{N}$ with $\alpha\in ]0,1[$ (see \eqref{sNvalue}), and $s_N^+\in ]0,\textstyle \frac{2\pi}{N}[$. Consider the following shift of the function $\Psi_N$ that we studied in the lemmas \ref{estimatehatS}-\ref{extboundPsiNjanv26}: \begin{equation}\label{defPsiu0N}
{\Psi_{u_0,N}(z)} := \Psi_N\left(\frac{1}{N}\Log(u_{0}) +z\right).
\end{equation}
It is holomorphic on the domain \eqref{holdomPsiN} shifted by $\textstyle -\frac{1}{N}\Log(u_{0})$. Finally, recall the log-parameters ${\bf l}_k' := \Log(u_{k}) +i\pi a_k$ from \eqref{tetcol}, and that the edge colors $a_0,a_1\in \mz$ were specified in Sections \ref{4_1} and \ref{loglimmars25}. Define
\begin{equation}\label{defrho0}
\rho(z) := e^{-\frac{{\bf l}_1'}{2i\pi}z}(1-e^{z})^{-\frac{{\bf l}_0'}{2i\pi}-\frac{1}{2}}.
\end{equation}
Note that our choice of $a_0\geq 2$ makes $\textstyle \mathfrak{R}(\frac{{\bf l}_0'}{2i\pi}+\frac{1}{2}) >0$, so $\rho(z)$ has poles at the points $z\in (2i\pi)\mz$. This will have important consequences in the analysis of the integrals in Lemma \ref{qcintjui25}, which will be done in Lemma \ref{lem:reformintqjui25}.

\begin{figure}[!h]
\scalebox{.8}{
	\begin{tikzpicture}
		\draw[color=lightgray,->] (-5,0)--(5,0);
		\draw[color=lightgray,->] (0,-1)--(0,6.5);
		\draw (4,-.05) node {\Huge $\cdot$};
		\draw (4,-.25) node {$\varepsilon$};
		\draw (+.025,-.05) node {\Huge $\cdot$};
		\draw (+.25,-.25) node {$0$};
		\draw (+.025,6-.05) node {\Huge $\cdot$};
		\draw (+.25+.2,6) node {$2 \pi i$};
		\draw (+.025,4.5-.05) node {\Huge $\cdot$};
		\draw (.55+.25+.2,4.5) node {$\textstyle 2\pi i - \frac{2\pi}{N}i$};		
		\draw (+.025,1.5-.05) node {\Huge $\cdot$};
		\draw (.3+.25+.2,1.5) node {$\textstyle \frac{2\pi}{N}i$};	
		
		\draw[color=red,->,thick] (0,1)--(3,1)--(3,5)--(0,5)	;
		\draw[color=blue] (+.025,1-.05) node {\Huge $\cdot$};		
		\draw[color=red] (3.5,3-.3) node {$C^+_N$};		
		\draw[color=red,<-,thick] (0,1)--(-3,1)--(-3,5)--(0,5)	;		
		\draw[color=blue] (+.025,5-.05) node {\Huge $\cdot$};		
		\draw[color=red] (-3.5,3-.3) node {$C^-_N$};
		
%
		
		\draw[color=blue] (-.3,1-.4) node{$i s_N^-$};	

		\draw[color=blue] (-.8,5+.3) node{$2\pi i - i s_N^+$};
		
		\draw [color=blue,very thick] (0,1)--(0,5);		
		
	\end{tikzpicture}
	}
	\caption{The contour $C_N$ and its subcontours  $C^+_N$ and $C^-_N$}
	\label{fig:contours:CN+-}
\end{figure}
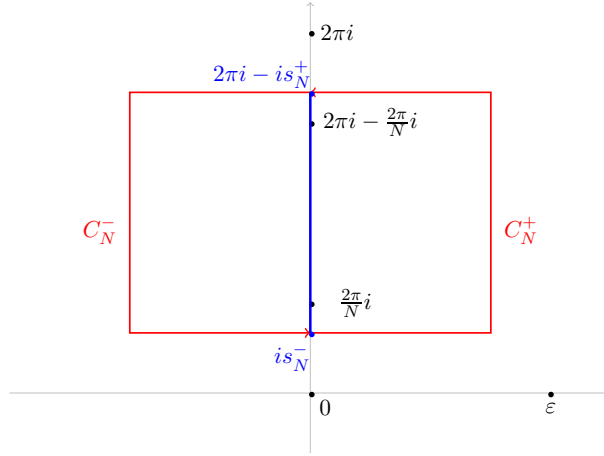

\begin{lem}\label{qcintjui25} If $\textstyle \bl_{0,\infty} = -2i\pi$ and $\textstyle \bl_{1,\infty} =: \ell\in \{\pi i, 2\pi i\}$ (as in Lemma \ref{rem:reduction:sigma:integral}) we have
	\begin{multline*} \int_{C_N} e^{\frac{N}{2i\pi}(z^2 - \bl_{1,N}z)}\hat S_N(\bl_{0,N} + z)\textstyle \coth(\frac{Nz}{2}) dz   \\
= \int_{C^+_N} \left( e^{\frac{N}{2i\pi} \mathcal{L}_{-}(z;-2i\pi,\ell)} - u_0 e^{\frac{N}{2i\pi} \mathcal{L}_{+}(z;-2i\pi,\ell)}\right) \rho(z) \exp\left(\frac{1}{N}{\Psi_{u_0,N}(z)}\right) \left(1 + \frac{1}{N}R_N(z)\right) dz \\ -\int_{C^-_N} \left( e^{\frac{N}{2i\pi} \mathcal{L}_{-}(z;-2i\pi,\ell)}  - u_0 e^{\frac{N}{2i\pi} \mathcal{L}_{+}(z;-2i\pi,\ell)}\right)\rho(z)  \exp\left(\frac{1}{N}{\Psi_{u_0,N}(z)}\right) \left(1 + \frac{1}{N}R_N(z)\right) dz\\
= 2\int_{C^+_N} \left( e^{\frac{N}{2i\pi} \mathcal{L}_{-}(z;-2i\pi,\ell)} - u_0 e^{\frac{N}{2i\pi} \mathcal{L}_{+}(z;-2i\pi,\ell)}\right) \rho(z) \exp\left(\frac{1}{N}{\Psi_{u_0,N}(z)}\right) \left(1 + \frac{1}{N}R_N(z)\right) dz,
	\end{multline*}
where the function $R_N$, with formula \eqref{defRNjanv26}, is holomorphic on the domain
\begin{equation}\label{domRN}
\mc \setminus \left(\bigcup_{k\in \mz} \left(2i\pi k  -\frac{1}{N}\Log(u_{0})+[0,+\infty[\right)\bigcup \left(\bigcup_{k\in \mz} \left(2i\pi k +[0,+\infty[\right)\right)\right),
\end{equation}
and $\vert R_N\vert$ is bounded from above by a constant independent of $N$ on any compact subset of a strip $\textstyle \mr + i(2\pi k + [\delta + \frac{\pi}{N}, 2\pi-\delta - \frac{\pi}{N}])$, $\textstyle \delta\in ]0,\pi[$, $k\in \mz$, contained in this domain. 
\end{lem}

\proof \underline{Step 1: Rephrasing $\textstyle  \hat S_N(\bl_{0,N} + z)$.}

Recall from \eqref{recapuvarmars25}-\eqref{situation2}-\eqref{l0inftymars25} that $\textstyle \bl_{0,N}= -2i\pi + \frac{1}{N}(\Log(u_{0}) +\pi ia_{0})$ where $a_0$ is even and $\mathfrak{I}(\Log(u_{0}) +\pi ia_{0})>0$. We can use the functional equation $\textstyle\hat S_{N}(z) = \hat S_{N}(z+\frac{2i\pi}{N})(1-e^{z+\frac{2i\pi}{N}})$ (see \eqref{eqfunchatS}) in order to shift $N-\textstyle \frac{a_0}{2}$ times the argument $\bl_{0,N} + z$: since $\textstyle \bl_{0,N} + z+\frac{2i\pi}{N}(N-\frac{a_0}{2}) = \frac{1}{N}\Log(u_{0}) +z$ we have
\begin{equation}\label{prodShatdec25prooflem51}
\hat S_N(\bl_{0,N} + z) = \hat S_N\left(\frac{1}{N}\Log(u_{0}) +z\right) \prod_{k=1}^{N-a_0/2}\left(1-e^{\bl_{0,N} + z+\frac{2i\pi}{N}k}\right).
\end{equation}

\

\underline{Step 2: Approximating $\textstyle \hat S_N\left(\frac{1}{N}\Log(u_{0}) +z\right)$.}

Let us approximate $\textstyle \hat S_N\left(\frac{1}{N}\Log(u_{0}) +z\right)$. Since  $-\pi < \mathfrak{I}\left(\Log(u_{0})\right)<  \pi$, for every $z\in C_N$ we have
\begin{equation}\label{stripjuil25}
-\frac{\pi}{N}+s_N^-\leq \mathfrak{I}\left(\frac{1}{N}\Log(u_{0}) +z\right) \leq 2\pi -s_N^+ + \frac{\pi}{N}.
\end{equation}
We can get finer bounds, depending on $u_0$. Indeed, assume that $\Im(\Log(u_0))$ satisfies 
\begin{equation}\label{uniformu0janv26}
-\pi  \upsilon < \Im(\Log(u_0)) < \pi \upsilon',
\end{equation} where $0<\upsilon,\upsilon'<1$. Let $\upsilon''>0$ be such that $\upsilon'<\upsilon''<1$. Take
\begin{equation}\label{fixsN+-}
\left\lbrace \begin{array}{l} s_N^- := \alpha^- \frac{\pi}{N}, \mathrm{with\ } \alpha^-:= \frac{1+\upsilon}{2},\\ \\ s_N^+ := \alpha^+ \frac{\pi}{N}, \mathrm{with\ } \alpha^+:= 1+\upsilon''.
\end{array}\right.
\end{equation}
Then the inequalities \eqref{stripjuil25} become:
\begin{equation}\label{eq:ineq:u0:depends:on:N}
 0<\frac{(1-\upsilon)\pi}{2N}= \frac{-\pi\upsilon}{N}+s_N^-\leq \mathfrak{I}\left(\frac{1}{N}\Log(u_{0}) +z\right) \leq 2\pi -s_N^+ + \frac{\upsilon'\pi}{N} < 2\pi -\frac{(\upsilon''-\upsilon')\pi}{N}-\frac{\pi}{N}.
\end{equation}
Note that the choices of $\textstyle s_N^-= \alpha^- \frac{\pi}{N}$, with $\alpha^-\in ]0,1[$, and of $\textstyle s_N^+= \alpha^+ \frac{\pi}{N}$, with $\alpha^+\in ]1,2[$, satisfy the assumptions of Section \ref{sec:SPM} (see \eqref{sNvalue}); moreover $\textstyle \alpha^-= \frac{1+\upsilon}{2}$ ensures that strips like \eqref{eq:ineq:u0:depends:on:N} avoid cuts. The choice of $\alpha^+$, which must lie within $]0,2[$ by the definition of $s_N^+$, in the subinterval $]1,2[$ is relevant only for Remark \ref{sharpdeltaN+-janv26} below.

By Remark \ref{rem:largeN} (ii) we can use  the equality \eqref{equalitylemmaasyoct25}. Taking the exponential we get 
$$\hat S_N\left(\frac{1}{N}\Log(u_{0}) +z\right)  = \exp\left(\frac{N}{2i\pi}{\rm Li}_2\left(e^{\frac{1}{N}\Log(u_{0}) + z}\right)-\frac{1}{2}\Log(1-e^{\frac{1}{N}\Log(u_{0}) + z})+\frac{1}{N}{\Psi_{u_0,N}(z)}\right)$$
where $\textstyle {\Psi_{u_0,N}(z)} = \Psi_N\left(\frac{1}{N}\Log(u_{0}) +z\right)$ (see \eqref{defPsiu0N}). Using that
\begin{align}
{\rm Li}_2\left(e^{\frac{1}{N}\Log(u_{0}) + z}\right)  & = {\rm Li}_2\left(e^z\right) -\frac{\Log(u_{0})}{N}\int_0^1 \Log\left(1-e^{\frac{\Log(u_{0})}{N}t + z}\right)dt\label{approx1dec25}\\
\Log\left(1-e^{\frac{\Log(u_{0})}{N}t + z}\right) & = \Log(1-e^z)- \frac{\Log(u_{0})}{N}e^{z}\int_0^t \frac{e^{\frac{\Log(u_{0})}{N}s}}{1-e^{\frac{\Log(u_{0})}{N}s + z}}ds \notag
\end{align}
we can write
\begin{align*} \hat S_N\left( \frac{1}{N}\Log(u_{0}) +z\right)   & =  \exp\left(\frac{N}{2i\pi}{\rm Li}_2\left(e^{z}\right)\right) \exp\left({\mathcal{T}_{u_0,N}(z)}+\frac{1}{N}{\Psi_{u_0,N}(z)}\right)\\ & = \exp\left(\frac{N}{2i\pi}{\rm Li}_2\left(e^{z}\right)\right) \exp\left(-\frac{1}{2}\left(1+\frac{\Log(u_0)}{i\pi}\right)\Log(1-e^{z})\right)\\ &  \hspace*{5cm} \times \exp\left(\frac{1}{N}\left({\Psi_{u_0,N}(z)}+
	{\mathcal{I}_{u_0,N}(z)}\right)\right),
\end{align*}
where 
\begin{align*} 
{\mathcal{T}_{u_0,N}(z)}
&:=  - \frac{N}{2i\pi}{\rm Li}_2\left(e^{z}\right)
+\frac{N}{2i\pi}{\rm Li}_2\left(e^{\frac{1}{N}\Log(u_{0}) + z}\right)-\frac{1}{2}\Log(1-e^{\frac{1}{N}\Log(u_{0}) + z})
\\
	&= -\frac{1}{2}\Log\left(1-e^{\frac{1}{N}\Log(u_{0}) + z}\right) -\frac{\Log(u_{0})}{2i\pi}\int_0^1 \Log\left(1-e^{\frac{\Log(u_{0})}{N}t + z}\right)dt \\  
	&  =
-\frac{1}{2}\left(1+\frac{\Log(u_0)}{i\pi}\right)\Log(1-e^{z})\\ & \hspace*{0.5cm}+\frac{\Log(u_0)}{N}e^z\left(\frac{\Log(u_0)}{2i\pi}\int_0^1\left(\int_0^t \frac{e^{\frac{\Log(u_{0})}{N}s}}{1-e^{\frac{\Log(u_{0})}{N}s + z}}ds\right)dt + \frac{1}{2}  \int_0^1 \frac{e^{\frac{\Log(u_{0})}{N}s}}{1-e^{\frac{\Log(u_{0})}{N}s + z}}ds\right)\\  
&= -\frac{1}{2}\left(1+\frac{\Log(u_0)}{i\pi}\right)\Log(1-e^{z}) + \mathcal{I}_{u_0,N}(z)
\end{align*}
and
$$\mathcal{I}_{u_0,N}(z):=\frac{\Log(u_0)}{N}e^z\left(\frac{\Log(u_0)}{2i\pi}\int_0^1\left(\int_0^t \frac{e^{\frac{\Log(u_{0})}{N}s}}{1-e^{\frac{\Log(u_{0})}{N}s + z}}ds\right)dt + \frac{1}{2}  \int_0^1 \frac{e^{\frac{\Log(u_{0})}{N}s}}{1-e^{\frac{\Log(u_{0})}{N}s + z}}ds\right)$$
is the last summand of $\mathcal{T}_{u_0,N}(z)$ made of the integrals.

By its definition the function ${\mathcal{T}_{u_0,N}(z)}$ is holomorphic on the domain \eqref{domRN}. Then so is ${\mathcal{I}_{u_0,N}(z)}$. Applying several times the triangle inequality on  $\textstyle \vert {\mathcal{I}_{u_0,N}(z)} \vert$, and using Lemma \ref{lem:bound:exp:1-exp} as at the end of Lemma \ref{estimatehatS}, we find
\begin{equation}\label{majorjanv26_2} \vert {\mathcal{I}_{u_0,N}(z)}\vert  \leq  \dfrac{B_M}{\delta\sqrt{1- \frac{\pi^2}{24}}}
\end{equation}
for every $\textstyle z\in ]-\infty,M] + i(2\pi k + [\delta+\frac{\pi}{N}, 2\pi-\delta - \frac{\pi}{N}])$, $\textstyle \delta\in ]0,\pi[$, $k\in \mz$ (so that $\textstyle \delta \leq \mathfrak{I}\left(\frac{1}{N}\Log(u_{0})s +z\right) \leq 2\pi -\delta$ for every $s\in [0,1]$), where $B_M>0$ is a constant independent of $N$, $\delta$ and $k$.

\

\underline{Step 3: Approximating $\textstyle \prod_{k=1}^{N-a_0/2}\left(1-e^{\bl_{0,N} + z+\frac{2i\pi}{N}k}\right)$.}

Next consider the product in \eqref{prodShatdec25prooflem51}. By factoring out 
$a_0/2$
 terms, and using that $\textstyle\bl_{0,N} = -2i\pi + O_{N\ra \infty}(\frac{1}{N})$, for every $z\in C_N$ we get
\begin{align}
\prod_{k=1}^{N-a_0/2}(1-e^{\bl_{0,N} + z+\frac{2i\pi}{N}k}) & =  \prod_{k=1}^{N}(1-e^{\bl_{0,N} + z+\frac{2i\pi}{N}k}) \prod_{k=N+1-a_0/2}^{N}(1-e^{\bl_{0,N} + z+\frac{2i\pi}{N}k})^{-1}\notag\\
& =  (1-e^{N(z+\bl_{0,N})}) \prod_{k=N+1-a_0/2}^{N}\left(1-e^{z+ \frac{1}{N}O_{N\ra \infty}(1)}\right)^{-1} \notag \\
& = (1-u_0 e^{Nz}) \left(1-e^{z+ \frac{1}{N} O_{N\ra \infty}(1)}\right)^{-a_0/2} \notag \\
& = \frac{1-u_0 e^{Nz}}{(1-e^{z})^{a_0/2}} \ .\ \frac{1}{(1-\frac{e^{z}}{e^z-1}\ . \frac{1}{N}O_{N\ra \infty}(1))^{a_0/2}}\notag \\ & = \frac{1-u_0 e^{Nz}}{(1-e^{z})^{a_0/2}} \left(1+ \frac{1}{N}O_{N\ra \infty}\left(1\right)\right), \label{facterrmars25}
\end{align}
where the last $\textstyle O_{N\ra \infty}(1)$ term is a sequence of holomorphic functions on $\mc\setminus \{2i\pi k, k\in \mz\}$, where each singularity $2i\pi k$ is removable, and with moduli bounded from above by some constant independent of $N$ on any fixed compact subset of $\mc$.
\smallskip

\underline{Step 4: Approximating $\textstyle \coth\left(\frac{Nz}{2}\right)$ and $e^{\frac{N}{2i\pi} (z^2 - \bl_{1,N}z)}$.}

Finally, for every $z$ in $\textstyle C_N \setminus \{is_N,2i\pi-is_N - \frac{2i\pi}{N}\}$ we have
$$\coth\left(\frac{Nz}{2}\right) = \mathrm{sign}\left(\mathfrak{R}(z)\right) \left(1+ o_{N\ra \infty}\left(\frac{1}{N}\right)\right),$$ where the $\textstyle o_{N\ra \infty}\left(\frac{1}{N}\right)$ term is a sequence of holomorphic functions on $\C \setminus \{2i\pi k, k \in \Z\}$, where each singularity $2i\pi k$ is a simple pole, and with moduli bounded from above by some constant independent of $N$ on any compact subset of this domain. 

Moreover, since $\textstyle \bl_{1,N}= \frac{1}{N}{\bf l}'_1+\pi i a_1$, with ${\bf l}'_1 =\Log(u_{1}) +\pi ia_{1}$, and $\pi i a_1 = \ell$ mod$(2i\pi)$, with $\ell\in \{\pi i,2\pi i\}$, by the comment after \eqref{int01} for every $z\in C_N$ we can replace $e^{-\frac{N}{2i\pi}\bl_{1,N}z}$ with
$$e^{-\frac{N}{2i\pi}(\frac{1}{N}(\Log(u_{1}) +\pi ia_{1})+\ell) z} = e^{-\frac{1}{2i\pi}(\Log(u_{1}) +\pi ia_{1})z} e^{-\frac{N}{2i\pi}\ell z} = e^{-\frac{{\bf l}'_1}{2i\pi}z} e^{-\frac{N}{2i\pi}\ell z}.$$

\underline{Step 5: Putting all together towards the RHS.}

Recalling the formula \eqref{defrho0}, the approximations we have obtained 
in Steps 2 to 5
imply that for all $z\in C_N$ one has:
\begin{align*}
e^{\frac{N}{2i\pi} (z^2 - \bl_{1,N}z)} & \hat S_N(\bl_{0,N} + z)\textstyle \coth(\frac{Nz}{2}) & \\
& = 
	\mathrm{sign}\left(\mathfrak{R}(z)\right)
	\exp\left(\frac{N}{2i\pi}\left({\rm Li}_2\left(e^{z}\right)+ z^2 - \ell z\right)\right){(1-u_0 e^{Nz})}\\ &\hspace*{4cm} \times \rho(z)    \exp\left(\frac{1}{N}{\Psi_{u_0,N}(z)}\right) \left(1 + \frac{1}{N}R_N(z)\right) 
\\ 
& = \mathrm{sign}\left(\mathfrak{R}(z)\right) 
\left( e^{\frac{N}{2i\pi} \mathcal{L}_{-}(z;-2i\pi,\ell)} - u_0 e^{\frac{N}{2i\pi} \mathcal{L}_{+}(z;-2i\pi,\ell)}\right) \\ & \hspace*{4cm} \times \rho(z)   \exp\left(\frac{1}{N}{\Psi_{u_0,N}(z)}\right) \left(1 + \frac{1}{N}R_N(z)\right),
\end{align*}
where
\begin{equation}\label{defRNjanv26}
1+\frac{1}{N}R_N(z) := \exp\left(\frac{1}{N}{\mathcal{I}_{u_0,N}(z)}\right)\left(1+\frac{1}{N}O_{N\ra \infty}\left(1\right)\right)
{
\dfrac{ \coth\left(\frac{Nz}{2}\right) }{ \mathrm{sign}\left(\mathfrak{R}(z)\right)}
}
\end{equation} satisfies the expected properties (the term $\textstyle \frac{1}{N}O_{N\ra \infty}(1)$ comes from \eqref{facterrmars25}, and recall that $\textstyle \frac{ \coth\left(\frac{Nz}{2}\right) }{ \mathrm{sign}\left(\mathfrak{R}(z)\right)}= \left(1+ o_{N\ra \infty}\left(\frac{1}{N}\right)\right)$
 from Step 5). More precisely, $R_N(z)$ is of the form
$$R_N(z)= \mathcal{I}_{u_0,N}(z)+O_{N\ra \infty}\left(1\right)+o_{N\ra \infty}\left(1\right) +O_{N\ra \infty}\left(\frac{1}{N}\right),$$
where the second, third and fourth terms in the right hand side are of course functions of $z$ as before. Denote by $C^+_N$ and $C^-_N$ the intersections of $C_N$ with the right and left half-spaces. Because of the term $\textstyle \mathrm{sign}\left(\mathfrak{R}(z)\right)$ in the above expressions, we can write
\begin{multline*} \int_{C_N} e^{\frac{N}{2i\pi}(z^2 - \bl_{1,N}z)}\hat S_N(\bl_{0,N} + z)\textstyle \coth(\frac{Nz}{2}) dz = 
			  \\
\int_{C^+_N} \left( e^{\frac{N}{2i\pi} \mathcal{L}_{-}(z;-2i\pi,\ell)} - u_0 e^{\frac{N}{2i\pi} \mathcal{L}_{+}(z;-2i\pi,\ell)}\right) \rho(z) \exp\left(\frac{1}{N}{\Psi_{u_0,N}(z)}\right) \left(1 + \frac{1}{N}R_N(z)\right) dz\\
		-  \int_{C^-_N} \left( e^{\frac{N}{2i\pi} \mathcal{L}_{-}(z;-2i\pi,\ell)}  - u_0 e^{\frac{N}{2i\pi} \mathcal{L}_{+}(z;-2i\pi,\ell)}\right)\rho(z)  \exp\left(\frac{1}{N}{\Psi_{u_0,N}(z)}\right) \left(1 + \frac{1}{N}R_N(z)\right) dz.
	\end{multline*}
The integrand is holomorphic on the intersection $U$ of the domains \eqref{holdomPsiN} shifted by $\textstyle -\frac{1}{N}\Log(u_{0})$ and the domain \eqref{domRN}. Since $-\pi <\mathfrak{I}(\Log(u_0))< \pi$ one can choose $s_N^-$, $s_N^+$ in \eqref{fixsN+-} so that $s_N^\pm > \textstyle \frac{1}{N} \mathfrak{I}(\Log(u_0))$, and therefore the contours $C^+_N$ and $C^-_N$ are homotopic within $U$ to $[is_N^-,2i\pi-is_N^+]$ relatively to their endpoints, with upwards and downwards orientations respectively. So the opposite of the integral over $C_N^-$ coincides with the integral over $C_N^+$, and the last expression reduces to $2$ times the integral over $C_N^+$ on the RHS. This concludes the proof.\cvd

\begin{remark}\label{sharpdeltaN+-janv26}{\rm We can give sharp uniform bounds on $\vert {\Psi_{u_0,N}(z)}\vert$ for points $z\in C_N$, as follows. Define (with $\upsilon,\upsilon',\upsilon''$ as in \eqref{uniformu0janv26}--\eqref{fixsN+-}):
\begin{align*}
\delta^-_N & := \frac{(1-\upsilon)\pi}{2N} + \frac{\pi}{N},\\
\delta^+_N & := \frac{(\upsilon''-\upsilon')\pi}{N}.
\end{align*}
Then the inequalities \eqref{eq:ineq:u0:depends:on:N} become
\begin{equation}\label{ineqdelta+-janv26}
\delta^-_N - \frac{\pi}{N}\leq \mathfrak{I}\left(\frac{1}{N}\Log(u_{0}) +z\right) \leq 2\pi -\delta^+_N - \frac{\pi}{N},
\end{equation} and thus define strips as in Lemma \ref{estimatehatS} and Remark \ref{rem:largeN} (iii) (with the variable $\textstyle \frac{1}{N}\Log(u_{0}) +z$ in place of $z$). In that lemma take now $M$ equal to the constant $\varepsilon$ used in the definition of the contours $C_N$. Then, by \eqref{UboundPsi}, for every $z$ in the strip \eqref{ineqdelta+-janv26} satisfying  $\textstyle \mathfrak{R}(z)\leq \varepsilon'$ for some $\varepsilon'>\varepsilon$ (so that $\textstyle \mathfrak{R}(\frac{1}{N}\Log(u_{0}) +z)\leq \varepsilon$ for $N$ large enough), we have 
\begin{equation}\label{majorjanv26_1}
\vert {\Psi_{u_0,N}(z)}\vert \leq \frac{B'}{\delta_N^+}+\dfrac{\pi e^{\varepsilon/2}}{2\delta_N^+}\sqrt{1- \frac{\pi^2}{24}}+B'',
\end{equation}
for constants $B',B''>0$ independent of $N$ and of $\delta_N^+$ (note that $\delta_N^+=\min(\delta_N^-,\delta_N^+)$). Since $N\delta_N^+ = \pi (\upsilon''-\upsilon')>0$, \eqref{majorjanv26_1} shows that $\textstyle \frac{1}{N}\vert {\Psi_{u_0,N}(z)}\vert$ is bounded from above by a constant independent of $N$, $\delta_N^-$, and $\delta_N^+$ on any half-strip given by \eqref{ineqdelta+-janv26} and $\textstyle \mathfrak{R}(z)\leq \varepsilon'$.

By Remark \ref{rem:largeN} (i), the bound \eqref{majorjanv26_1} holds true also when $\delta_N^-$ in \eqref{ineqdelta+-janv26} is replaced with any $\delta\in ]0,\pi[$, as long as $\textstyle \frac{1}{N}\Log(u_{0}) +z$ is not contained in a cut of the domain \eqref{holdomPsiN} shifted by $\textstyle -\frac{1}{N}\Log(u_{0})$. In particular we can take $\delta_N^-$ of the form $\delta_N^- = \alpha'\textstyle \frac{\pi}{N}$, with $\alpha'>0$ arbitrary. Since  $N\delta_N^- = \alpha'\pi >0$, as in the proof of Lemma \ref{extboundPsiNjanv26} (ii) we find that $\textstyle \frac{1}{N}\vert {\Psi_{u_0,N}(z)}\vert$ is bounded from above by a constant independent of $N$, $\delta_N^+$, and such a $\delta_N^-$ on any half-strip given by \eqref{ineqdelta+-janv26} and $\textstyle \mathfrak{R}(z)\leq \varepsilon'$.}
\end{remark}

By Lemma \ref{qcintjui25} we have 
\begin{equation}\label{finalestimatefev26}\int_{C_N} e^{\frac{N}{2i\pi}(z^2 - \bl_{1,N} z)}\hat S_N(\bl_{0,N} + z)\textstyle \coth(\frac{Nz}{2}) dz =
 2I_{N,-}(\ell) -2u_0  I_{N,+}(\ell),
 \end{equation}
 where
\begin{equation} I_{N,\pm}(\ell)  := \int_{C_N^+} g_N(z)e^{\frac{N}{2i\pi} \mathcal{L}_{\pm}(z;-2i\pi,\ell)} dz,\label{defINpmellfev26}
\end{equation}
and we put
\begin{align}
g_N(z)& := \rho(z) \exp\left(\frac{1}{N}\Psi_{u_0,N}(z)\right)\left(1 + \frac{1}{N}R_N(z)\right)\label{defgNjanv26}\\
& =\rho(z) \exp\left(\frac{1}{N}\Psi_{u_0,N}(z)\right)
\exp\left(\frac{1}{N}\mathcal{I}_{u_0,N}(z)\right)\left(1+\frac{1}{N}O_{N\ra \infty}\left(1\right)\right)
	\dfrac{ \coth\left(\frac{Nz}{2}\right) }{ \mathrm{sign}\left(\mathfrak{R}(z)\right)}.\notag
\end{align}
Lemma \ref{lem:reformintqjui25} below extends to the integrals $I_{N,\pm}(\ell)$ in \eqref{defINpmellfev26} the asymptotics obtained for the integrals in Section \ref{sec:SPM}. The proof requires the following generalizations of Proposition \ref{prop:SPM:PV}; why we need these generalizations is explained before Lemma \ref{lem:reformintqjui25}. First, we will need a special case of Perron's method for a contour starting at a maximum.
\begin{prop} \label{Perronfev26} (i) {\rm (\cite[Theorem 1.2, case of $\mu=1$ and $S=1$]{O'Sul})} Assume that $\lambda$ is a compact piecewise $\mathcal{C}^1$ contour from $z_0$ to $z_1$ in $\mc$, and $f$ and $g$ are holomorphic functions on a domain containing $\lambda$. Assume furthermore that $f'(z_0)\ne 0$, and $\mathfrak{I}(f(z)) < \mathfrak{I}(f(z_0))$ for all $z\in \lambda$, $z\ne z_0$. Then we have 
\begin{equation}\label{1Perronmars26} \int_{\lambda} g(z) e^{n (-i) f(z)} dz = 
e^{n (-i) f(z_0)} \left (\dfrac{g(z_0)}{i f'(z_0)}\dfrac{1}{n}+ O_{n \to \infty}\left(\frac{K_g}{n^2}\right) \right ),
\end{equation} 
where $\textstyle K_g:= \sup_{z\in \lambda \cup \mathcal{B}}\vert g(z)\vert$ and $\textstyle O_{n \to \infty}\left(\frac{K_g}{n^2}\right)\leq C\cdot \frac{K_g}{n^2}$ for a constant $C>0$ independent of $g$ and $n$.\\
(ii) Suppose that the assumptions in (i) on $\lambda$ and $f$ hold, and a sequence of holomorphic functions $(g_n)$ converges uniformly to a function $g$ on a domain containing $\lambda$. Then 
\begin{equation}\label{2Perronmars26}\int_{\lambda} g_n(z) e^{n (-i) f(z)} dz =  e^{n (-i) f(z_0)} \left (\dfrac{g(z_0)}{i f'(z_0)}\dfrac{1}{n}+ o_{n \to \infty}\left(\dfrac{1}{\vert f'(z_0)\vert}\cdot \frac{1}{n}\right) \right ),
\end{equation}
where $\textstyle o_{n \to \infty}\left(\frac{1}{\vert f'(z_0)\vert}\cdot \frac{1}{n}\right) \leq \varepsilon_n \frac{1}{\vert f'(z_0)\vert}\cdot \frac{1}{n}$ for a sequence of real positive numbers $(\varepsilon_n)$ converging to $0$ and independent of $f$ and $n$.\\
(iii) The $\textstyle O_{n \to \infty}$ and $\textstyle o_{n \to \infty}$ error terms in \eqref{1Perronmars26} and \eqref{2Perronmars26} do not depend on $z_0$ when $z_0$ varies in a fixed compact set which does not contain critical points of $f$.
\end{prop}
Note that (iii) says that the expansions \eqref{1Perronmars26} and \eqref{2Perronmars26} are {\it uniform} with respect to $z_0$ in the given range; see \cite[Chapter VII]{W} for the notion of uniform asymptotic expansions.
\medskip

\noindent {\it Proof of (ii) and (iii).} The statement (ii) is a direct consequence of the proof of (i) given in \cite{O'Sul}. Indeed write $g_n = g + (g_n-g)$, and apply (i) to the integrals with integrands $g$ and $g_n-g$. The uniform convergence of $(g_n-g)$ to the $0$ function implies an asymptotic expansion like \eqref{1Perronmars26} for the integral with integrand $g_n-g$, with coefficients $\textstyle \frac{(g_n-g)(z_0)}{i f'(z_0)} $ and $K_{g_n-g}$ arbitrarily close to $0$ as $n\to +\infty$. As there is no assumption on the convergence speed of the sequence $(g_n)$, the error term is only known to be an $\textstyle o_{n \to \infty}\left(\frac{1}{\vert f'(z_0)\vert}\cdot \frac{1}{n}\right)$.\\
For the statement (iii) one uses an estimate of the constant $C$ in (i). Namely, by the equations (3-10) to (3-15) of \cite{O'Sul} in the case $\mu=1$ and $S=1$, we have a lower bound for $C$ of the following form, described below:
\begin{equation}\label{impliedCfev26}
O_{n \to \infty}\left(\frac{K_g}{n^2}\right) \leq \frac{16K_\phi}{\vert f'(z_0)\vert^2}\cdot \frac{K_g}{n^2} +  O_{n \to \infty}\left(K_ge^{-\varepsilon n} \right).
\end{equation}
Here $O_{n \to \infty}\left(K_ge^{-\varepsilon n} \right)\leq C' K_ge^{-\varepsilon n}$, with $C'>0$ a constant bounded by the length of $\lambda$, and $\varepsilon>0$ depends only on the contour $\lambda$ and a constant $R_f(z_0)>0$; precisely, $\varepsilon$ is chosen so that $\mathfrak{I}(f(z)-f(z_0))\leq -\varepsilon$ for all $z$ satisfying $\vert z-z_0\vert \geq R_f(z_0)$ and belonging to a small deformation of $\lambda$, where the deformation occurs in the disk $D(z_0,R_f(z_0)]$ (see the proof of \cite[Proposition 1.3]{O'Sul}). The constant $R_f(z_0)$ is such that the following properties hold (\cite[Proposition 2.1]{O'Sul}): (*)  the solutions $(r,\theta)$ of the equation $\mathfrak{I}(f(z_0+re^{i\theta})-f(z_0))/r=0$ have the form $(r,\theta)=(r,h_k(r))$ when $r$ belongs to on an interval containing $[0, R_f(z_0)]$, where the functions $h_k(r)$, $k\in \{0,1\}$,  are all defined and differentiable, and $\textstyle h_k(0) = -\arg(if'(z_0))+\pi(k+\frac{1}{2})$.\\
Now assume that $z_0$ varies in a compact set $E$ which does not contain critical points of $f$. Then $R_f(z_0)$ as above is defined and can be chosen continuously as a function of $z_0\in E$. We have $\textstyle R_f:= \min_{z_0\in E} (R_f(z_0))>0$, and the above properties (*) hold true uniformly for all $z_0\in E$, replacing $R_f(z_0)$ with $R_f$. Using $R_f$ to define the $\varepsilon>0$ above,  it depends only on the compact set $E$. Also, in \eqref{impliedCfev26} we have (see \cite[equations (2-1) and (3-1)]{O'Sul})
$$K_\phi := \max_{z\in D_{z_0}} \left \vert 1-\frac{f(z)-f(z_0)}{f'(z_0)(z-z_0)}\right\vert,$$
where $D_{z_0}$ is a small closed disk neighborhood of $z_0$.  Then $\textstyle \frac{16K_\phi}{\vert f'(z_0)\vert^2}$ is bounded from above by a constant independent of $z_0\in E$.\\
It follows that the error term $\textstyle e^{n (-i) f(z_0)} O_{n \to \infty}\left(K_g/n^2\right)$ in \eqref{1Perronmars26} is a $\textstyle o_{n \to \infty}$ of the leading order term $\textstyle e^{n (-i) f(z_0)} e^{i \theta} \frac{g(z_0)}{f'(z_0)}\frac{1}{n}$, independent of $z_0\in E$, as claimed.\hfill $\Box$
\medskip

The second generalization of Proposition \ref{prop:SPM:PV} that we will need is for contours passing {\it through} a maximum. It is a classical result that the general form of Perron's method, which generalizes Proposition \ref{Perronfev26} (i), implies Proposition \ref{prop:SPM:PV} (see \cite[Corollary 1.4, case of $\mu=2$ and $k_2=k_1+1$]{O'Sul}). In the situation of a sequence $(g_n)$ converging uniformly to $g$ as in Proposition \ref{Perronfev26} (ii), the proof of this classical result generalizes immediately to give the statement (ii) in the next Proposition. In particular, the arguments show that if we are in the situation of a sequence $(g_n)$ converging to $g$ as in the statement (i), the error term $\textstyle o_{n \to \infty}\left(1/\sqrt{n}\right)$ is in fact a $O_{n \to \infty}\left(1/n\right)$, as stated. We point out this particular case, as it can be found verbatim in the litterature.
\begin{prop}\label{gensaddlen} Keep the same assumptions on the path $\lambda$ and the function $f$ as in Proposition \ref{prop:SPM:PV}.\\
(i) {\rm (\cite[Remark 3.3]{Oh}, \cite[page 145, note of $\S$5]{W})} If there is a neighborhood of the critical point $z_0$ of $f$ where for all $n$ large enough $\textstyle g_n(z) = g(z)+ g_{n,1}(z)\frac{1}{n}$, with $g$ holomorphic and independent of $n$, and $\vert g_{n,1} \vert$ bounded from above by a constant independent of $n$, then we have
$$ \int_{\lambda} g_n(z) e^{n (-i) f(z)} dz = 
e^{n (-i) f(z_0)} \left ( \sqrt{2\pi} e^{i \theta} \dfrac{g(z_0)}{\sqrt{|f''(z_0)|}}\dfrac{1}{\sqrt{n}} + O_{n \to \infty}\left(\frac{1}{n}\right) \right )
$$
for some explicit constant $\theta\in \mr$.\\
(ii) Suppose $(g_n)$ is a sequence of holomorphic functions converging uniformly to a function $g$ on a domain containing $\lambda$. Then we have, with $\theta$ as in (i):
$$ \int_{\lambda} g_n(z)  e^{n (-i) f(z)} dz = e^{n (-i) f(z_0)} \left ( \sqrt{2\pi} e^{i \theta} \dfrac{g(z_0)}{\sqrt{|f''(z_0)|}}\dfrac{1}{\sqrt{n}} + o_{n \to \infty}\left(\frac{1}{\sqrt{n}}\right) \right ).$$
\end{prop}
We now describe in Lemma \ref{lem:reformintqjui25} some asymptotics of the integrals $I_{N,\pm}(\ell)$ in \eqref{defINpmellfev26}. The proof will follow the main lines of the analysis of Section \ref{sec:SPM}. We will use Proposition \ref{Perronfev26} because $\rho(z)$ (and hence $g_N(z)$, see \eqref{defgNjanv26}) has poles at $z=0$ and $z=2i\pi$; indeed, some subintegrals of $I_{N,\pm}(\ell)$, which were negligible in the context of Section \ref{sec:SPM} (by the results of Sections \ref{smallpijuil25}-\ref{smallIsN}), will now dominate its asymptotics in case (b) ($\ell = 2i\pi$). Proposition \ref{gensaddlen} (ii) will play the same role as Proposition \ref{prop:SPM:PV} in Section  \ref{sec:SPM}, to estimate the integrals (where now $g_N$ depends on $N$) on subcontours with endpoints away from $0$ and $2i\pi$. Recall that ${\bf l}_0' := \Log(u_{0}) +i\pi a_0$, and the endpoint $s_N^-$ of $C_N^\pm$ is written as $\textstyle \alpha^-\frac{\pi}{N}$.
\begin{lem}\label{lem:reformintqjui25} For $g_N(z)$ as in \eqref{defgNjanv26} we have
\begin{align*} \lim_{N\to +\infty} \frac{2\pi}{N}\left\vert I_{N,-}(i\pi) \right\vert & = \frac{1}{2}{\rm Vol}(M),\\
\vert I_{N,+}(i\pi) \vert & \sim_{N\to +\infty} \dfrac{N^{\frac{{\bf l}_0'}{2i\pi}-\frac{1}{2}+\frac{\alpha^-}{2}}}{\Log(N)} \left(\frac{1}{\alpha^-\pi}\right)^{\frac{{\bf l}_0'}{2i\pi}+\frac{1}{2}+\frac{\alpha^-}{2}} \vert v_N\vert,\\
\vert I_{N,+}(2i\pi) \vert & \sim_{N\to +\infty} \dfrac{N^{\frac{{\bf l}_0'}{2i\pi}-\frac{1}{2}+\frac{\alpha^-}{2}}}{\Log(N)} \left(\frac{1}{\alpha^-\pi}\right)^{\frac{{\bf l}_0'}{2i\pi}+\frac{1}{2}+\frac{\alpha^-}{2}}\vert v_N\vert,\\ 
\vert I_{N,-}(2i\pi) \vert  & \sim_{N\to +\infty} \dfrac{N^{\frac{{\bf l}_0'}{2i\pi}-\frac{1}{2}+\frac{\alpha^-}{2}}}{\Log(N)} \left(\frac{1}{\alpha^-\pi}\right)^{\frac{{\bf l}_0'}{2i\pi}+\frac{1}{2}+\frac{\alpha^-}{2}} e^{-\alpha^-\pi}\vert v_N\vert, 
\end{align*}
where $(v_N)$  (the same in the last three asymptotic formulas) is a bounded sequence, bounded from below by a constant $>0$, and we choose $a_0\geq 4$ (even, as usual) so that $\textstyle \mathfrak{R}(\frac{{\bf l}_0'}{2i\pi}-\frac{1}{2}) >-\frac{1}{2}$. 
\end{lem}
\medskip

The choice of $a_0\geq 4$ will be explained in Step 6 of the proof. When $a_0=2$ the asymptotics of the integrals $I_{N,+}(i\pi)$, $I_{N,+}(2i\pi)$ and $I_{N,-}(2i\pi)$ depend on the sign of $\mathfrak{I}(\Log(u_0))$; for instance, if it is negative, then $I_{N,+}(2i\pi)$ and $I_{N,-}(2i\pi)$ are of the form in \eqref{asyIjui25c}. The choice $a_0=2$ would not affect the conclusions of the subsequent results.

\begin{remark}{\rm One could worry about the factor $e^{-\alpha^-\pi}$ in the asymptotic formula of $I_{N,-}(2i\pi)$ but not of $I_{N,+}(2i\pi)$; indeed, by \eqref{finalestimatefev26} these formulas combine to give the leading order term as $N\to +\infty$ of the LHS integral for $\ell=\bl_{1,\infty}= 2i\pi$, which does not depend on $\alpha^-$. However, one should be careful that a full asymptotic expansion of this integral is needed to describe it faitfully. In the present situation $I_{N,-}(\ell)$ and $I_{N,+}(\ell)$ emerged from the different asymptotical behaviours of the integrand on the half-spaces $\mathfrak{R}(z)<0$ and $\mathfrak{R}(z)>0$, with terms becoming exponentially small in the former region and exponentially large in the latter (see \eqref{facterrmars25}, and also \eqref{equiv3mars26}). Therefore $\mathfrak{R}(z)=0$ is an anti-Stokes line of the asymptotics of the integrand, and exponentially small terms coming from a full asymptotic expansion of $I_{N,-}(\ell)$ and $I_{N,+}(\ell)$ are required to recover the LHS integral of \eqref{finalestimatefev26} as $N\to +\infty$.}
\end{remark}
\proof \underline{Step 1: Conditions on moving the contours.}\\
We adapt the analysis of Sections \ref{sub:SPM:a} (case (a), $\ell=i\pi$) and \ref{sub:SPM:b} (case (b), $\ell=2i\pi$) to the integrals $I_{N,\pm}(\ell)$. Recall that the integrals in Section \ref{sub:SPM:a}-\ref{sub:SPM:b} have $g_N\equiv 1$ and $C_N^-$ or $C_N^+$ replaced by $[is_N^-,2i\pi-is_N^+]$ (the two being homotopic relatively to their endpoints, with the opposite or the same orientation respectively). 

First we show that the function $g_n:z\mapsto g_N(z)$, $\textstyle n:=\frac{N}{2\pi}$, is holomorphic on a domain where we can deform $C_N^-$ or $C_N^+$ to a contour very close to those in Figure \ref{fig:contour:a:+}, Figure \ref{fig:contour:b:+:perron}, and Figure \ref{fig:contourC':b:+}. In each situation the difference between the two contours will occur on short subcontours close to $0$ or $2i\pi$, where the integrals will have polynomial growth at worst.

These deformations occur in the set $\mathcal{S}_\eta$ formed by the strip $-\eta \leq \mathfrak{I}(z)\leq 2\pi$ with cuts $[0,+\infty[$ and $2i\pi + [0,+\infty[$, where $\eta>0$ is defined in the path $C'_3$ below Figure \ref{fig:contour:b:+:perron}. By definition  $\rho(z)$ and the term $\textstyle O_{N\ra \infty}(1)$ in \eqref{defRNjanv26} are holomorphic on $\mathcal{S}_\eta$. The function $\Psi_{u_0,N}$ is holomorphic on the domain \eqref{holdomPsiN} shifted by $\textstyle -\frac{1}{N}\Log(u_{0})$, and $\mathcal{I}_{u_0,N}$ on \eqref{domRN}, whence the two on $\mathcal{S}_\eta$ minus the two half-lines (the poles $\textstyle -\frac{1}{N}\Log(u_{0})+ 2i\pi+2pi\pi/N$ are not in $S_\eta$ since $-\pi < \mathfrak{I}\left(\Log(u_{0})\right)<  \pi$):
$$ L_1:=  -\frac{1}{N}\Log(u_{0})+[0,+\infty[ \ ,\  L_2:= -\frac{1}{N}\Log(u_{0})+2i\pi + [0,+\infty[.$$
In Figure \ref{fig:contour:b:+:shifted:cuts}, these two shifted cuts are represented with teal half-lines, either dashed or dotted. The dotted part around the y-axis covers the case when $\Re\Log(u_0)>0$. For each of $L_1, L_2$, two copies are drawn, according to the sign of $\Im\Log(u_0)=\arg(u_0)$.\\

\underline{Step 2: Deformed contour in case (b), integral $I_{N,-}(2i\pi)$, or case (a), integral $I_{N,-}(i\pi)$.}\\
Figure \ref{fig:contour:b:+:shifted:cuts} actually expands Figure \ref{fig:contourC':b:+}, so let us justify the shifted cuts are where they seem to be in the corresponding case (case (b), integral $I_{N,-}(2i\pi)$). Recall that we assume that there exists some $\upsilon \in (0,1)$ such that 
$-\upsilon \pi\leq \mathfrak{I}(\Log(u_{0}))\leq \upsilon \pi$
(which makes sense even in a more general case when $u_0$ depends on $N$). In this situation, choosing $\textstyle \alpha^-=\frac{1+\upsilon}{2}$ and  $\textstyle s_N^-=\alpha^-\frac{\pi}{N}$ as in the proof of Lemma \ref{qcintjui25}, we conclude that  $\textstyle \alpha^-> \frac{1}{\pi}\vert \mathfrak{I}(\Log(u_{0})\vert$ and thus $\textstyle s_N^-> \frac{1}{N}\vert \mathfrak{I}(\Log(u_{0})\vert$. 
Hence we can deform $C_N^\pm$ to the contour shown in Figures \ref{fig:contourC':b:+} and \ref{fig:contour:b:+:shifted:cuts} without crossing the cuts $L_1$, $L_2$.

By the same arguments $g_N(z)$ is also holomorphic on a domain where we can deform $C_N^\pm$ to $[is_N^-,is_N^+]$, which is the contour occuring in case (a) for the integral $I_{N,-}(i\pi)$.

\begin{figure}[!h]
	\begin{tikzpicture}
		\begin{scope}[xshift=4cm,yshift=2cm]
			\draw[black,fill=gray!50,opacity=0.5, ] (0,0) circle (2ex); 		
			\draw[black,fill=gray!50,opacity=0.5, ] (-.6,0) circle (2.5ex);  
			\draw[black,fill=gray!50,opacity=0.5, ] (-1.2,0) circle (3ex);  
			\draw[black,fill=gray!50,opacity=0.5, ] (-1.8,0) circle (3.2ex);  
			\draw[black,fill=gray!50,opacity=0.5, ] (-2.4,0) circle (3ex);  
			\draw[black,fill=gray!50,opacity=0.5, ] (-2.9,0) circle (2.7ex);  
			\draw[black,fill=gray!50,opacity=0.5, ] (-3.4,0) circle (1.9ex);  
			\draw[black,fill=gray!50,opacity=0.5, ] (-3.65,0) circle (1.4ex);  
			\draw[black,fill=gray!50,opacity=0.5, ] (-4,0) circle (2ex);
		\end{scope}

		\draw[red] (1,3) node {$\Im f_- <0$};		
		\draw[red] (3.2,1.2) node {$\Im f_- <0$};				
		\draw[red] (3,3) node {$\Im f_- >0$};		
		\draw[red] (1,1) node {$\Im f_- >0$};				
		%
		%
		
		\draw[color=black,->] (-2,0)--(8,0);		
		\draw[color=black,->] (0,-2)--(0,5);
		
		\draw (2.25+.75,2+.3) node {$b+i\pi$};
		\draw[red,fill=red] (2.25,2) circle (.5ex); 
		\draw (-.25-.2,-.1-.25) node {$0$};
		\draw[red,fill=red] (0,2) circle (.5ex); 
		\draw (-.25-.2,0.3+.2) node {\color{blue}$i s_N^-$};
		\draw[blue,fill=blue] (0,.3) circle (.5ex); 		
		\draw (-.25-.2,2-.25) node {$i \pi$};
		\draw[red,fill=red] (0,4) circle (.5ex); 
		\draw (-0.9,4-.4) node {\small \color{blue}$2i\pi - i s_N^+$};
		\draw[blue,fill=blue] (0,4-.3) circle (.5ex); 		
		\draw (-.25-.2,4+.25+.1) node {$2i \pi$};
		\draw[color=red,very thick] (0,2)--(8,2);
		\draw[color=blue,thick] (0,0)--(0,4);
		\draw[color=red, very thick, dashed] (0,0) ..controls +(3,0) and +(3,0).. (0,4);
		
		
		\draw[red,fill=red] (0,0) circle (.5ex); 
		
		\draw[color=blue,very thick,dashed] (0,0.3)--(4,.3);
		
		\draw (4+.9,0.3+0.2) node {\color{blue}$1+i s_N^-$};
		\draw[blue,fill=blue] (4,.3) circle (.5ex); 	
		
		\draw[color=blue,very thick,dashed] (4,.3)--(4,2-.15);		
		\draw[blue,fill=blue] (4,2-.15) circle (.3ex); 			
		\draw (4+1.1,2-.3) node {\color{blue}\small $1+i (\pi-\epsilon)$};		
		\draw[color=blue,very thick,dashed] (4,2-.15)--(2.4,2-.15)--(2.15,2.15)--(0,2.15);		
		\draw[blue,fill=blue] (0,2+.15) circle (.3ex); 					
		\draw (-.95,2+.3) node {\color{blue}\small $i (\pi+\epsilon)$};

		\draw[color=violet , very thick ] (0,0)--(8,0);		
		\draw[color=violet ] (2+6.5,0) node {Cut};
		
		\draw[color=violet , very thick ] (0,0+4)--(8,0+4);		
		\draw[color=violet ] (2+6.5,4) node {Cut};		
		
		\draw[color=teal ,dotted, very thick ] (0-.5,0+.15)--(0+.5,0+.15);			
		\draw[color=teal ,dashed, very thick ] (0+.5,0+.15)--(8,0+.15);		
		\draw[color=teal ] (2+6,+.3+.15) node {\tiny Shifted cut $L_1$ (if $\arg(u_0)<0$)};

		\draw[color=teal ,dotted, very thick ] (0-.5,4+0+.15)--(0+.5,4+0+.15);					
		\draw[color=teal,dashed , very thick ] (0+.5,0+4+.15)--(8,0+4+.15);		
		\draw[color=teal ] (2+6,.3+4+.15) node {\tiny  Shifted cut $L_2$ (if $\arg(u_0)<0$)};
		
		\draw[color=teal ,dotted, very thick ] (0-.5,0-.15)--(0+.5,0-.15);					
		\draw[color=teal ,dashed, very thick ] (0+.5,0-.15)--(8,0-.15);		
		\draw[color=teal ] (2+6,-.3-.15) node {\tiny Shifted cut $L_1$ (if $\arg(u_0)>0$)};
		
		\draw[color=teal ,dotted, very thick ] (0-.5,0+4-.15)--(0+.5,0+4-.15);					
		\draw[color=teal,dashed , very thick ] (0+.5,0+4-.15)--(8,0+4-.15);		
		\draw[color=teal ] (2+6,-.3+4-.15) node {\tiny  Shifted cut $L_2$ (if $\arg(u_0)>0$)};
		
	\end{tikzpicture}
	\caption{The deformed contour of Figure \ref{fig:contourC':b:+} with the shifted cuts}\label{fig:contour:b:+:shifted:cuts}	
\end{figure}
\smallskip

\underline{Step 3: Deformed contour in case (b), integral $I_{N,+}(2i\pi)$.}\\
Consider the situation of Figure \ref{fig:contour:b:+:perron}, expanded upon in Figure \ref{fig:contour:b:+:perron:shifted:cuts} with the shifted cuts $L_1, L_2$. Since $\eta,\eta'>0$ are fixed constants, and $\eta'\geq  s_N^+$ with $\textstyle s_N^+=\alpha^+\frac{\pi}{N}$, $\alpha^+\in ]1,2[$, we have
$\textstyle \alpha^+> \vert \mathfrak{I}(\Log(u_{0}))\vert$ and then $\textstyle \eta'>\vert \frac{1}{N}\mathfrak{I}(\Log(u_{0}))\vert$; we can also take $N$ large enough so that $\textstyle \eta>\vert \frac{1}{N}\mathfrak{I}(\Log(u_{0}))\vert$. Then we can deform $C_N^\pm$ to the contour shown in Figure \ref{fig:contour:b:+:perron} without crossing $L_1$ and $L_2$, except in the case $\mathfrak{R}(\Log(u_{0}))>0$, where  the contour must circumvent $L_1$ in a small subarc $C_a$ close to $0$, joining $-i\eta$ to $i\textstyle s_N^-$ (drawn in dashed black in Figure \ref{fig:contour:b:+:perron:shifted:cuts}). 

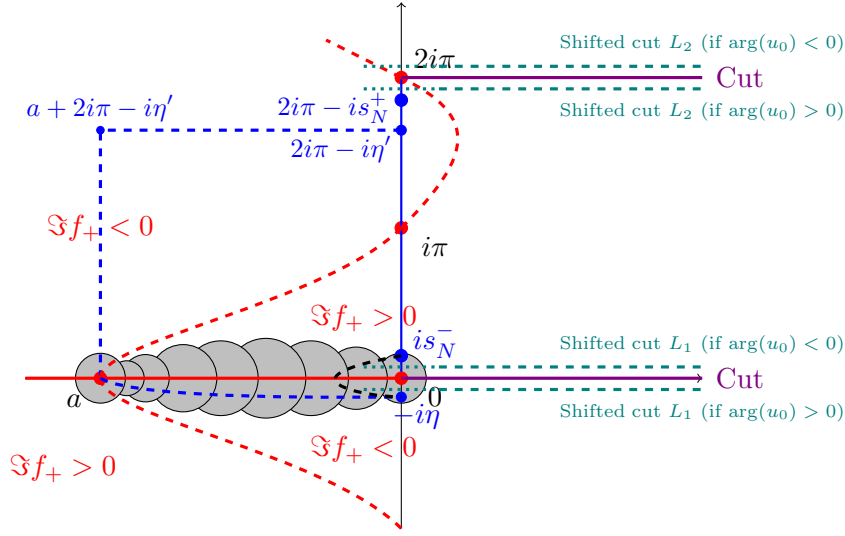
\begin{figure}[!h]
	\begin{tikzpicture}
		
		\draw[black,fill=gray!50,opacity=0.5, ] (0,0) circle (2ex); 		
		\draw[black,fill=gray!50,opacity=0.5, ] (-.6,0) circle (2.5ex);  
		\draw[black,fill=gray!50,opacity=0.5, ] (-1.2,0) circle (3ex);  
		\draw[black,fill=gray!50,opacity=0.5, ] (-1.8,0) circle (3.2ex);  
		\draw[black,fill=gray!50,opacity=0.5, ] (-2.4,0) circle (3ex);  
		\draw[black,fill=gray!50,opacity=0.5, ] (-2.9,0) circle (2.7ex);  
		\draw[black,fill=gray!50,opacity=0.5, ] (-3.4,0) circle (1.9ex);  
		\draw[black,fill=gray!50,opacity=0.5, ] (-3.65,0) circle (1.4ex);  
		\draw[black,fill=gray!50,opacity=0.5, ] (-4,0) circle (2ex);

		\draw[red] (-.5,.8) node {$\Im f_+ >0$};
		\draw[red] (-.5,-.95) node {$\Im f_+ <0$};		
		\draw[red] (-4.5,-1.2) node {$\Im f_+ >0$};		
		\draw[red] (-4,2) node {$\Im f_+ <0$};		
		
		\draw[color=black,->] (-5,0)--(4,0);		
		\draw[color=black,->] (0,-2)--(0,5);
		%
		%
		
		\draw (-4-.35,-.3) node {$a$};
		\draw[red,fill=red] (-4,0) circle (.5ex); 
		\draw (-4,+.25+4-.65) node {\small \color{blue} $a+2i\pi-i\eta'$};
		\draw[blue,fill=blue] (-4,4-.7) circle (.3ex); 		
		\draw (+.25+.2,-.25) node {$0$};
		\draw[red,fill=red] (0,2) circle (.5ex); 
		\draw (+.25+.2,0.5) node {\color{blue}$i \textstyle s_N^-$};
		\draw[blue,fill=blue] (0,.3) circle (.5ex); 		
		\draw (+.25+.2,2-.25) node {$i \pi$};
		\draw[red,fill=red] (0,4) circle (.5ex); 
		\draw (-.9,4-.4) node {\small \color{blue}$2i\pi - i \textstyle s_N^+$};
		\draw[blue,fill=blue] (0,4-.3) circle (.5ex); 		
		\draw (-.8,4-.95) node {\small \color{blue}$2i\pi - i \eta'$};
		\draw[blue,fill=blue] (0,4-.7) circle (.4ex); 		
		\draw (+.25+.2,4+.25) node {$2i \pi$};
		\draw[color=red,very thick] (0,0)--(-5,0);
		\draw[color=blue,very thick,dashed] (0,0+4-.7)--(-4,0+4-.7);
		\draw[color=blue,very thick,dashed ] (-4,0)--(-4,4-.7);
		\draw[color=blue,thick] (0,0)--(0,4);
		\draw[color=red, very thick, dashed] (-4,0) ..controls +(0,-.5) and +(-1,1).. (0,-2);
		\draw[color=red,very thick, dashed] (-4,0) ..controls +(0,.5) and +(-1,-1).. (0,2);
		\draw[color=red,very thick, dashed] (0,2) ..controls +(1,1) and +(1,-.5).. (0,4);
		\draw[color=red, very thick,dashed] (0,4) ..controls +(-1,.5) and +(1,-.5).. (-1,4.5);
		
		\draw[color=blue , very thick, dashed] (-4,0) ..controls +(0.5,-.3) and +(-.5,0).. (0,-.2-.05);
		\draw[blue,fill=blue] (0,-.2-.05) circle (.4ex); 
		\draw (+.2,-0.5) node {\color{blue}$-i \eta $};		
		
		\draw[red,fill=red] (0,0) circle (.5ex);

		\draw[color=violet , very thick ] (0,0)--(8-4,0);		
		\draw[color=violet ] (2+6.5-4,0) node {Cut};
		
		\draw[color=violet , very thick ] (0,0+4)--(8-4,0+4);		
		\draw[color=violet ] (2+6.5-4,4) node {Cut};		
		
		\draw[color=teal ,dotted, very thick ] (0-.5,0+.15)--(0+.5,0+.15);			
		\draw[color=teal ,dashed, very thick ] (0+.5,0+.15)--(8-4,0+.15);		
		\draw[color=teal ] (2+6-4,+.3+.15) node {\tiny Shifted cut $L_1$ (if $\arg(u_0)<0$)};
		
		\draw[color=teal ,dotted, very thick ] (0-.5,4+0+.15)--(0+.5,4+0+.15);					
		\draw[color=teal,dashed , very thick ] (0+.5,0+4+.15)--(8-4,0+4+.15);		
		\draw[color=teal ] (2+6-4,.3+4+.15) node {\tiny  Shifted cut $L_2$ (if $\arg(u_0)<0$)};
		
		\draw[color=teal ,dotted, very thick ] (0-.5,0-.15)--(0+.5,0-.15);					
		\draw[color=teal ,dashed, very thick ] (0+.5,0-.15)--(8-4,0-.15);		
		\draw[color=teal ] (2+6-4,-.3-.15) node {\tiny Shifted cut $L_1$ (if $\arg(u_0)>0$)};
		
		\draw[color=teal ,dotted, very thick ] (0-.5,0+4-.15)--(0+.5,0+4-.15);					
		\draw[color=teal,dashed , very thick ] (0+.5,0+4-.15)--(8-4,0+4-.15);		
		\draw[color=teal ] (2+6-4,-.3+4-.15) node {\tiny  Shifted cut $L_2$ (if $\arg(u_0)>0$)};
		
		\draw[color=black , very thick, dashed] (0,0.3) ..controls +(-1.7,-.3) and +(-.5,0).. (0,-.2-.05);

	\end{tikzpicture}
	\caption{The deformed contour of Figure \ref{fig:contour:b:+:perron} with the shifted cuts.}\label{fig:contour:b:+:perron:shifted:cuts}
\end{figure}

To be more precise, recall that $a$ is the critical point of $f_+$, and let $r,\upsilon''\in ]0,a[$ be such that $\eta,\upsilon''>r\textstyle >\frac{1}{N}|\Re(\Log(u_0))|$. We define $C_a$ as the concatenation of
\begin{itemize}
\item $C^-_a$, going from $-i\eta$ to $-\upsilon''$, while staying in the domain where $\mathfrak{I}f_+<0$ and $\mathfrak{R}(z)<0$, at a distance $\geq \textstyle r$ from $0$, and avoiding $L_1$ if $\arg(u_0)>0$;
\item $C^{+}_a$, going from $-\upsilon''$ to $is_N^-$, while staying at a distance $\geq \textstyle s_N^-$ from $0$, and avoiding $L_1$ if $\arg(u_0)<0$.
\end{itemize}

See Figure \ref{fig:contour:Ca} for an illustration of this. One should understand Figure \ref{fig:contour:Ca} as a zoomed part of Figure \ref{fig:contour:b:+:perron:shifted:cuts} around the origin.

\begin{figure}[!h]
	\begin{tikzpicture}
		
		\draw[color=black,->] (-6,0)--(2,0);		
		\draw[color=black,->] (0,-3)--(0,2+.3);


	\draw[color=black , very thick, dashed] (-4,0) ..controls +(0.5,-.3) and +(-.5,0).. (0,-2-0.5);
				\draw[black,fill=black] (-4,0) circle (.4ex); 
\draw (-4,0+.4) node {\color{black}$-\upsilon''$};
			\draw[color=black , very thick, dashed] (-4,0) ..controls +(0.5,.3) and +(-.5,0).. (0,.9);
				\draw[blue,fill=blue] (0,.9) circle (.4ex); 
\draw (+.3+.2,1) node {\color{blue}$i s_N^- $};		
				\draw[blue,fill=blue] (0,-2-.5) circle (.4ex); 
				\draw (+.3+.2,-2-0.5) node {\color{blue}$-i \eta $};

		\draw[color=violet , very thick ] (0,0)--(4,0);		
		\draw[color=violet ] (2+6.5-4,0) node {Cut};

		\draw[color=teal ,dotted, very thick ] (0-.5,0+.15+.5)--(0+.5,0+.15+.5);			
		\draw[color=teal ,dashed, very thick ] (0+.5,0+.15+.5)--(4,0+.15+.5);		
		\draw[color=teal ] (2+6-4,+.3+.15+.5) node {\tiny Shifted cut $L_1$ (if $\arg(u_0)<0$)};

		\draw[color=teal ,dotted, very thick ] (0-.5,0-.15-.5)--(0+.5,0-.15-.5);					
		\draw[color=teal ,dashed, very thick ] (0+.5,0-.15-.5)--(8-4,0-.15-.5);		
		\draw[color=teal ] (2+6-4,-.3-.15-.5) node {\tiny Shifted cut $L_1$ (if $\arg(u_0)>0$)};

		\draw[color=black] (-2+.4,-1.1) node {$C_{a}^-$};				
		\draw[color=black] (-2+.4,1.3) node {$C_{a}^{+}$};						
	\end{tikzpicture}
	\caption{The small contour $C_a$, which avoids the shifted cut $L_1$}\label{fig:contour:Ca}	
\end{figure}
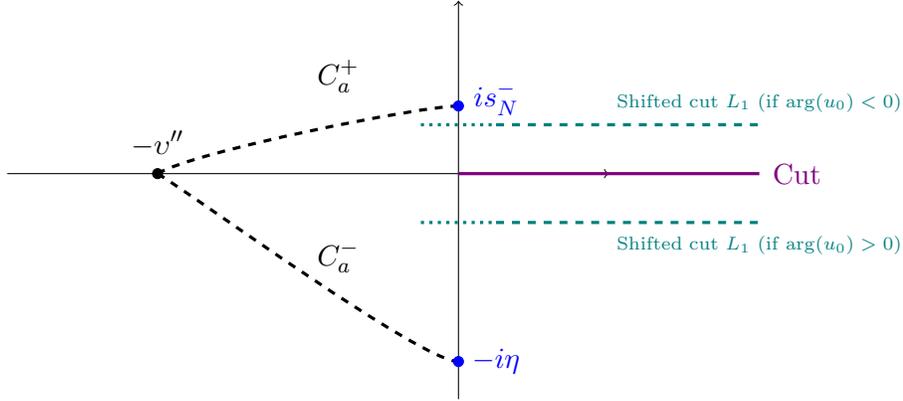
\smallskip

Let us compute asymptotics of the integral $I_{N,+}(2i\pi)$ on the subcontours $C_a^-$, $C_a^+$.  
\smallskip

\underline{On the subarc $C_a^-$:} we can use the fact that $\Im(f_+(z))< \Im(f_+(\upsilon''))=0$ for all $z$ in the interior of $C_a^-$. Let us decompose $C_a^-$ into consecutive subarcs $C_a^{--}$, $C_a^{-+}$ with initial endpoint $-i\eta$ and final endpoint $-\upsilon''$ respectively, such that along $C_a^{--}$ we have $\Im(f_+)<-c$ for some $c>0$. Then:
\begin{align}
	\left \vert \int_{C_a^{--}} g_N(z) e^{\frac{N}{2i\pi}\mathcal{L}_+(z,-2i\pi,2i\pi)}dz\right \vert 
	&\leq \max_{C_a^{--}}(|g_N(z)|)
	\int_{C_a^{--}} e^{\frac{N}{2\pi}\Im(f_+(z))}\vert dz\vert \notag\\
	&\leq  \max_{C_a^{--}}(|g_N(z)|)\int_{C_a^{--}} e^{-\frac{N}{2\pi}c}\vert dz\vert = o_{N\to\infty}\left (\frac{1}{{N}}\right ).\label{intCa-Ofev26}
\end{align}
Here we use that $\textstyle |g_N(z)|$ is bounded from above by a constant independent of $N$ on $C_a^-$. Indeed $C_{a}^-$ is compact and stays away from a fixed small disc centered at $0$, so \eqref{defrho0} implies that $\vert \rho(z)\vert$ is bounded from above by a constant independent of $N$ on $C_{a}^-$. The formula of ${\mathcal{I}_{u_0,N}(z)}$ above \eqref{majorjanv26_2} and Lemma \ref{extboundPsiNjanv26} imply this as well for $\textstyle \vert \exp\left(\frac{1}{N}{\mathcal{I}_{u_0,N}(z)}\right)\vert$ and $\textstyle \exp\left(\frac{1}{N}{\Psi_{u_0,N}(z)}\right)$, respectively, and clearly it holds for the term $\textstyle \left(1+\frac{1}{N}O_{N\ra \infty}\left(1\right)\right)\coth\left(\frac{Nz}{2}\right)$ in \eqref{defRNjanv26}. Hence $\textstyle g_N(z) = \rho(z)\exp\left(\frac{1}{N}{\Psi_{u_0,N}(z)}\right)\textstyle (1 + \frac{1}{N}R_N(z))$ has modulus bounded on $C_{a}^-$ by a constant independent of $N$.\\
Along the subarc $C_a^{-+}$ we can apply Proposition \ref{Perronfev26} (ii) with $n=N/2\pi$, $f=f_+$, and $(g_n)$ the sequence $(g_N)$; indeed, $C_a^{-+}$ is a compact subset of $U$, so Lemma \ref{extboundPsiNjanv26} (iii) and the bounds discussed above for $R_N(z)$ imply that $(g_N)$ converges uniformly to $\rho(z)$ on $C_a^{-+}$. We have $\Im(f_+(\upsilon''))=0$,  $f_+'(\upsilon'')\ne 0$, and $\rho(\upsilon'') \ne 0$, so from \eqref{2Perronmars26} we deduce 
$$\left \vert \int_{C_a^{-+}} g_N(z) e^{\frac{N}{2i\pi}\mathcal{L}_+(z,-2i\pi,2i\pi)}dz \right \vert \sim_{N\to +\infty} D\cdot \frac{1}{N}$$
for a scalar $D\ne 0$ (we will not use its actual value). 

\smallskip

\underline{Consider now the subarc $C_a^{+}$.} Unlike in the case of $C_a^-$ above, the function $g_N$ is not bounded on $C_a^+$ since its factor $\rho(z)$ satisfies $\vert \rho(z) \vert \to +\infty$ as $z:=is_N^-\to 0$. Let us see precisely what happens.\\
Part of the above discussion about $C_a^-$ is still true. Namely, by Lemma \ref{extboundPsiNjanv26} (ii) the function $\textstyle \vert \exp\left(\frac{1}{N}{\Psi_{u_0,N}}(z)\right) \vert$ is bounded from above by a constant independent of $N$ and $z\in C_{a^+}$. Also, by the comment after \eqref{facterrmars25} the same is true for obvious reasons for the modulus of the factor $\textstyle \left(1+\frac{1}{N}O_{N\ra \infty}\left(1\right)\right)\coth(Nz/2)$ in \eqref{defgNjanv26}, since $\coth(Nis_N^-/2) = \coth(i\alpha^-\pi/2)$.\\
Consider $\rho(z)$. We claim that, because $z=is_N^-$ is the point of $C_a^+$ which is the closest to the singularity $z=0$ of $\rho(z)$, it is enough to consider the behaviour of $\rho(is_N^-)$. Indeed, we can deform $C_a^+$ slightly upwards near $is_N^-$, so that it reaches a point $i\eta$ with $s_N^-< \eta< \pi$, and then goes vertically to $is_N^-$. By the same arguments as for $C_{a}^-$, $\vert \rho(z) \vert$ is bounded from above by a constant independent of $N$ on the subarc from $-\upsilon''$ to $i\eta$. For $z=it$ on the vertical subarc from $i\eta$ to $is_N^-$, it is readily checked from the formula \eqref{defrho0} that
$$\vert \rho(it) \vert = {\rm (Constant})\cdot \exp\left(-t\left(\frac{\Log\vert u_1\vert}{2\pi}+\frac{\Log\vert u_0\vert}{4\pi}\right)-\frac{\mathfrak{I}({\bf l}_0')+\pi}{2\pi}\cdot\ln(\sin(t/2))\right).$$
Our choice of $a_0\geq 2$ makes $\mathfrak{I}({\bf l}_0'+\pi)/2\pi = \mathfrak{R}(({\bf l}_0'+i\pi)/2i\pi) >0$, and the derivative of $\vert \rho(it) \vert$  is 
$$\left(-\left(\frac{\Log\vert u_1\vert}{2\pi}+\frac{\Log\vert u_0\vert}{4\pi}\right)-\frac{\mathfrak{I}({\bf l}_0')+\pi}{2\pi} {\rm cotan}\left(\frac{t}{2}\right)\right)\cdot\vert \rho(it) \vert,$$
which is negative for $t>0$ and small, as $\textstyle \lim_{t\to 0^+} {\rm cotan}\left(\frac{t}{2}\right) = +\infty$. The maximum of $\vert \rho(it) \vert$ is therefore attained at $\vert \rho(is_N^-)\vert$ on the subarc from $-\upsilon''$ to $i\eta$. This concludes the proof of our claim.\\
Now we have $1-e^{z} = \textstyle \frac{1}{N}O(1)\ne 0$ for every $\textstyle z\sim_{N\to +\infty} \frac{u}{N}$, $u\in \mc^*$, so the formula \eqref{defrho0} implies that $N^{-({\bf l}_0'+i\pi)/2i\pi}\vert \rho(is_N^-)\vert$ has modulus bounded from above by a constant independent of $N$ (note again that our choice of $a_0\geq 2$ makes $\mathfrak{R}(({\bf l}_0'+i\pi)/2i\pi) >0$, whence $N^{({\bf l}_0'+i\pi)/2i\pi}\to +\infty$ as $N\to +\infty$). A similar argument applied to the integral formula of ${\mathcal{I}_{u_0,N}(z)}$ shows that the same conclusion holds true for $\textstyle \exp\left(\frac{1}{N}{\mathcal{I}_{u_0,N}(is_N^-)}\right)$.\\
Summing up, we obtain that $N^{-({\bf l}_0'+i\pi)/2i\pi}g_N$ is bounded from above by a constant independent of $N$ on $C_{a^+}$. Then let us write
\begin{equation}\label{intWatsonjanv26}\int_{C_a^+} g_N(z) e^{\frac{N}{2i\pi}\mathcal{L}_+(z,-2i\pi,2i\pi)}dz =  N^{({\bf l}_0'+i\pi)/2i\pi} \int_{C_a^+} \frac{g_N(z)}{N^{({\bf l}_0'+i\pi)/2i\pi}} e^{\frac{N}{2\pi}\Im(f_+(z))}dz.
\end{equation}
Now, from Section \ref{sub:b:+:perron} we have
$$\dfrac{\partial}{\partial x} \Im (f_+(x+it))  = 2t - \arg\left (1-e^{x+it}\right ).$$
Since $\arg\left (1-e^{x+it}\right )<0$ when $x<0$ and $t>0$ is close to $0$, it follows that, if one initially chooses $\upsilon''$ small enough (which, by the definition of $C_a^+$ and $C_a^-$ is always possible by letting $N$ be large enough), then $\textstyle \frac{\partial}{\partial x} \Im (f_+(z))>0$ for any $z\in C_a^{+}$, and thus for any such $z=x+it$ we have 
\begin{equation}\label{boundf+caseQbjanv26}
e^{\frac{N}{2\pi}\Im(f_+(z))} \leq e^{\frac{N}{2\pi}\Im(f_+(it))}=e^{\frac{N}{2\pi}{\rm Cl}_2(t)}.
\end{equation}
Then, a simple upper bound of the modulus of the RHS integral in \eqref{intWatsonjanv26} is
\begin{align}
\max_{C_a^+}\left(\frac{|g_N(z)|}{N^{({\bf l}_0'+i\pi)/2i\pi}}\right) \int_{C_a^+} e^{\frac{N}{2\pi}\Im(f_+(z))}dz\ & \leq  \max_{C_a^+}\left(\frac{|g_N(z)|}{N^{({\bf l}_0'+i\pi)/2i\pi}}\right) \cdot  O_{N\to\infty}(1) \cdot \int_{0}^{s_N} e^{\frac{N}{2\pi}{\rm Cl}_2(t)}dt \notag \\ &= O_{N\to\infty}\left (\frac{1}{N^{1-\frac{\alpha^-}{2}}}\right )=  o_{N\to\infty}\left (\frac{1}{\sqrt{N}}\right ),\label{domCa+1fev26}
\end{align}
where we used \eqref{boundf+caseQbjanv26} in the inequality, and the asymptotics of the last integral computed in Section \ref{smallIsN} in the equality. Note that the factor $O_{N\to\infty}(1)$ on the RHS of the inequality comes from the change of measure from the (compact) contour $C^{+}_a$ to the vertical contour $[0,i s_N]$.\\
This estimate is not enough for our purposes, because it only implies that the integral \eqref{intWatsonjanv26} has order $\leq \textstyle N^{({\bf l}_0'+i\pi)/2i\pi}o_{N\to\infty}\left (\frac{1}{\sqrt{N}}\right )$, whereas the integral over the subpath $C_3'\cup C_4'$ (going from $a+2i\pi-i\eta'$ to $-i\eta$, see Figure \ref{fig:contour:b:+:perron:shifted:cuts}) will have order exactly $\textstyle \frac{1}{\sqrt{N}}$ as in the classical case of Section \ref{sub:b:+:perron} (see Step 5). There could be cancellations between these two integrals (the ones over the other subpaths in Figure \ref{fig:contour:b:+:perron:shifted:cuts} will have asymptotics of order $\textstyle o_{\to \infty}(\frac{1}{\sqrt{N}})$), and therefore we have to compute more precisely the order of the integral on the RHS of \eqref{intWatsonjanv26}. We do this by using Proposition \ref{Perronfev26}.\\ 
Let us deform the arc $C_a^+$ slightly downwards near its endpoint $is_N^-$, so that it reaches the axis $i\mr$ in a point $is_N'$, where $s_N':= \textstyle \frac{\alpha'}{N}<s_N^-$, and then goes vertically upwards to $is_N^-$. Denote by $C_a^{+-}$ the subarc from $-\upsilon''$ to $is_N'$, and write
\begin{multline}\label{decompintfev26}
\int_{C_a^+} \frac{g_N(z)}{N^{({\bf l}_0'+i\pi)/2i\pi}} e^{\frac{N}{2\pi}\Im(f_+(z))}dz = \int_{C_a^{+-}} \frac{g_N(z)}{N^{({\bf l}_0'+i\pi)/2i\pi}} e^{\frac{N}{2\pi}\Im(f_+(z))}dz \\ +  \int_{[is_N',is_N^-]} \frac{g_N(z)}{N^{({\bf l}_0'+i\pi)/2i\pi}} e^{\frac{N}{2\pi}\Im(f_+(z))}dz.
\end{multline}
Instead of the integral over $[is_N',is_N^-]$ on the RHS, let us consider for a while the same integral but with the contour $[i\eta_1,i\eta_2]$, with $\textstyle \frac{\pi}{3}>\eta_2>\eta_1>0$ fixed. We know that for $z=it$, we have $\Im(f_+(z)) := \Im(\mathcal{L}_+(z,-2i\pi,2i\pi))= \Im({\rm Li}_2(e^{it}) -t^2) = {\rm Cl}_2(t)$, and the function ${\rm Cl}_2(t)$ is strictly increasing and ${\rm Cl}_2'(t)= -\Log\vert 2\sin(t/2)\vert \ne 0$ for $t\in \textstyle ]0,\frac{\pi}{3}[$. In particular $\Im(f_+(z)) <\Im(f_+(i\eta_2))$ for every $z\in [i\eta_1,i\eta_2]$. All the other assumptions of Proposition \ref{Perronfev26}(ii) are clearly satisfied (the uniform convergence of $(g_N)$ to $\rho$ on $[i\eta_1,i\eta_2]$ follows from Lemma \ref{extboundPsiNjanv26} and the properties of $R_N$ discussed in Lemma \ref{qcintjui25}), so this result implies (the sign in front of the RHS is because $[i\eta_1,i\eta_2]$ {\it ends} at the maximum $i\eta_2$ of $\Im(f_+(i\eta_2))$):
\begin{multline}\label{estimateshortstep3fev26}\int_{[i\eta_1,i\eta_2]} \frac{g_N(z)}{N^{({\bf l}_0'+i\pi)/2i\pi}} e^{\frac{N}{2\pi}\Im(f_+(z))}dz \\ \sim_{N\to +\infty} -\dfrac{e^{\frac{N}{2i\pi}( {\rm Li}_2(e^{i\eta_2}) -\eta_2^2)}}{N^{({\bf l}_0'+i\pi)/2i\pi}} \dfrac{\rho(i\eta_2)}{-i\Log\vert 2\sin(\eta_2/2)\vert}\cdot \dfrac{1}{N}.  
\end{multline}
Now we come back to $\eta_2:=s_N^- = \textstyle \alpha^-\frac{\pi}{N}$, $\eta_1 := s_N'=\textstyle \frac{\alpha'}{N}< s_N^-$. By Proposition \ref{Perronfev26} (iii) the equivalent \eqref{estimateshortstep3fev26} is uniform in $i\eta_2$ as long as it stays a maximum of $\Im(f_+)$ along the contour, it varies in a compact set not containing a critical point of $f_+$, and $(g_N)$ converges uniformly to $\rho$ on this compact set. As noted above, the first two conditions are satisfied here. As for the convergence of $(g_N)$, $i\eta_2=is_N^-$ leaves any compact subset of $\textstyle ]0,i\frac{\pi}{3}[$ as $N\to +\infty$, but the bounds \eqref{boundPsiNseqstripjanv26} and \eqref{majorjanv26_2} (with $\delta:=\alpha\textstyle \frac{\pi}{N}$) on any strip $\textstyle (\alpha+1)\frac{\pi}{N}< \mathfrak{I}(z)< 2\pi-(\alpha+1)\frac{\pi}{N}$, $\alpha>0$, imply that $g_N(\textstyle \alpha^-\frac{\pi}{N}) = \rho(\textstyle \alpha^-\frac{\pi}{N})v_N$, $\textstyle v_N:= \exp\left(\frac{1}{N}{\Psi_{u_0,N}(\alpha^-\frac{\pi}{N})}\right)\textstyle (1 + \frac{1}{N}R_N(\alpha^-\frac{\pi}{N}))$, where $(v_N)$ is a bounded sequence, bounded from below by a constant $>0$. Then, an immediate adaptation of the proof of Proposition \ref{Perronfev26} (ii) implies that \eqref{estimateshortstep3fev26} remains valid with $\eta_2:=s_N^-$ and $\eta_1 := s_N'$, by multiplying $\rho(i\eta_2)=\rho(is_N^-)$ with $v_N$.\\
The result is as follows. We have ${\rm Cl}_2'(s_N^-) = -\Log\vert 2\sin(\alpha^-\pi N/2)\vert \sim_{N\to \infty} \Log(N)$, and hence $\textstyle o_{n \to \infty}\left(1/n\vert f'(z_0)\vert \right) = o_{n \to \infty}\left(1/N\Log(N)\right)$ in the asymptotic formula of Proposition \ref{Perronfev26} (ii). Also ${\rm Cl}_2(s_N^-)\sim_{N\to +\infty} -s_N^-\Log(s_N^-)$, $e^{-\frac{N}{2i\pi} (s_N^-)^2}\sim_{N\to +\infty} 1$, and
$$\frac{\rho(is_N^-) }{-i\Log\vert 2\sin(s_N^-/2)\vert} \sim_{N\to +\infty} 
\frac{1}{i}\cdot e^{-\frac{{\bf l}_1'\alpha^-}{2N}}\left(\frac{Ni}{\alpha^-\pi}\right)^{\frac{{\bf l}_0'}{2i\pi}+\frac{1}{2}}\cdot \frac{1}{\Log(\frac{N}{\alpha^-\pi})}.$$
Therefore, taking the bounded sequence $(v_N)$ into account as explained above, from \eqref{estimateshortstep3fev26} we get
\begin{align}
\int_{[is_N',is_N^-]} & \frac{g_N(z)}{N^{({\bf l}_0'+i\pi)/2i\pi}} e^{\frac{N}{2\pi}\Im(f_+(z))}dz \notag \\ & \sim_{N\to +\infty} i^{\frac{{\bf l}_0'}{2i\pi}+\frac{3}{2}}e^{\frac{N}{2\pi}(\alpha^-\frac{\pi}{N}\cdot \Log(\frac{N}{\alpha^-\pi}))} e^{-\frac{{\bf l}_1'\alpha^-}{2N}}\left(\frac{1}{\alpha^-\pi}\right)^{\frac{{\bf l}_0'}{2i\pi}+\frac{1}{2}}v_N\cdot \frac{1}{\Log(\frac{N}{\alpha^-\pi})}\cdot \dfrac{1}{N}
\notag
\\ & \sim_{N\to +\infty} \dfrac{i^{\frac{{\bf l}_0'}{2i\pi}+\frac{3}{2}}}{N^{1-\frac{\alpha^-}{2}}\Log(N)}\left(\frac{1}{\alpha^-\pi}\right)^{\frac{{\bf l}_0'}{2i\pi}+\frac{1}{2}+ \frac{\alpha^-}{2}}v_N.\label{equiv6fev26}
\end{align}
We can now conclude. Consider the integral over $C_a^{+-}$ on the RHS of \eqref{decompintfev26}. As in \eqref{domCa+1fev26} one finds
\begin{align*}\left \vert \int_{C_a^{+-}} g_N(z)e^{\frac{N}{2i\pi}\mathcal{L}_+(z,-2i\pi,2i\pi)}dz\right \vert = O_{N\to\infty}\left (\frac{1}{N^{1-\frac{\alpha'}{2}}}\right ).\end{align*}
Since $\alpha'< \alpha^-$, it is a $o_{N\to +\infty}$ of \eqref{equiv6fev26}. From this and \eqref{intWatsonjanv26} we deduce
\begin{multline}\label{intCa+8fev26}
\int_{C_a^+} g_N(z) e^{\frac{N}{2i\pi}\mathcal{L}_+(z,-2i\pi,2i\pi)}dz \\ \sim_{N\to +\infty} i^{\frac{{\bf l}_0'}{2i\pi}+\frac{3}{2}}\cdot \dfrac{N^{\frac{{\bf l}_0'}{2i\pi}-\frac{1}{2}+\frac{\alpha^-}{2}}}{\Log(N)} \left(\frac{1}{\alpha^-\pi}\right)^{\frac{{\bf l}_0'}{2i\pi}+\frac{1}{2}+ \frac{\alpha^-}{2}}v_N.
\end{multline}}
\smallskip

This concludes Step 3.

\

\underline{Step 4: Deformed contour in case (a), integral $I_{N,+}(i\pi)$.}\\
Similarly, in the situation of Figure \ref{fig:contour:a:+} we can deform $C_N^\pm$ to the contour shown in the figure except in the case $\mathfrak{R}(\Log(u_{0}))>0$, where the contour must circumvent $L_1$ in a small subarc close to $0$ and joining a point $-\upsilon''$, $\upsilon''>0$, to $is_N^-$, if $\mathfrak{I}(\Log(u_{0}))<0$, or the contour must circumvent $L_2$ in a small subarc close to $2i\pi$ and joining $-\upsilon''+2i\pi$ to $2i\pi - is_N^+$, if $\mathfrak{I}(\Log(u_{0}))>0$. Indeed, we can take such subarcs to be $C^+_a$ in the first case, and $\bar{C}^+_a +2i\pi$ in the second case, with $C^+_a$ defined as in Step 3, and $\bar{C}^+_a$ the complex conjugate of $C^+_a$ but with endpoint $-is_N^+$ instead of $-is_N^-$. 

Here we get the following horizontal derivative:
$$\dfrac{\partial}{\partial x} \Im (f_+(x+it))  
= 2t - \arg\left (1-e^{x+it}\right )+\pi,$$
which is positive when $t>0$. Thus the integrals of $e^{\frac{N}{2\pi}\Im(f_+(z))}$ over the subarcs $C^+_a$ or $\bar{C}^+_a +2i\pi$ will be dominated by a constant times the same integrals over the projections of these sub-arcs on the vertical axis, where we know that $\Im(f_+(z)) = {\rm Cl}_2(t)$. By the same arguments as in Step 3 for $C_a^+$ one proves that $\textstyle \max_{\bar{C}^+_a +2i\pi}\left(\frac{|g_N(z)|}{N^{({\bf l}_0'+i\pi)/2i\pi}}\right)<+\infty$. Then, the same reasoning as from \eqref{intWatsonjanv26} to the inequality in \eqref{domCa+1fev26} gives
\begin{align}
	\left \vert \int_{\bar{C}^+_a +2i\pi} g_N(z)\right.
 &\left. e^{\frac{N}{2i\pi}\mathcal{L}_+(z,-2i\pi,i\pi)}dz\right \vert \notag
	\\ &\leq  N^{({\bf l}_0'+i\pi)/2i\pi} \max_{\bar{C}^+_a +2i\pi}\left(\frac{|g_N(z)|}{N^{({\bf l}_0'+i\pi)/2i\pi}}\right) \cdot  O_{N\to\infty}(1) \cdot 
	\int_{2\pi-s_N}^{2\pi} e^{\frac{N}{2\pi}{\rm Cl}_2(t)}dt \notag \\
	&\quad \quad \quad \quad = N^{({\bf l}_0'+i\pi)/2i\pi} O_{N\to\infty}\left (\frac{1}{N}\right ).\label{intbarCa+step48fev26}
\end{align}
where the asymptotics of the last integral were obtained in Section \ref{smallpijuil25}. On the other hand, as in \eqref{intCa+8fev26} we have:
\begin{equation}\label{intCa+step48fev26}
 \int_{C_a^+}g_N(z) e^{\frac{N}{2i\pi}\mathcal{L}_+(z,-2i\pi,i\pi)}dz \sim_{N\to +\infty} i^{\frac{{\bf l}_0'}{2i\pi}+\frac{3}{2}}\cdot \dfrac{N^{\frac{{\bf l}_0'}{2i\pi}-\frac{1}{2}+\frac{\alpha^-}{2}}}{\Log(N)} \left(\frac{1}{\alpha^-\pi}\right)^{\frac{{\bf l}_0'}{2i\pi}+\frac{1}{2}+ \frac{\alpha^-}{2}}v_N.
\end{equation}
For future reference in Step 6, we note that \eqref{intbarCa+step48fev26} is negligible with respect to \eqref{intCa+step48fev26}, since $\alpha^->0$ and $(v_N)$ is bounded from below by a constant $>0$, and \eqref{intCa+step48fev26} tends to $+\infty$ as $N\to +\infty$, since $\textstyle \mathfrak{R}(\frac{{\bf l}_0'}{2i\pi}-\frac{1}{2}) >0$ by our choice $a_0\geq 2$.

\

\underline{Step 5: Checking the assumptions of Proposition \ref{gensaddlen} away from the ends $is_N^-$, $is_N^+$.}\\
By Steps 2 to 4 the function $g_N(z)e^{\frac{N}{2i\pi}\mathcal{L}_\pm(z,-2i\pi,\ell)}$ is holomorphic on a domain where we can deform the contour $C_N^-$ or $C_N^+$ to those described in Step 2 (Figure \ref{fig:contour:b:+:shifted:cuts} or $[is_N^-,is_N^+]$), Step 3 (Figure \ref{fig:contour:b:+:perron:shifted:cuts}), and Step 4.  By Cauchy's theorem the integral $I_{N,\pm}(\ell)$ is unchanged if one replaces its contour by any of these ones.\\ 
By following the analysis of Sections \ref{sub:SPM:a}--\ref{sub:SPM:b} we now check that the hypothesis of Proposition \ref{gensaddlen}(ii) are satisfied by the integral $I_{N,\pm}(\ell)$, restricted to subcontours that stay away from the ends $is_N^-$ and $is_N^+$ in the situations of Figures \ref{fig:contour:b:+:shifted:cuts} and \ref{fig:contour:b:+:perron:shifted:cuts}, and by the integral $I_{N,-}(i\pi)$ on a subcontour of $[is_N^-,is_N^+]$. Namely, these subcontours go from $2i\pi-i\eta'$ to $-i\eta$ in Figure \ref{fig:contour:b:+:perron:shifted:cuts}, and from $i\eta$ to $2i\pi - i\eta'$ in Figure \ref{fig:contour:b:+:shifted:cuts}, where $\eta>s_N^-$, $\eta'>s_N^+$, and one deformes slightly the contour near $s_N^-$ and $2i\pi - is_N^+$ so that it goes first vertically to $i\eta$ and finishes vertically from $2i\pi - i\eta'$ to $2i\pi - is_N^+$, respectively. The subcontour goes from $i\eta$ to $2i\pi-i\eta'$ vertically in the situation of the integral $I_{N,-}(i\pi)$.\\
Recall that the critical point $z_0$ of $\mathcal{L}_{\pm}(z,-2i\pi,\ell)$, where $\ell=i\pi$ or $2i\pi$, is: $\textstyle z_0= i\frac{\pi}{3}$ for $\mathcal{L}_{-}(z;-2i\pi,i\pi)$, $\textstyle z_0= \Log((\sqrt{5}-1)/2)$ for $\mathcal{L}_{+}(z;-2i\pi,2i\pi)$, and $\textstyle z_0= \Log((\sqrt{5}+1)/2+i\pi$ for $\mathcal{L}_{-}(z;-2i\pi,2i\pi)$.

In each case $g_N(z)e^{\frac{N}{2i\pi}\mathcal{L}_\pm(z,-2i\pi,\ell)}$ is holomorphic in a neighborhood of $z_0$, and by the formulas \eqref{defrho0} and \eqref{defgNjanv26} we see immediately that it is non vanishing at $z=z_0$, since $\rho(z_0)\ne 0$. In Lemma \ref{extboundPsiNjanv26} we proved that $\textstyle \exp\left(\frac{1}{N}\Psi_{u_0,N}(z_0)\right)$ converges uniformly to $1$ on compact neighborhoods of $z_0$ in $U$, and from \eqref{defRNjanv26} and the defining formula of $\mathcal{I}_{u_0,N}(z)$ one finds easily an upper bound for $\vert R_N(z)\vert$, independent of $N$, in a sufficiently small compact neighborhood of $\textstyle z_0$. Hence the hypothesis of Proposition \ref{gensaddlen}(ii) are satisfied in each case.  

\

\underline{Step 6: Comparing the asymptotics on subcontours, and conclusion.}\\
Here we combine the Steps $1$ to $5$ to describe asymptotics of the integrals $I_{N,\pm}(\ell)$.\\
Consider the integral $I_{N,+}(i\pi)$. The analysis of Section \ref{sec:case(a)In+} can be repeated verbatim for the integral restricted to the subcontour from $2i\pi-\upsilon''$ to $-\upsilon''$ in Step 4 (Figure \ref{fig:contour:a:+}). This analysis shows again that this integral is a $O_{N\to +\infty}(1)$ summand of $I_{N,+}(i\pi)$. By Step 4 the integrals \eqref{intbarCa+step48fev26} and \eqref{intCa+step48fev26}, along the arcs $\bar{C}^+_a +2i\pi$ and $C_a^+$, have distinct growths, and \eqref{intCa+step48fev26} dominates \eqref{intbarCa+step48fev26} and is not bounded. Then the integral \eqref{intCa+step48fev26} dominates the asymptotics of $I_{N,+}(i\pi)$, which are therefore of the form stated in Lemma \ref{lem:reformintqjui25}.

Consider the integral $I_{N,-}(i\pi)$. By Step 5 and Proposition \ref{gensaddlen}(ii) the analysis of Section \ref{sec:case(a)In+} can be repeated verbatim for the integral restricted to the subcontour $[i\eta,2i\pi-i\eta']$. This analysis shows again that this integral grows exponentially, with growth rate equal to $\textstyle \frac{1}{2}{\rm Vol}(M)$. The integrals on the subcontours $[is_N^-,i\eta]$ and $[2i\pi-i\eta', 2i\pi-is_N^+]$ have polynomial or subpolynomial growth, like in \eqref{intbarCa+step48fev26} and \eqref{intCa+step48fev26}, so the asymptotics of $I_{N,-}(i\pi)$ are of the form stated in Lemma \ref{lem:reformintqjui25}.

A similar reasoning applies to the integrals $I_{N,-}(2i\pi)$ and $I_{N,+}(2i\pi)$ by using the asymptotics computed in Step 3 for these integrals restricted to the small subcontours near $0$ or $2i\pi$, and the asymptotics computed from Step 5 for the integrals restricted to the subcontour from $i\eta$ to $2i\pi - i\eta'$ in Figure \ref{fig:contour:b:+:shifted:cuts}, and to the subcontour from $2i\pi-i\eta'$ to $-i\eta$ in Figure \ref{fig:contour:b:+:perron:shifted:cuts}. These asymptotics from Step 5 have equivalents of order $1/\sqrt{N}$, as in \eqref{asyIjui25c}. The integral $I_{N,+}(2i\pi)$ restricted to the small subcontour $C_a^+$ has the asymptotics \eqref{intCa+8fev26}, and the integral $I_{N,-}(2i\pi)$ restricted to the small subcontour $[is_N^-,i\eta]$ has the asymptotics \eqref{intCa+8fev26} multiplied with $-e^{-\alpha^-\pi}$ (the sign is because the contour starts at $is_N^-$, and the factor $e^{-\alpha^-\pi}$ comes for the change of $\mathcal{L}_+(z,-2i\pi,2i\pi)$ to $\mathcal{L}_-(z,-2i\pi,2i\pi)$). 
Since $a_0\geq 4$, we have $\textstyle \mathfrak{R}(\frac{{\bf l}_0'}{2i\pi}-\frac{1}{2}) >-\frac{1}{2}$, and therefore these small integrals dominate the asymptotics of $I_{N,+}(2i\pi)$ and $I_{N,-}(2i\pi)$, respectively. This concludes the proof.\cvd
\medskip

Recall Lemma \ref{rem:reduction:sigma:integral}: the two cases (a) and (b) correspond to $\bl_{1,\infty}^\bx=i\pi$, $\bl_{0,\infty}^\bx=-2i\pi$ and $\bl_{1,\infty}^\bx=2i\pi$, $\bl_{0,\infty}^\bx=-2i\pi$ respectively, for both $\bx=\bu$ and $\bv^*$.
\begin{cor}\label{cor:reductionintjui25} We have
$$\lim_{N\rightarrow +\infty}\dfrac{2\pi}{N}\Log \left \vert \int_{C_N} e^{\frac{N}{2i\pi}(z^2 - \bl_{1,N}^\bx z)}\hat S_N(\bl_{0,N}^\bx + z)\textstyle \coth(\frac{Nz}{2})dz \right \vert  = \left \{\begin{matrix}
\dfrac{1}{2}{\rm Vol}(M) & \text{in\ case} \ (a),\\
0  & \text{in\ case}\  (b),
\end{matrix}
\right.$$
and 
$$\lim_{N\rightarrow +\infty}\dfrac{2\pi}{N}\Log \left \vert \Sigma_N(\bx_{0,N},\bx_{1,N})\right \vert = \lim_{N\rightarrow +\infty}\dfrac{2\pi}{N}\Log \left \vert \int_{C_N} e^{\frac{N}{2i\pi}(z^2 - \bl_{1,N}^\bx z)}\hat S_N(\bl_{0,N}^\bx + z)\textstyle \coth(\frac{Nz}{2})dz \right \vert$$
for $\bx =\bu$ or $\bv^*$.
\end{cor}
\proof By Lemma \ref{qcintjui25} and the first two asymptotics in Lemma \ref{lem:reformintqjui25}, if $\ell = i\pi$ we have $\textstyle \lim_{N\to +\infty}\textstyle \frac{2\pi}{N}\Log  \vert 2I_{N,-}(\ell) -2u_0  I_{N,+}(\ell)\vert = \textstyle \frac{1}{2}{\rm Vol}(M)$. When $\ell = 2i\pi$, we can always choose $\alpha_-$ so that $\vert u_0\vert e^{-\alpha^-\pi} \ne 1$. Then the two terms in the Log have different growths, and Lemma \ref{lem:reformintqjui25} implies $\textstyle \lim_{N\to +\infty}\textstyle \frac{2\pi}{N}\Log  \vert 2I_{N,-}(2i\pi) -2u_0  I_{N,+}(2i\pi)\vert=0$. The first claim follows.

Consider the second claim. By Lemma \ref{lem:PM:SE} the limit on the right does not change if we multiply the integral by $N/4i\pi \hat S_N(\bl_{0,N}^\bx)$, since it is in $PM(N)$ by Lemma \ref{factorest}, using that ${\bf l}_{k,N}^\bx \in (\pi i) \mz + O(N^{-1})$ by \eqref{recapuvarmars25} and \eqref{l0inftymars25}-\eqref{l1inftymars25}. And it does not change either if we add $1$ to the term inside $\Log|.|$, again because of its asymptotics given by Lemma \ref{lem:reformintqjui25}.\cvd

\begin{remark}\label{fulldev}{\rm The proof of Lemma \ref{lem:reformintqjui25} gives more than the stated limits or equivalents, and expliciting the phase factor $e^{i\theta}$ in Proposition \ref{gensaddlen} and improving the asymptotics of Lemma \ref{factorest}, one can deduce subexponential terms of $\Sigma_N(\bx_{0,N},\bx_{1,N})$ (see Lemma \ref{lemintrep}), and therefore of $\Hh_{N}^{red}(M,\rho,\kappa; \mathfrak{s})$ by using the formula \eqref{form1}. 

Indeed, consider for instance the case $\ell = i\pi$. In this case $f_-(z) = {\rm Li}_2(e^z)+z^2-i\pi z$, and the computations in Section \ref{sub:a:-} and the formula $\textstyle \mathfrak{R}({\rm Li}_2(e^{it})) = \sum_{n\geq 1} \frac{\cos(nt)}{n^2} = \frac{\pi^2}{6} - \frac{\pi t}{2}+\frac{t^2}{4}$ (\cite{Z}) imply that $f_-$ has the critical point $\textstyle z_0=\frac{i\pi}{3}$, and $\textstyle f_-(\frac{i\pi}{3}) = \frac{\pi^2}{4} + i\frac{1}{2}{\rm Vol(M)}$, $f''_-(e^{\frac{i\pi}{3}}) = 1+ e^{\frac{i\pi}{3}}$. Also, from \eqref{defrho0} we get $\textstyle g(e^{\frac{i\pi}{3}}):=\rho(e^{\frac{i\pi}{3}}) = e^{-\frac{{\bf l}_1'}{2i\pi}\cdot \frac{i\pi}{3}}(1-e^{\frac{i\pi}{3}})^{-\frac{{\bf l}_0'}{2i\pi}-\frac{1}{2}}=e^{-\frac{{\bf l}_1'}{6}-\frac{{\bf l}_0'}{3}-\frac{i\pi}{3}}$, where we recall that ${\bf l}_k' = \Log(u_k) + i\pi a_k$, for $x=u$ or $v^*$. Therefore 
\begin{align*} \int_{C_N} e^{\frac{N}{2i\pi}(z^2 - \bl_{1,N}^\bx z)} \hat S_N(\bl_{0,N}^\bx + z) \textstyle & \coth(\frac{Nz}{2}) dz\\ & \sim_{N\to +\infty} 2I_{N,-}(i\pi) \\
& \sim_{N\to +\infty} 2\cdot \frac{4\pi}{\sqrt{N}} e^{i\theta} \cdot e^{-\frac{{\bf l}_1'}{6}-\frac{{\bf l}_0'}{3}-\frac{i\pi}{3}} \cdot \frac{1}{3^\frac{1}{4}} e^{\frac{N}{2\pi}(\frac{1}{2}{\rm Vol}(M)-i\frac{\pi^2}{4})}
\end{align*}
where $\theta := \pi -\textstyle \frac{1}{2}\arg(-\frac{1}{f''_-(e^{\frac{i\pi}{3}})})$ (see e.g. \cite[equation (19)]{PV}).}
\end{remark}
\color{black}
\appendix

\section{The QHI of $S^3\setminus K$ vs. the Kashaev invariant $\langle K\rangle_N$}\label{sec:KvsQHI} The Kashaev invariant of the figure-eight knot $K$ is (\cite{Kvol})
$$\langle K \rangle_N = \sum_{j=0}^{N-1} \frac{1}{\prod_{k=1}^{j} \vert 1-\zeta^{k}\vert^2}.$$
It coincides with the evaluation of $J_N(K)/J_N(U)$ at $q=e^{\frac{2\pi i}{N}}$, where $J_N(K)$ is the $N$-colored Jones polynomial of $K$ and $U$ the unknot (\cite{MM}). On another hand (see Section \ref{sec:statesum})
\begin{align*}
\Hh_{N}^{red}(S^3\setminus K,\rho,\kappa; \mathfrak{s})
 & =\frac{[\bv_0]_Ng_N(\bu_0)}{g_N(\bv_0)}\sum_{\alpha,\beta=0}^{N-1} \zeta^{\beta^2-\alpha^2} \frac{\omega_N(\bu_0,\bu_1^{-1}\vert \beta)}{\omega_N(\bv_0/\zeta,\bv_1^{-1}\vert \alpha)}.  \end{align*}
Recall the notations in Section \ref{GhypN}, and in particular that $\kappa = \kappa^X\circ \widetilde{hol}$. We mentioned after \eqref{splitHNteo1.1} that $\Hh_{N}^{red}(S^3\setminus K,\rho,\kappa; \mathfrak{s})$ defines a rational function on ${}_NX_{hyp}(M)$. We can take $\kappa^X(\lambda_K)$ as a local coordinate on the smooth locus of ${}_NX_{hyp}(M)$. Then, the next formula shows that $\langle K \rangle_N$ is a value of a function $J_N(\bu_0,\bv_0 \vert 0)$, which is the regular part of $\Hh_{N}^{red}(S^3\setminus K,\rho,\kappa; \mathfrak{s})$ in a neighborhood of $\kappa^X(\lambda_K)=0$ (which maps via $\pi_N^X$ onto a neighborhood of $\delta_\rho(\lambda_K)=0$ in $\tilde X_{hyp}(M)$). 
 
\begin{lem} The rational functions $J_N(\cdot,\cdot \vert i)$, $i\in \{0,\ldots,N-1\}$, defined on $\mc^2$ by
$$J_N(\bu_0,\bv_0 \vert i) := \frac{g_N(\bu_0)\bv_0^{2(N-1)}}{Ng_N(\bv_0)} \sum_{j=0}^{N-1}
    \zeta^{4ij+i}\frac{\prod_{k=1}^{i+N-j-1}\bv_0^{-2}(1-\bv_0\zeta^{k})}{\prod_{k=1}^{i+j}\bu_0^{-2}(1-\bu_0\zeta^{k})},$$
are regular in a neighborhood of $(\bu_0,\bv_0)=(1,1)$, and satisfy $J_N(1,1 \vert 0) = \langle K \rangle_N$ and
\begin{equation}\label{ker8}
\Hh_{N}^{red}(S^3\setminus K,\rho,\kappa; \mathfrak{s}) = \sum_{i=0}^{N-1} \kappa(\lambda_K)^{-i} J_N(\bu_0,\bv_0 \vert i).
\end{equation}
\end{lem}

\proof The first claim is clear, and the second comes as follows:
\begin{align*}
J_N(1,1 \vert 0) & =   \frac{1}{N} \sum_{j=0}^{N-1}
    \frac{\prod_{k=1}^{N-j-1}(1-\zeta^{k})}{
    \prod_{k=1}^{j}(1-\zeta^{k})} = \frac{1}{N} \sum_{j=0}^{N-1}
    \frac{\prod_{k=1}^{N-j-1}(1-\zeta^{k})\prod_{k=1}^{j}(1-\zeta^{-k})}{
    \prod_{k=1}^{j} \vert 1-\zeta^{k}\vert^2} \\ & = \sum_{j=0}^{N-1}
    \frac{1}{\prod_{k=1}^{j} \vert 1-\zeta^{k}\vert^2} = \langle K \rangle_N,
    \end{align*}
using $\textstyle \prod_{k=1}^{N-1}(1-\zeta^{k}) = N$. Finally, in the above formula of $\Hh_{N}^{red}(\bu_0,\bu_1,\bv_0,\bv_1)$ put $\beta = i+j$,  $\alpha = i+N-j$. Then
$$\Hh_{N}^{red}(\bu_0,\bu_1,\bv_0,\bv_1)   = \frac{[\bv_0]_Ng_N(\bu_0)}{g_N(\bv_0)}\sum_{i,j=0}^{N-1}
    \zeta^{4ij} \bu_1^{-(i+j)}\bv_1^{i+N-j}\frac{\prod_{k=1}^{i+N-j}(1-\bv_0\zeta^{k-1})}{
    \prod_{k=1}^{i+j}(1-\bu_0\zeta^{k})}.$$
Let us collect ``global'' factors in the summands. By equations \eqref{eqGN}-\eqref{compbord} the boundary weight $\kappa:= \kappa_\bw\in H^1(\partial \bar M;\mc^*)$ associated to $\bw:=(\bu_0,\bu_1,\bv_0,\bv_1)$ satisfies
\begin{equation*}
\bu_0^{4}\bu_1^{2}  = \zeta^{-1}\kappa(\lambda_K)\ ,\ \bu_0\bu_1\bv_0\bv_1= \zeta^{-1}\kappa(\mu_K)^{-1}.
\end{equation*}
Using the edge relation $\bu_0^{2}\bu_1\bv_0^{2}\bv_1 = 1$ we can write
$$\bu_1^{-(i+j)}\bv_1^{i+N-j} = (\bu_0^4\bu_1^2)^{-i} \bu_0^{2(i+j)}\bv_0^{-2(i-j)}\bv_1^N. $$
Note that $[\bv_0]_N\bv_1^N = 1/N(1-\bv_0)$. Put
$$J_N(\bu_0,\bv_0 \vert i) = \frac{g_N(\bu_0)}{Ng_N(\bv_0)} \sum_{j=0}^{N-1}
    \zeta^{4ij+i}  \bu_0^{2(i+j)}\bv_0^{-2(i-j)} \frac{\prod_{k=1}^{i+N-j-1}(1-\bv_0\zeta^{k})}{
    \prod_{k=1}^{i+j}(1-\bu_0\zeta^{k})}.$$    
Then we obtain the formula \eqref{ker8}. This concludes the proof. \fin

\end{document}